%% file: Main.tex
\documentclass[12pt,english]{smfart}
\usepackage[latin1]{inputenc}
\usepackage{amsmath}
\usepackage{amsfonts}
\usepackage{amssymb}
\usepackage{graphicx}
\usepackage{pstricks}
\usepackage{pst-plot}
\usepackage{epsfig}
\usepackage{hyperref}
\usepackage[left=3.5cm,right=3.5cm,top=5cm,bottom=6cm]{geometry}
\theoremstyle{plain}
\newtheorem{thm}{Theorem}[section]
\newtheorem{cor}[thm]{Corollary}
\newtheorem{lem}[thm]{Lemma}
\newtheorem{sublem}[thm]{Sublemma}
\newtheorem{propo}[thm]{Proposition}
\newtheorem{mainthm}{Theorem}
\theoremstyle{definition}

\theoremstyle{remark}
\newtheorem*{rem}{Remark}
\allowdisplaybreaks
\input{Macros10.tex}

\begin{document}
\title[A proof of Marmi Moussa Yoccoz conjecture for high types]
{A Proof of the Marmi-Moussa-Yoccoz conjecture for rotation numbers of high~type}
\author{Davoud Cheraghi}
\address[Davoud Cheraghi]{Department of Mathematics, Imperial College London, UK}
\email[Davoud Cheraghi]{d.cheraghi@imperial.ac.uk}
\author{Arnaud Ch\'eritat}
\address[Arnaud Ch\'eritat]{Institut de Math\'ematiques de Toulouse; UMR5219
Universit\'e de Toulouse; CNRS
UPS IMT, F-31062 Toulouse Cedex 9, France \\ \and
Institut de Math\'ematiques de Bordeaux; UMR5251
Universit\'e de Bordeaux; CNRS
351, cours de la Lib\'eration - F 33405 TALENCE cedex}
\email[Arnaud Ch\'eritat]{arnaud.cheritat@math.univ-toulouse.fr}
\keywords{}
\subjclass{Primary 37F50,  Secondary 35J60, 30E20, 11A55}
\date{\today}

\begin{abstract}
Marmi Moussa and Yoccoz conjectured that some error function $\Upsilon$, related to the approximation of the 
size of Siegel disk by some arithmetic function of the rotation number $\theta$, is a H\"older continuous function 
of $\theta$ with exponent $1/2$. 
Using the renormalization invariant class of Inou and Shishikura, we prove this conjecture for the restriction 
of $\Upsilon$ to a class of high type numbers.
\end{abstract}
\maketitle
\input{introduction.tex}
\input{section2.tex}
\input{section3.tex}
\input{section4.tex}
\input{section5.tex}
\subsection*{Acknowledgement}
The first author would like to thank the Leverhulme Trust in London for their partial financial 
support while carrying out this research.

\bibliographystyle{smfalpha}
\bibliography{main}
\end{document}

%% file: Macros10.tex
\renewcommand{\u}{\cup} 
\newcommand{\sm}{\setminus}

\newcommand{\ci}{\subseteq} 

\newcommand{\tas}{\text{as }}
\newcommand{\tif}{\text{if }}

\newcommand{\tand}{\text{ and }}

\newcommand{\tfor}{\text{for }}

\newcommand{\twith}{\text{with }}

\newcommand{\ga}{\alpha}
\newcommand{\gb}{\beta}

\newcommand{\gep}{\varepsilon}
\newcommand{\gf}{\varphi}
\newcommand{\gga}{\gamma}
\newcommand{\gh}{\eta}
\newcommand{\gi}{\iota}

\newcommand{\gl}{\lambda}
\newcommand{\gm}{\mu}

\newcommand{\gp}{\pi}

\newcommand{\gs}{\sigma}
\newcommand{\gt}{\tau}

\newcommand{\gx}{\xi}
\newcommand{\gy}{\psi}
\newcommand{\gz}{\zeta}

\newcommand{\gD}{\Delta}
\newcommand{\gF}{\Phi}
\newcommand{\gG}{\Gamma}

\newcommand{\gO}{\Omega}
\newcommand{\gP}{\Pi}

\newcommand{\gU}{\Upsilon}

\newcommand{\gY}{\Psi}
 
\def\B#1{\textbf{#1}}
\newcommand{\C}[1]{{\mathcal{#1}}} 
\newcommand{\D}[1]{{\mathbb{#1}}} 

\newcommand{\R}{\D{R}}
\newcommand{\Q}{\D{Q}}
\newcommand{\N}{\D{N}}
\newcommand{\Z}{\D{Z}}
\newcommand{\CC}{\D{C}}

{\par \samepage}%
{\par}

\newcommand{\ol}{\overline}

\newcommand{\q}{\quad}
\newcommand{\qq}{\qquad}

\newcommand{\wt}[1]{{\widetilde{#1}}}
\newcommand{\pc}{\mathcal{PC}}

\newcommand{\im}{\operatorname{Im}}
\newcommand{\dil}{\operatorname{dil}}
\newcommand{\Dil}{\operatorname{Dil}}
\newcommand{\re}{\operatorname{Re}}
   
\newcommand{\Dom}{\operatorname{Dom}}

\newcommand{\darun}{\operatorname{int}\,}   
\newcommand{\IS}{\mathcal{I}\hspace{-.5pt}\mathcal{S}}
\newcommand{\QIS}{\mathcal{Q}\hspace*{-.5pt}\mathcal{I}\hspace{-.5pt}\mathcal{S}}
\newcommand{\Csh}{\mathcal{C}^\sharp}

\newcommand{\ex}{\operatorname{\mathbb{E}xp}}
 
\newcommand{\ea}{e^{2\pi \ga \mathbf{i}}}
\newcommand{\cv}{\textup{cv}}
\newcommand{\cp}{\textup{cp}}
\newcommand{\co}[1]{^{\circ {#1}}}
\newcommand{\Td}{\operatorname{d_{\text{Teich}}}}
\newcommand{\HT}{\operatorname{HT}}
\newcommand{\SD}{\operatorname{S}}
\renewcommand{\mod}{\operatorname{mod}}

\newcommand{\ds}{\displaystyle}
\newcommand{\on}{\operatorname}
\newcommand{\ov}{\overline}
\newcommand{\cal}{\mathcal}
\newcommand{\Frac}{\on{Frac}}
\newcommand{\Pol}{\cal{P}}

\newcommand{\tend}{\longrightarrow}
\newcommand{\bEA}{\begin{eqnarray*}}
\newcommand{\eEA}{\end{eqnarray*}}
\newcommand{\bEAn}{\begin{eqnarray}}
\newcommand{\eEAn}{\end{eqnarray}}
\newcommand{\setof}[2]{\big\{{#1}\,\big|\,{#2}\big\}}

%% file: introduction.tex
\section{Introduction}\label{S:Intro}

A \emph{Siegel disk} of a complex one dimensional dynamical system $z\mapsto f(z)$ is a maximal open set $\Delta$ on which $f$ is conjugate to a rotation on a disk. There is a unique fixed point inside $\Delta$, and its eigenvalue for $f$ is equal to $e^{2 \pi i\alpha}$ for some $\alpha\in\R$ called the \emph{rotation number}.
For a Siegel disk contained in $\mathbb{C}$ and whose fixed point is denoted $a$, we define its \emph{conformal radius} as the unique $r\in(0,+\infty]$ such that there exists a conformal diffeomorphism $\phi : B(0,r) \to \Delta$ with $\phi(0)=a$ and $\phi'(0)=1$. 
Since the self conformal diffeomorphisms of $\mathbb{C}$ and of the unit disk are well known, it is not hard to see that such a $\phi$ necessarily conjugates $f$ to the rotation $z\mapsto e^{2\pi i\alpha}z$, i.e.\ it is a \emph{linearizing map}.

Given a holomorphic map $f$ with a fixed point $a\in\mathbb{C}$, if $f'(a)=e^{2\pi i\alpha}$ for some $\alpha\in\R$ we may wonder if $f$ has a Siegel disk centered on $a$, i.e.\ if it is linearizable. 
This is a subtle question. We will skip its long and interesting history (see \cite{M}, Section~11) and jump to the matter needed here.

For irrational numbers $\alpha\in\R$, Yoccoz defined the \textit{Brjuno function} 
\[ B(\alpha) = \sum_{n=0}^{+\infty} \beta_{n-1} \log\frac{1}{\alpha_n} \in (0,+\infty]
.\]
Let us explain what these numbers are. 
The sequence $\alpha_n\in (0,1)$ is the sequence associated to the continued fraction algorithm:
$\alpha_0 = \Frac(\alpha)$ and $\alpha_n=\Frac(1/\alpha_n)$. And $\beta_n = \alpha_0 \cdots \alpha_n$, with the convention that $\beta_{-1}=1$. 
The set of \emph{Brjuno numbers} is the set 
\[ \cal B=\setof{\alpha\in\R\setminus\Q}{B(\alpha)<+\infty}
.\]
This is a dense subset of $\mathbb{R}$ of full Lebesgue measure.

Consider now the degree two polynomial $P_\alpha(z)=e^{2\pi i\alpha}z+z^2$. It is considered as one of the simplest non linear examples one can think of. Depending on $\alpha$, it may or may not have a Siegel disk. If it does, let $r(\alpha)$ denote the conformal radius. Otherwise, let $r(\alpha)=0$. 
Fatou and Julia knew that for $\alpha$ rational, $r(\alpha)=0$.
Yoccoz \cite{Yoc95} has completely characterized the set of irrational numbers for 
which $r(\alpha)>0$: it coincides with the set of Brjuno numbers.
Moreover, he provided a good approximation of $r(\alpha)$: he showed that $r(\alpha)> e^{-B(\alpha)-C}$, 
for some constant $C>0$. 
He almost showed that $r(\alpha)< e^{-B(\alpha)+C}$, for some constant $C>0$. 
Two proofs for the latter were given by Buff and Ch\'eritat in \cite{BC04} and \cite{BC11}.

Yoccoz and Marmi considered the error function
\[\Upsilon(\alpha)=\log(r(\alpha))+B(\alpha)
.\]
Computer experiments made by Marmi \cite{Mar89} revealed a continuous graph for $\Upsilon$. It was proved in \cite{BC06b} that $\Upsilon$ is the restriction to $\cal B$ of a continuous function over $\R$.
In \cite{MMY97}, Marmi, Moussa and Yoccoz conjectured that $\Upsilon$ is in fact $1/2$-H\"older continuous. 
Somehow, it cannot be better: it was proved by Buff and Ch\'eritat in \cite{HDR08} that $\Upsilon$ cannot be $a$-H\"older continuous for $a>1/2$. 
Numerical studies carried out in \cite{Carletti03} support the conjecture. 
Here we will prove that the restriction of $\Upsilon$ to some Cantor set of rotation numbers is indeed $1/2$-H\"older continuous.

In 2002, Inou and Shishikura constructed a class of maps $\IS$ that is invariant under 
some renormalization operator, \cite{IS06}. 
This provides a powerful tool that may be used to obtain a deep understanding of the 
dynamics of the maps in $\IS$.
It was used in 2006 to prove the existence of quadratic polynomials with positive area Julia 
sets \cite{BC12}. 
This renormalization scheme requires the rotation number to be of \emph{high type}, i.e. 
that all its continued fraction entries are at least $N$, where $N$ is some constant introduced by 
Inou and Shishikura. 

To be more precise, Inou and Shishikura use \emph{modified continued fractions}, which is almost the same as the standard ones 
(see the discussion in Section~\ref{sec:mcf}).
Let $\alpha = a_0 \pm 1/(a_1 \pm 1/(a_2 \pm \cdots$ be the modified continued fraction, with $a_n\geq 2$, for 
$\forall n\geq 1$. 
For $N\geq 2$ let
\[\HT_N = \setof{\alpha\in\R\setminus\Q}{\forall n>0, a_n\geq N}\]
be the set of high type numbers associated to $N$. 
Then there exists $N_0\geq 2$ such that maps in $\IS$ and with rotation number in 
$\HT_{N_0}$ are infinitely renormalizable.
A value for $N_0$ is not explicitly computed in \cite{IS06}, only its existence is proved. 

A detailed quantitative analysis of the Inou-Shishikura renormalization scheme has been 
carried out by the first author in \cite{Ch10-I} and \cite{Ch10-II}. 
In particular, a new analytic approach is introduced in these papers that allows fine control on 
the geometric quantities that appear in the renormalization. 
These have resulted in a fine control on the post-critical sets of the maps \cite{Ch10-I, Ch10-II}, 
and are used to describe the asymptotic distribution of the orbits \cite{AC12}. 
The flexibility of the scheme and the techniques allow one to study the dynamics of $P_\ga$ for small 
perturbations of $\ga$ away from the real line. 
This has led to a breakthrough on the MLC conjecture (local connectivity of the Mandelbrot set)
by Cheraghi and Shishikura \cite{ChSh14}. 

The object of the present paper is to prove the following theorem:

\begin{mainthm}
There exists a constant $N$ such that the restriction of $\Upsilon$ to $\HT_{N}$ is $1/2$-H\"older continuous.
\end{mainthm}
All the above mentioned consequences apply to maps that are infinitely renormalizable in 
the sense of Inou and Shishikura, and the above theorem will be no exception. 
Indeed, we will prove a stronger statement, see Theorem~\ref{thm:map} in Section~\ref{subsec:map}.

\medskip

The proof has two flavors, an analytic and an arithmetic. 
In the analytic part, we prove a form of Lipschitz dependence of the Inou-Shishikura 
renormalization of $f$ with respect to $f\in\IS$. 
More precisely, it is Lipschitz in the direction of the nonlinearity of $f$, but it is only Lipschitz 
with respect to a modified distance defined by $ds = |d\alpha|\times \big| \log|\alpha| \big|$ in 
the direction of rotation. 
This involves the analytic techniques developed by the first author in \cite{Ch10-I,Ch10-II}.
In the arithmetic part, we deduce H\"older continuity of an extension of $\Upsilon$ to $\IS$, 
using boundedness of $\Upsilon$ and applying the renormalization operator infinitely many times. 
This argument is very much in the spirit of \cite{MMY97}, and involves arithmetic estimates 
on the continued fraction that we reproduce here.

While the techniques used in this paper restrict the set of parameters that we can address, 
it is widely believed that there should be an analogous notion of renormalization scheme with 
similar qualitative features for which there is no assumption to make on the rotation number. 
So the arguments developed here might eventually be applied to the general case. 
On the other hand, a wide range of dynamical behaviors are present in the class of maps 
discussed here since irrational numbers of significant types, such as bounded type, 
Brjuno, non-Brjuno, Herman, non-Herman, and Liouville, are present in $\HT_{N}$. 

The article is divided as follows. Section~\ref{S:Intro} is the present introduction.
In Section~\ref{sec:mcf}, we recall modified continued fractions, the Brjuno functions 
$B=B_1$ and $B=B_{1/2}$, and make comments on different definitions of high 
type numbers.
In Section~\ref{sec:prelim}, we recall the Inou and Shishikura class of maps $\IS$, and the renormalization theorem.
In Section~\ref{S:lipschitz}, we study the dependence of the renormalization on the data.
In Section~\ref{sec:arn}, we use renormalization to prove H\"older continuity of $\Upsilon$.

%% file: section2.tex
\section{Modified continued fractions}\label{sec:mcf}

The terminology of \emph{modified continued fractions} might not be standard but it is a very old notion.
The connection with our work starts with the discussion in \cite{MMY97}. In their notations, what we call $\alpha$ they call $x$ and they use $\alpha$ as a parameter for a family of continued fraction algorithms. The one we present below corresponds to their continued fraction with parameter $\alpha=1/2$.

\subsection{Definition}\label{subsec:def}

\begin{figure}
\begin{center}\includegraphics[width=6cm]{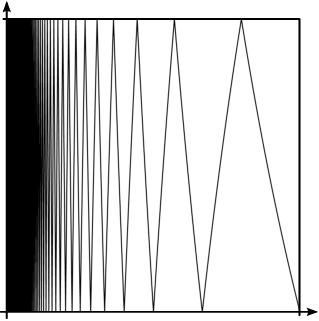}\end{center}
\caption{The graph of $G:(0,1/2]\to[0,1/2]$.}
\label{fig:Ggraph}
\end{figure}

Let $\alpha\in\R$. Let
\[\alpha_0 = d(\alpha,\Z)
,\]
where $d$ denotes the euclidean distance on $\R$. Let
\[ \alpha_{n+1} = d(1/\alpha_n,\Z) = G(\alpha_n)
\]
where $G:(0,1/2]\to[0,1/2]$, $x\mapsto d(1/x,\Z)$, see Figure~\ref{fig:Ggraph}.

The sequence $\alpha_n$ is defined for all $n\in\N$ iff $\alpha\notin\Q$.
Let $a_0\in\Z$ be such that $\alpha\in [a_0-1/2,a_0+1/2)$. Let
\begin{itemize}
\item $s_0$ be undefined if $\alpha-a_n = 0$ or $-1/2$.
\item $s_0 = -1$ if $\alpha-a_0\in(-1/2,0)$,
\item $s_0= 1$ if $\alpha-a_0\in(0,1/2)$.
\end{itemize}
Similarly, let $a_n\in \Z$ be such that $\alpha_{n-1}^{-1} -a_n \in [-1/2,1/2)$, and note that 
$a_n\geq 2$.
Define\footnote{We could have taken $a_n$ to be the floor of $1/\alpha_n$ instead of the 
nearest integer. 
There is a simple way to pass from one convention to the other, since $(a_n,s_n)$ in one 
convention depends solely on $(a_n,s_n)$ in the other. 
The important thing is to have a way to label the intervals on which the maps $H_n$ defined 
below are bijections.}
\begin{itemize}
\item $s_n$ undefined if $\frac{1}{\alpha_{n-1}}-a_n = 0$ or $-1/2$.
\item $s_n = -1$ if $\frac{1}{\alpha_{n-1}}-a_n\in(-1/2,0)$,
\item $s_n= 1$ if $\frac{1}{\alpha_{n-1}}-a_n\in(0,1/2)$.
\end{itemize}
The map $H_n:\alpha\mapsto \alpha_n$ can be decomposed as
\[H_n=\on{saw}\circ\on{inv}\circ \cdots \circ \on{inv}\circ \on{saw},\]
where $\on{inv}$ appears $n$ times and $\on{saw}$ $n+1$ times, $\on{saw}(x)=d(x,\Z)$, whose graph is like a saw, and $\on{inv}(x)=1/x$.
The biggest open intervals on which $H_n$ is a bijection are called \emph{fundamental intervals} (of generation $n$).
They consist in those $\alpha$ with a fixed sequence $(a_0, s_0)$, \ldots, $(a_n, s_n)$, called the \emph{symbol} of the interval.
See Figure~\ref{fig:fundInt}.
The map $H_n$ is a bijection from each fundamental interval to $(0,1/2)$.
For $\alpha$ irrational, we will denote the $n$-th generation fundamental interval containing 
$\alpha$ by
\[I_n(\alpha).\]

\begin{figure}
\begin{center}\includegraphics*[width=12.5cm]{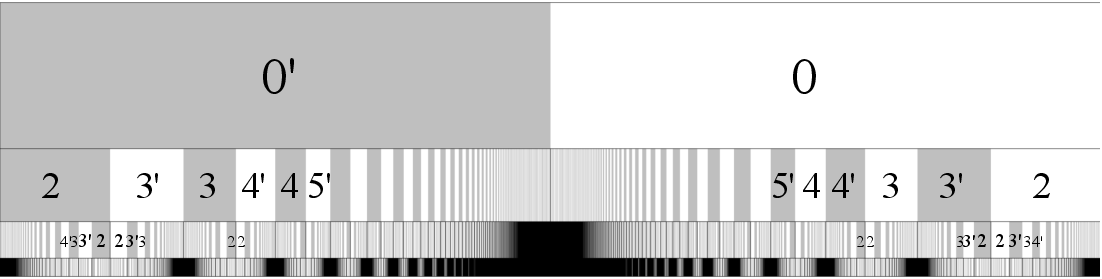}\end{center}
\caption{Symbolic decomposition of $[-1/2,1/2]$ according to the modified continued fraction algorithm. 
An element $x\in[-1/2,1/2]$ is represented by a vertical line of abscissa $x$. 
The rectangle spans the interval $[-1/2,1/2]$. The top row gives $(a_0,s_0)$, depending on $x$. 
The row below it gives $(a_1,s_1)$, the next row $(a_2,s_2)$ and the last row $(a_3,s_3)$. 
The notation $3\text{\textquoteright}$ is a shorthand for $(3,-)$ whereas $3$ means $(3,+)$. 
The decomposition in fundamental intervals is materialized by the alternating white and gray intervals.}
\label{fig:fundInt}
\end{figure}

Let us recall the following classical result:
\begin{lem}\label{lem:bd}
There exists $C>0$ such that for all $n\geq 0$, for all fundamental interval $I$ in the $n$-th generation, 
$H_n$ has distortion $\leq C$ on $I$, that is 
\[\forall x,y\in I, \quad e^{-C}\leq \left| \frac{H_n'(x)}{H_n'(y)}\right | \leq e^C
.\]
\end{lem}
The union of the fundamental intervals of a given generation is the complement of a 
countable closed set.
Let
\[\beta_{-1} = 1\qquad\text{and}\qquad\beta_n = \alpha_0 \cdots \alpha_n
.\]
The map $H_n$ is differentiable on each fundamental interval, and
\[H'_n(\alpha) = \pm \frac{1}{\beta_{n-1}^2}
.\]
A corollary of bounded distortion\footnote{It can also be proved using an inductive 
computation on $H_n$ as a M\"obius map on fundamental intervals} is the following 
estimate on the length of the $n$-th generation fundamental interval containing $\alpha$:
\[e^{-C} \leq \frac{\big|I_n(\alpha)\big|}{\beta_{n-1}^2/2}\leq e^{C}
.\]

\subsection{The modified Brjuno function}\label{SS:modified-sums}

Originally, Brjuno introduced the function $B'(\alpha) =\sum \frac{\log q_{n+1}}{q_n}$, where 
$p_n/q_n$ is the sequence of convergents associated to the continued fraction expansion of $\alpha$. 
He proved the following. 
Consider a holomorphic germ defined in a neighborhood of the origin in $\mathbb{C}$, 
with expansion $f(z)=e^{2\pi i\alpha} z +\dots$. 
If $\alpha$ satisfies $B'(\alpha)<+\infty$ then $f$ is locally linearizable at $0$ \cite{Brj71}. 
This generalized an earlier result by Siegel \cite{Sie42} under the Diophantine condition.
Brjuno's condition turned out to be sharp in the quadratic family as proved by Yoccoz \cite{Yoc95}.
Yoccoz then defined two functions $B(\alpha)$, for $\alpha$ irrational, as
\[ B(\alpha) = \sum_{n\geq 0} \beta_{n-1}\log \frac{1}{\alpha_n} \in (0,+\infty]
\]
where $\alpha_n$ and $\beta_n$ are the sequences associated to $\alpha$ in either the classical continued fraction algorithm, in which case we will denote this sum $B_1(\alpha)$, or the modified continued fraction algorithm, in which case we denote it $B_{1/2}(\alpha)$.
Yoccoz proved that both $B_1$ and $B_{1/2}$ take finite values exactly at the same irrationals, a.k.a.\  
the Brjuno numbers. 
Indeed, he showed that the difference $B_1-B_{1/2}$ is uniformly bounded from above on irrational numbers.
The map $B_1$ satisfies the following functional equations:
\[B_1(\alpha+1) = B_1(\alpha)\quad\text{and}\quad\forall \alpha\in(0,1),\ B_1(\alpha) = \log(1/\alpha) + \alpha B_1(1/\alpha)
.\]
It is being understood that in both equations, the right hand side is finite if and only if the 
left hand side is.
The map $B_{1/2}$ satisfies 
\begin{equation}\label{E:Yoccoz-properties}
\begin{gathered}
B_{1/2}(\alpha+1) = B_{1/2}(\alpha), \; B_{1/2}(-\alpha) = B_{1/2}(\alpha)\\
\forall \alpha\in(0,1/2),\; B_{1/2}(\alpha) = \log(1/\alpha) + \alpha B_{1/2}(1/\alpha).
\end{gathered}
\end{equation}
In \cite{MMY97} (Theorem~4.6), it was proven that $B_1-B_{1/2}$ is H\"older-continuous with exponent $1/2$.
It follows that the main theorem of the present article is independent of the choice of the continued fraction 
expansion, $B=B_1$ or $B=B_{1/2}$. 
In the sequel, we use
\[B=B_{1/2}.\]

Because modified continued fractions are better suited to the Inou-Shishikura renormalization, 
we will replace the function $B$ by a variant defined using the modified continued fraction instead 
of the classical one, see Section~\ref{sec:mcf}. 
The difference of these two $B$ functions being $1/2$-H\"older continuous, the statement of 
the main theorem is equivalent with either one.

\begin{rem}
Marmi Moussa and Yoccoz in \cite{MMY01} study the function $B$ as a cocycle under the action 
of the modular group $PGL(2, \mathbb{Z})$, and in particular, they introduce and analyze the 
behavior of a complex extension of this function. 
See also \cite{BaMa12,LiSi09,LMNN10,Riv10}, and the references therein, for other relevant 
studies of the Brjuno function.
\end{rem}

\subsection{High type numbers}

For $N\geq 1$, let
\[\HT^{c}_N = \setof{\alpha\in\R\setminus\Q}{\forall n>0, a^{c}_n\geq N}
\]
where $a^{c}_n$ is the sequence in the classical continued fraction expansion of $\alpha$.
For $N\geq 2$ let
\[\HT^{m}_N = \setof{\alpha\in\R\setminus\Q}{\forall n>0, a^{m}_n\geq N}
\]
where $a^{m}_n$ is the sequence in the modified continued fraction.
Note that
\[\HT^{c}_1 = \R\setminus\Q = \HT^{m}_2
.\]
There is a simple algorithm to deduce the two sequences $(a^{c}_n)$ and $((a^{m}_n,s_n))$ from each other, see~\cite{Yoc95}.
For $N\geq 2$,
\[\HT^{c}_{N} \subset \HT^{m}_{N}
.\]
Indeed, if all entries $a^{c}_n\geq 2$ for $n\geq 1$, then the modified expansion has symbols $(a^{m}_n,s_n) = (a^{c}_n,+)$.
Note that $HT^m_{N}$ is not contained in $HT^c_N$, nor in $HT^c_{N-1}$, etc ,  
because the presence of some $s_n=-$ will imply the 
presence of a $1$ in the sequence $a^{c}_n$. 
So the statement made in our main theorem is slightly more general in its form, 
which uses $\HT_N = \HT^{m}_N$, rather than $\HT_N = \HT^{c}_N$.

%% file: section3.tex
\section{Renormalization}\label{sec:prelim} 
\subsection{Inou-Shishikura class}\label{SS:Inou-Shishikura}
Consider the ellipse and the map 
\[E:=\Big\{x+\B{i} y\in \D{C}\mid (\frac{x+0.18}{1.24})^2+(\frac{y}{1.04})^2\leq 1 \Big\}, \tand 
g(z):=-\frac{4z}{(1+z)^2}.\]
The polynomial $P(z):=z(1+z)^2$ restricted to the domain 
\begin{equation}\label{V}
V:= g(\hat{\D{C}} \setminus E),
\end{equation}
has a fixed point at $0\in V$ with multiplier $1$, a critical point at $-1/3\in V$ 
that is mapped to $-4/27$. It has another critical point at $-1\in \D{C}\setminus V$ with $P(-1)=0$.

Following \cite{IS06}, we define the class of maps
\begin{displaymath}
\IS_0:=\left\{f:=P\circ \gf_f^{-1}\!\!:V_f \rightarrow \D{C} \;\middle| %
\begin{array}{l} 
\text{$\gf_f \colon V \to V_f=\gf_f(V)$ is univalent,} \\
\text{$\gf_f(0)=0$, $\gf'_f(0)=1$, and } \\
\text{$\gf$ has a quasi-conformal extension to $\D{C}$.} 
\end{array}
\right\},
\end{displaymath}
and for $A\ci \D{R}$,
\[\IS_A:=\{z \mapsto f (\ea z):e^{-2\gp\ga \B{i}}\cdot V_f \to \D{C} 
 \mid f\in \IS_0, \ga\in A\}.\]
Abusing the notation, $\IS_\gb=\IS_{\{\gb\}}$, for $\gb\in \D{R}$.
We have the natural projection
\[\pi:\IS_\D{R}\to \IS_0,  \pi(f)=f_0, \text{  where } f_0(z)=f(e^{-2\pi\ga\B{i}}z)
.\]

The Teichm\"uller distance between any two elements $f=P\circ \gf_f^{-1}$ and $g= P\circ \gf_g^{-1}$ in $\IS_0$ 
is defined as
\[\Td(f,g):=\inf \Big\{\log \Dil (\hat{\gf}_g\circ \hat{\gf}^{-1}_f)\;\Big|%
\begin{array}{l} 
\hat{\gf}_f \tand \hat{\gf}_g \text{ are quasi-conformal extensions}\\
\text{of }\gf_f \tand \gf_g \text{ onto } \CC \text{, respectively}
\end{array}
\Big\}
.\]
This metric is inherited from the one to one correspondence between $\IS_0$ and the Teichm\"uller space of 
$\CC\setminus \ol{V}$.
It is known that this Teichm\"uller space with the above metric is a complete metric space.
The convergence in this metric implies the uniform convergence on compact sets.

Every map in $\IS_\D{R}$ has a neutral fixed point at $0$ and a unique critical point 
at $e^{-2\pi\ga \B{i}}\cdot \gf(-1/3)$ in $e^{-2\pi\ga \B{i}}\cdot V_f$, where $\ga$ is the 
rotation number of $f$ at zero. 
The class $\IS_\R$ naturally embeds into the space of univalent maps on the unit disk with a 
neutral fixed point at $0$. 
Therefore, by Koebe distortion Theorem \ref{T:distortion-estimates}, $\IS_\R$ is a precompact class 
in the compact-open topology\footnote{We will thus work with two topologies on $\IS_0$, the 
one induced by $\Td$, for which it is complete but not pre-compact, and the one from 
compact-open topology, for which it is pre-compact but not complete.}.
Moreover, one can use the area Theorem to show (see \cite[Main Theorem 1-a]{IS06} for details) 
that with the particular choice of $P$ and $V$ 
\begin{equation}\label{E:bound-on-f''}
\{|f''(0)|; f\in \IS_\R\} \ci [2,7].
\end{equation}

For $h\in \IS_\R$ with $h'(0)=e^{2\gp \gb\B{i}}$, define $\ga(h):=\gb$.
Also, $\cp_h$ denotes the unique critical point of $h$. 
According to Inou-Shishikura \cite{IS06}, when $\ga$ is small, any map $h(z)=f(\ea z)\in\IS_\R$ with 
$\ga=\ga(h)\neq 0$ has a distinguished non-zero fixed point $\sigma_h$ near $0$ in $V_h$. 
The $\sigma_h$ fixed point depends continuously on $h$ and has asymptotic expansion 
$\sigma_h=-4\pi \ga \B{i}/f''(0)+o(\ga)$, when $h$ converges to $f\in\IS_0$ in a fixed neighborhood of $0$. 

To deal with maps in $\IS_{\D{R}}$ and the quadratic family at the same time, 
we normalize the quadratic maps to have their critical value at $-4/27$; 
\[Q_\ga(z):=\ea z+\frac{27}{16}e^{4\pi \ga\B{i}}z^2.\]

According to Inou and Shishikura \cite{IS06}, there exists an $\ga^* >0$ such that 
for every $h\colon V_h \to \D{C}$ in $\IS_\R$ or $h=Q_\ga:\D{C}\to\D{C}$, with 
$\ga(h)\in (0,\ga^*]$, there exist a domain $\C{P}_h \subset V_h$ and a univalent map 
$\Phi_h\colon \C{P}_h \to \D{C}$ satisfying 
the following properties:
\begin{figure}[ht]
\begin{center}
\begin{pspicture}(9,5.5)
\rput(4.5,2.75){\includegraphics[width=8cm]{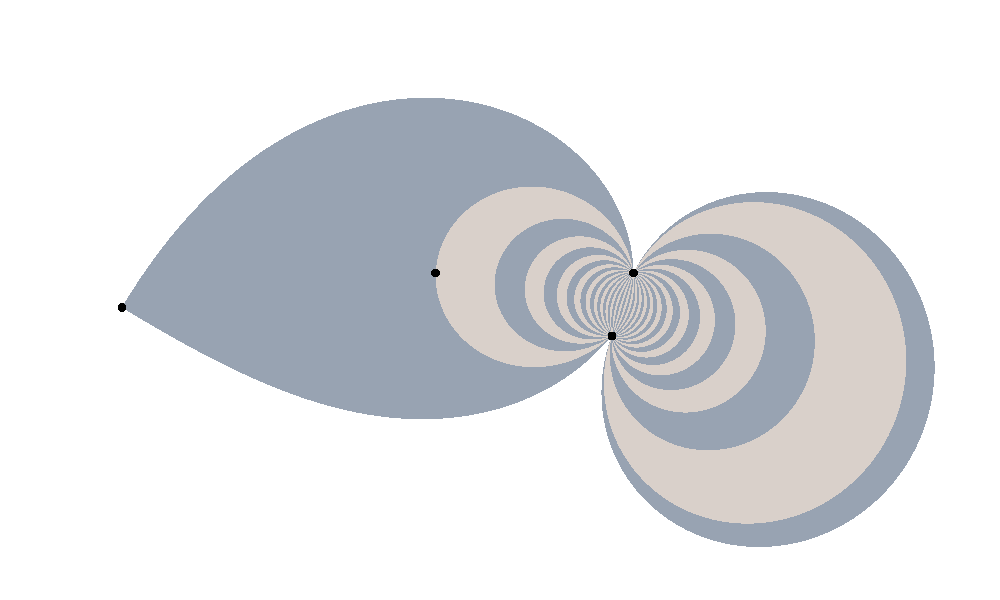}}
\rput(2.0,2.65){$\cp_h$}
\rput(3.65,2.95){$\cv_h$}
\rput(5.65,3.4){$0$}
\psline(5.2,2.05)(5.36,2.4)
\rput(5.1,1.9){$\sigma_h$}
\end{pspicture}
\end{center}
\caption{The perturbed petal $\C{P}_h$ and the various special points associated to some $h\in\IS_\R$. It has been colored so that the map $\Phi_h$ sends each strip to an infinite vertical strip of width $1$.}
\label{fig:ppetal}
\end{figure}
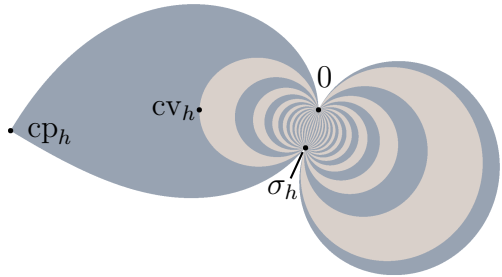

\begin{itemize}  
\item[--] The domain $\C{P}_h$ is bounded by piecewise smooth curves and is compactly contained in $V_h$. 
Moreover, it contains $\cp_h$, $0$, and $\sigma_h$ on its boundary.
\item[--] $\im \Phi_h(z) \to +\infty$ when $z\in \C{P}_h\to 0$, and $\im \Phi_h(z)\to-\infty$ 
when $z \in \C{P}_h \to \sigma_h$.
\item[--] $\Phi_h$ satisfies the \textit{Abel functional equation} on $\C{P}_h$, that is, 
\[\Phi_h(h(z))=\Phi_h(z)+1, \text{ whenever $z$ and $h(z)$ belong to $\C{P}_h$}.\] 

\item[--] $\Phi_h$ is uniquely determined by the above conditions together with normalization $\Phi_h(\cp_h)=0$. 
Moreover, the normalized $\Phi_h$ depends continuously on $h$.     
\end{itemize}

Furthermore, it is proved in \cite[Proposition 1.4]{Ch10-I} and  \cite[Proposition 12]{BC12}, 
independently, that there are integers $\B{k}$ and $\hat{\B{k}}$ independent of $h$ such that 
$\gF_h$ satisfies the following properties\footnote{The class $\IS_0$ is denoted by $\mathcal{F}_1$ 
in \cite{IS06}. 
All above statements, except the existence of uniform $\hat{\B{k}}$ and $\B{k}$, follow 
from Theorem 2.1, and Main Theorems 1, 3 in \cite{IS06}. 
The existence of uniform constants also follows from those results but requires some extra work. 
A detailed treatment of these statements is given in \cite[Proposition 1.4]{Ch10-I} and 
\cite[Proposition 12]{BC12}.}.

\medskip

\begin{itemize}  
\item[--] There exists a continuous branch of argument defined on $\C{P}_h$ such that 
\[\max_{w,w'\in \C{P}_h} |\arg(w)-\arg(w')|\leq 2 \pi \hat{\B{k}}.\]
\item[--] $\Phi_h(\C{P}_h)=\{w \in \D{C} \mid 0 < \re(w) < 1/\ga -\B{k}\}$.
\end{itemize}

The map $\Phi_h: \mathcal{P}_h\to \D{C}$ is called the \textit{perturbed Fatou coordinate}, 
or the \textit{Fatou coordinate} for short,  of $h$, and $\C{P}_h$ is called the \textit{perturbed petal}. 
See Figure~\ref{fig:ppetal}.

\subsection{Renormalization}\label{SS:renormalization}
Let us now describe how renormalization is defined in \cite{IS06}.
Let $h\colon V_h \to \D{C}$ either be in $\IS_{\ga}$ or be the quadratic polynomial $Q_\ga$, with $\ga$ in $(0,\ga^*]$. 
Let $\Phi_h\colon \C{P}_h \to \D{C}$ denote the normalized Fatou coordinate of $h$.
Define 
\begin{equation}\label{E:sector-def}
\begin{gathered}
\C{C}_h:=\{z\in \C{P}_h : 1/2 \leq \re(\Phi_h(z)) \leq 3/2 \: ,\: -2< \im \Phi_h(z) \leq 2 \},\\
\Csh_h:=\{z\in \C{P}_h : 1/2 \leq \re(\Phi_h(z)) \leq 3/2 \: , \: 2\leq \im \Phi_h(z) \}.
\end{gathered}
\end{equation}
By definition, $\cv_h\in \darun(\C{C}_h)$ and $0\in \partial(\Csh_h)$. 

Assume for a moment that there exists a positive integer $k_h$, depending on $h$, with the following properties:
\begin{itemize}
\item For every integer $k$, with $0\leq k \leq k_h$, there exists a unique connected component of $h^{-k}(\Csh_h)$ 
which is compactly contained in $\Dom h=V_h$, \footnote{The operator $\Dom$ denotes ``the domain of definition'' 
(of a map).},  and contains $0$ on its boundary. We denote this component by $(\Csh_h)^{-k}$. 
\item For every integer $k$, with $0\leq k \leq k_h$, there exists a unique connected component of 
$h^{-k}(\C{C}_h)$ which has 
non-empty intersection with $(\Csh_h)^{-k}$, and is compactly contained in $\Dom h$. 
This component is denoted by $\C{C}_h^{-k}$. 
\item The sets $\C{C}_h^{-k_h}$ and $(\Csh_h)^{-k_h}$ are contained in 
\[\{z\in\C{P}_h \mid  1/2< \re \Phi_h(z) <1/\ga -\B{k}-1/2\}.\] 
\item The maps $h: \C{C}_h^{-k}\to \C{C}_h^{-k+1}$, for $2\leq k \leq k_h$, and $h: (\Csh_h)^{-k}\to (\Csh_h)^{-k+1}$, for $1\leq k \leq k_h$, are univalent. 
The map $h: \C{C}_h^{-1}\to \C{C}_h$ is a degree two branched covering.
\end{itemize}
Let $k_h$ be the smallest positive integer satisfying the above four properties, and define
\[S_h:=\C{C}_h^{-k_h}\cup(\Csh_h)^{-k_h}.\]
\begin{figure}[ht]
\begin{center}
 \begin{pspicture}(-.5,1.2)(11.4,9)
\epsfxsize=6.3cm
\rput(3.5,5.9){\epsfbox{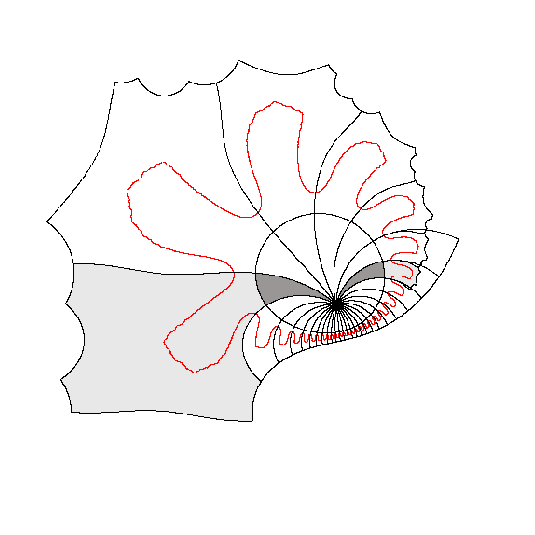}}  
  \psset{xunit=1cm}
  \psset{yunit=1cm}
    \pscurve[linewidth=.6pt,linestyle=dashed,linecolor=black]{->}(5.3,5.1)(5.3,5.5)(5.2,5.8)
    \pscurve[linewidth=.6pt,linestyle=dashed,linecolor=black]{->}(2.2,3.7)(2.3,4.1)(2.2,4.5)
    \pscurve[linewidth=.6pt,linestyle=dashed,linecolor=black]{->}(.7,6.4)(2,6.4)(3.2,6.2)(3.6,5.6)
    \rput(5.3,4.7){$S_h$}
    \rput(2.2,3.4){$\C{C}_h^{-1}$}
    \rput(.1,6.4){$(\Csh_h)^{-1}$}
    \psdots[dotsize=2pt](2.3,5.04)(3.4,4.97)
    \rput(2.,5){\small{$\cp_h$}}
    \rput(3.33,4.77){\small{$\cv_h$}}
\newgray{Lgray}{.99}
\newgray{LLgray}{.88}
\newgray{LLLgray}{.70}
\psdots(7.6,5.5)(8,5.5)(11,5.5)
\pspolygon[fillstyle=solid,fillcolor=LLLgray](7.8,7.3)(7.8,6.1)(8.2,6.1)(8.2,7.3)
\pspolygon[fillstyle=solid,fillcolor=LLgray](7.8,6.1)(7.8,4.9)(8.2,4.9)(8.2,6.1)

\pspolygon[fillstyle=solid,fillcolor=LLLgray](10.2,7.3)(10.05,6.2)(10.45,6.2)(10.6,7.3)
\pspolygon[fillstyle=solid,fillcolor=LLgray](10.05,6.2)(9.8,5.1)(9.9,5)(10,4.97)(10.1,5)(10.2,5.1)(10.45,6.2)
\psdots(7.6,5.5)(8,5.5)
\psline{->}(7.6,5.5)(11.3,5.5)
\psline{->}(7.6,4.7)(7.6,7.3)
\rput(8,5.3){\tiny{$1$}}
\rput(10.9,5.3){\tiny{$\frac{1}{\ga}-\B{k}$}}
\rput(7.3,4.9){\tiny{$-2$}}

\psline{->}(6,5.9)(7.5,5.9)
\rput(6.8,6.1){$\Phi$}

\pscurve[linestyle=dashed]{<-}(7.9,6.7)(8.9,6.9)(10.3,6.7)
\rput(9.1,7.1){\tiny{induced map}}

\psline{->}(7.4,4.7)(6.6,3.2)
\rput(7.8,4.){$\ex$}
\psellipse(6,2.1)(1.2,.9)
\psdots[dotsize=2pt](6,2.1)
\rput(4,2.6){$\C{R}' (h)$}
\rput(6.2,2.1){\small{$0$}}
\rput(1,8){$h$}
\psline[linewidth=.5pt]{->}(5.7,1.55)(5.85,1.45)(6.05,1.4)
\NormalCoor
\psdot[dotsize=1pt](5.65,1.6)
\psdot[dotsize=1pt](6.1,1.4)
 \end{pspicture}
\caption{The figure shows the sets $\C{C}_h$, $\Csh_h$,..., $\C{C}_h^{-k_h}$, and 
$(\Csh_h)^{-k_h}$. 
The ``induced map'' projects via $\ex$ to a map $\C{R}' (h)$ defined near $0$.}
\label{sectorpix}
\end{center}
\end{figure}
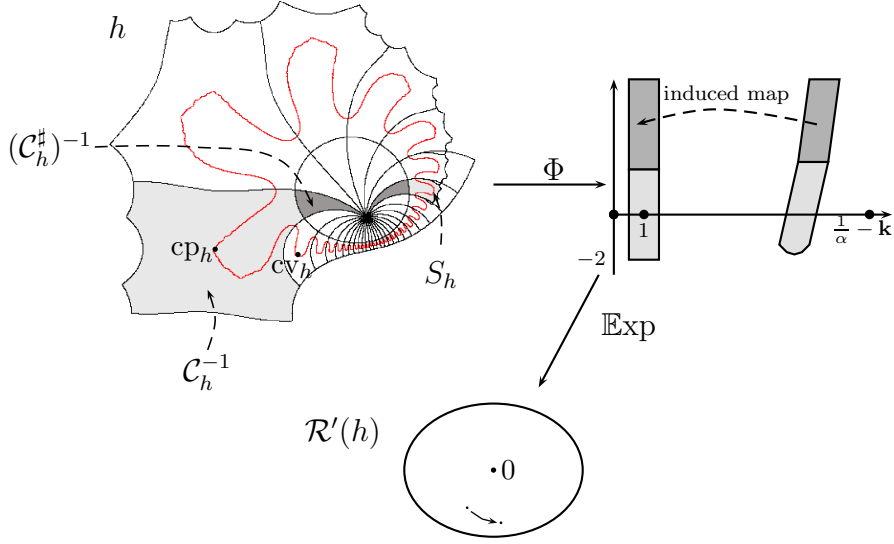

By the Abel functional equation, the map 
\begin{equation}\label{renorm-def}
\Phi_h \circ h\co{k_h} \circ \Phi_h^{-1}:\Phi_h(S_h) \to \D{C}
\end{equation}
projects via $w\mapsto z=\frac{-4}{27}e^{2 \pi \B{i} w}$ to a well-defined map $\C{R}' (h)$ defined on a set containing 
a punctured neighborhood of $0$. 
It has a removable singularity at $0$ with asymptotic expansion $e^{2 \pi \frac{-1}{\ga}\B{i}}z+ O(z^2)$ near it. 
See Figure~ \ref{sectorpix}. 

The conjugate map 
\[\C{R}''(h):= s\circ \C{R}' (h)\circ s^{-1}, \text{ where } s(z):=\bar{z},
\]
is defined on the interior of $s(\frac{-4}{27}e^{2\pi \B{i}(\Phi_h(S_h))})$ and can be extended at $0$.  
It has the form $z \mapsto e^{2 \pi \frac{1}{\ga}\B{i}}z+O(z^2)$ near $0$, and 
is normalized to have the critical value at $-4/27$. 
For future reference, we define the notation
\begin{equation}\label{E:ex} 
\ex(\gz):=s(\frac{-4}{27}e^{2\pi \B{i}\gz}).
\end{equation}

The following theorem \cite[Main theorem 3]{IS06} states that this definition of renormalization
can be carried out for certain perturbations of maps in $\IS_0$. 
In particular, this implies the existence of $k_h$ satisfying the four properties listed in the definition of 
the renormalization. 
See \cite[Proposition 13]{BC12} for a detailed argument on this\footnote{The sets $\C{C}_h^{-k}$ and 
$(\Csh_h)^{-k}$ defined here are (strictly) contained in the sets denoted by $V^{-k}$ 
and $W^{-k}$ in \cite{BC12}. 
The set $\Phi_h(\C{C}_h^{-k}\cup (\Csh_h)^{-k})$ is contained in the union 
\[D^\sharp_{-k} \cup D_{-k} \cup D''_{-k} \cup D'_{-k+1} \cup D_{-k+1} \cup D^\sharp_{-k+1}\] 
in the notations used in \cite[Section 5.A]{IS06}.}.

\begin{thm}[Inou-Shishikura]\label{Ino-Shi2} 
There exist a constant $\ga^*>0$ and a Jordan domain $U\supset \ol{V}$ such that if 
$h \in \IS_\ga\cup\{Q_\ga\}$ with $\ga \in (0,\ga^*]$, then $\C{R}''(h)$ is well-defined and has an appropriate 
restriction $\C{R}(h)$ to a smaller domain which belongs to the class $\IS_{1/\ga}$. 
Moreover, 
\begin{itemize}
\item[a)] with $\C{R}(h)=P\circ \gy^{-1}$, $\gy$ has a univalent extension onto $e^{-2\pi \B{i}/\alpha} U$,
\item[b)] there exists $\lambda<1$ such that for all $f,g\in \IS_\ga$, with $\ga\in (0,\ga^*]$, we have 
\[\Td(\pi\circ \C{R}(f),\pi\circ\C{R}(g))\leq \gl \Td(\pi(f),\pi(g)).\]
\end{itemize}
\end{thm}

The restriction $\C{R}(h)$ of $\C{R}''(h)$ to a smaller domain is called the \textit{near-parabolic renormalization} 
of $h$ by Inou and Shishikura. 
We simply refer to it as the \textit{renormalization} of $h$.
Also, when it does not lead to any confusion, we refer to both maps as $\C{R}(h)$.




\begin{rem}
A detailed numerical study of the fixed point of the parabolic renormalization has been carried 
out in \cite{LaYa11}. 
Also, a slightly modified version of Theorem~\ref{Ino-Shi2} is now announced  \cite{C14} 
for the wider class of unisingular maps.   
That is, the maps with a single critical point of (integer) order $d\geq 2$.
\end{rem}

%% file: section4.tex
\section{Lipschitz}\label{S:lipschitz}
\subsection{Statement of the main inequalities}

For a map $f\in \IS_\R$ or $f=Q_\ga$ that is linearizable at $0$, let us define 
\[\gU(f):=\log (r(\gD(f)))+B(\ga(f)),\]
where $\gD(f)$ denotes the Siegel disk of $f$, and $r(\gD(f))$ the conformal radius of 
$\gD(f)$ centered at $0$. 
Also, $B(\ga(f))=B_{1/2}(\ga(f))$ is the modified Brjuno function defined in 
Section~\ref{SS:modified-sums}.
If in addition $f$ is renormalizable in the sense of Section~\ref{sec:prelim}, we define
\[C(f)=\gU(f)-\ga(f)\gU(\C{R}(f)).\]
The main purpose of this section is to study the dependence of the above map on the linearity and non-linearity of 
$f$, that is, on $\ga(f)$ and $\pi(f)$. 
To this end, we define the Riemannian metric $ds=-\log |x| |dx|$ on the interval $[-1/2,1/2]$, where $|dx|$ is the 
standard Euclidean metric on $\D{R}$. 
Let $d_{\log}(x,y)$ denote the induced distance on the interval from this metric. 
We have $d_{\log}(-1/2,1/2)=1+\log 2<\infty$. 

Let
\[\QIS:=\{f\in \IS_\ga\u\{Q_\ga\} ; \ga\in(0,\ga_*]\},\] 
where $\ga_*\leq \ga^*$ is a positive constant that will be determined in Lemma~\ref{L:basic-estimates-lift}.
This section is devoted to the proofs of the following propositions on this class. 
The proof of Proposition~\ref{P:Lipschitz} is split in Sections~\ref{SS:reductions}, \ref{SS:the-lift} and 
\ref{SS:Proofs-estimates}. 
Proposition~\ref{P:contracting} is proved in Section~\ref{SS:Schwarzian-derivatives}, 
Propositions~\ref{P:Uniformly-bounded} and \ref{prop:samelim} are proved in 
Section~\ref{SS:size-of-Siegels}.
These propositions are used in Section~\ref{sec:arn}. 
\begin{propo}\label{P:Lipschitz}
There exists a constant $K_1$ such that
\begin{itemize}
\item for all $f,g\in \IS_{(0,\ga_*]}$ we have
\begin{equation*}
|C(f)-C(g)| \leq K_1\big[\Td(\pi(f),\pi(g))+d_{\log}(\ga(f),\ga(g))\big];
\end{equation*}
\item for all $\ga,\gb\in (0,\ga_*]$ we have 
\begin{equation*}
|C(Q_\ga)-C(Q_\gb)| \leq K_1 d_{\log}(\ga,\gb).
\end{equation*}
\end{itemize}
\end{propo}

\begin{propo}\label{P:contracting}
There exist constants $K_2$ and $\gl<1$ such that 
\begin{itemize}
\item for all $f,g\in \IS_{(0,\ga_*]}$ we have
\begin{equation*}
\Td(\pi(\C{R}(f)),\pi(\C{R}(g))) \leq \gl \Td(\pi(f),\pi(g))+K_2 |\ga(f)-\ga(g)|;
\end{equation*}
\item for all $\ga,\gb\in (0,\ga_*]$ we have
\begin{equation*}
\Td(\pi(\C{R}(Q_\ga)),\pi(\C{R}(Q_\gb))) \leq K_2 |\ga-\gb|.
\end{equation*}
\end{itemize}
\end{propo}

\begin{propo}\label{P:Uniformly-bounded}
We have 
\[\sup \{|\gU(f)| ;  f\in \QIS\}< \infty.\]
\end{propo}

\begin{propo}\label{prop:samelim}
For all $f_0 \in \IS_0 \u \{Q_0\}$, 
\[\lim_{\ga\to 0^+}\gU(z\mapsto f_0(\ea z))=\lim_{\ga\to 0^-}\gU(z\mapsto f_0(\ea z))=\log(4\pi/|f''(0)|)  
.\]
\end{propo}

\begin{rem}
Yoccoz \cite{Yoc95} has proved that $\gU(f)$ is uniformly bounded from below, 
when $f$ is univalent on a fixed domain and $\ga(f)$ is Brjuno. 
On the other hand, it is proved in \cite{BC04}, and differently in \cite{BC11}, that $\gU$ is uniformly bounded 
from above on the class of quadratic polynomials with Brjuno rotation number at $0$. 
It is proved in \cite{Ch10-I}, using a different argument, that $\gU$ is uniformly bounded from above on the set of maps in $\QIS$ with Brjuno rotation number at zero. 
We repeat a slight modification of this proof here, for the readers convenience.  
\end{rem}

\begin{rem}
The operator $\C{R}$ studied here is fiber preserving. 
That is, it sends all the maps with the same rotation number $\ga$ at $0$ to maps with the 
rotation number $-1/\ga$ at zero. 
There is an analogous notion of (bottom) near-parabolic renormalization that is defined 
by considering the ``return map'' on sectors landing at the other fixed point close to $0$: $\gs$. 
However, this near-parabolic renormalization operator (associated with $\gs$) is not 
fiber preserving. 
That is because the multiplier at the $\gs$-fixed point not only depends on the multiplier at $0$ 
but also depends on the non-linearity of the map. 
The statement of Proposition~\ref{P:contracting} has been recently announced in \cite{ChSh14} 
for this renormalization as well. 
\end{rem}

In the next subsection we reduce these propositions to some statements on the dependence of the 
perturbed Fatou coordinate on the map. 

\subsection{Reduction to the asymptotic expansion}\label{SS:reductions}
Substituting the definition of $\Upsilon$ in the definition of $C$, and using Equation~\eqref{E:Yoccoz-properties}, 
we get that for every $\ga\in (0,1/2)$
\begin{equation}\label{E:C(f)}
C(f)=\log r(f)-\ga(f) \log r(\C{R}(f))-\log \ga(f).
\end{equation}

For $f\in \QIS$, we let $\gD(f)$ denote the Siegel disk of $f$, when $f$ is linearizable. 
Define
\[\wt{\gD}(f):= \ex^{-1}(\gD(f)),\] 
where $\ex$ is defined in Equation~\eqref{E:ex}, and 
\[\C{H}(r):= \{a\in \CC\mid \im a> r\}
.\] 
By the definition of the conformal radius, there is a conformal map 
\[\gY_f:\C{H}(-\log r(f))\to  \wt{\gD}(f),\] 
such that for every $a\in \C{H}(-\log r(f))$ we have 
\[\gY_f(a+2\gp)=\gY_f(a)+1 \tand \gY_f(a)-(\frac{a}{2\gp}-\frac{\B{i}}{2\gp}\log \frac{27}{4})\to 0,
 \tas \im a\to +\infty.\]
Similarly, we have 
\[\gY_{\C{R}(f)}:\C{H}(-\log r(\C{R}(f)))\to \wt{\gD}(\C{R}(f)),\]
satisfying the above conditions. See Figure~\ref{F:asymptotes}.

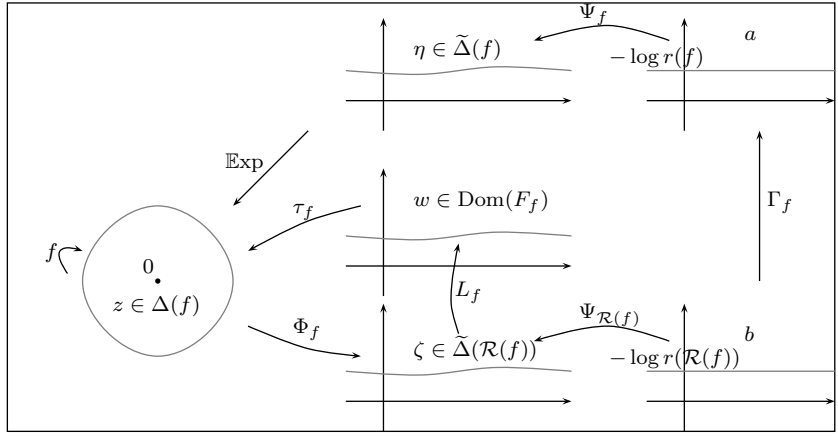
\begin{figure}
\begin{center}
\begin{pspicture}(11,5.7)
\pspolygon[linewidth=.5pt](0,0)(11,0)(11,5.7)(0,5.7)
\psline[linewidth=.5pt]{->}(8.5,.4)(11,.4)
\psline[linewidth=.5pt]{->}(9,0)(9,1.7)
\psline[linewidth=.5pt,linecolor=gray](8.5,.8)(11,.8)
\put(8,.9){\tiny $-\log r(\C{R}(f))$}
\put(9.8,1.2){\tiny $b$}
\psline[linewidth=.5pt]{->}(10,2)(10,4)
\put(10.1,3){\tiny $\gG_f$}
\pscurve[linewidth=.5pt]{->}(8.8,1.2)(8,1.4)(7,1.2)
\put(7.6,1.5){\tiny $\gY_{\C{R}(f)}$}

\psline[linewidth=.5pt]{->}(4.5,.4)(7.5,.4)
\psline[linewidth=.5pt]{->}(5,0)(5,1.7)
\pscurve[linewidth=.5pt,linecolor=gray](4.5,.8)(5.5,.75)(6.5,.85)(7.5,.8)
\put(5.4,1){\tiny $\gz\in \wt{\gD}(\C{R}(f))$}
\pscurve[linewidth=.5pt]{->}(6,1.3)(5.9,1.7)(6,2.5)
\put(5.95,1.8){\tiny $L_f$}

\psline[linewidth=.5pt]{->}(8.5,4.4)(11,4.4)
\psline[linewidth=.5pt]{->}(9,4)(9,5.5)
\psline[linewidth=.5pt,linecolor=gray](8.5,4.8)(11,4.8)
\put(8,4.9){\tiny $-\log r(f)$}
\put(9.8,5.2){\tiny $a$}
\pscurve[linewidth=.5pt]{->}(8.8,5.2)(8,5.4)(7,5.2)
\put(7.6,5.5){\tiny $\gY_f$}

\psline[linewidth=.5pt]{->}(4.5,4.4)(7.5,4.4)
\psline[linewidth=.5pt]{->}(5,4)(5,5.5)
\pscurve[linewidth=.5pt,linecolor=gray](4.5,4.8)(5.5,4.75)(6.5,4.85)(7.5,4.8)
\put(5.4,5){\tiny $\gh\in \wt{\gD}(f)$}
\psline[linewidth=.5pt]{->}(4,4)(3,3)
\put(2.9,3.5){\tiny $\ex$}

\psline[linewidth=.5pt]{->}(4.5,2.2)(7.5,2.2)
\psline[linewidth=.5pt]{->}(5,1.8)(5,3.5)
\pscurve[linewidth=.5pt,linecolor=gray](4.5,2.6)(5.5,2.55)(6.5,2.65)(7.5,2.6)
\put(5.4,3){\tiny $w\in \Dom(F_f)$}
\pscurve[linewidth=.5pt]{->}(4.7,3)(4,2.8)(3.2,2.4)
\put(3.8,2.9){\tiny $\gt_f$}

\psdot[dotsize=2pt](2,2)
\psccurve[linewidth=.5pt,linecolor=gray](2,3)(1,2)(2,1)(3,2)
\put(1.4,1.6){\tiny $z\in \gD(f)$}
\put(1.8,2.1){\tiny $0$}
\pscurve[linewidth=.5pt]{->}(3.2,1.4)(4,1.1)(4.7,1)
\put(3.8,1.3){\tiny $\gF_f$}
\pscurve[linewidth=.5pt]{->}(.8,2.1)(.7,2.4)(1,2.4)
\put(.5,2.3){\tiny $f$}
\end{pspicture}
\caption{A schematic presentation of the maps and coordinates.}\label{F:asymptotes}
\end{center}
\end{figure}

Recall the Fatou coordinate $\gF_f: \C{P}_f\to \CC$ of $f$. 
We claim that $\gF_f^{-1}$ extends to a covering map from $\wt{\gD}(\C{R}(f))$ to 
$\gD(f)\sm\{0\}$\footnote{The difficulty here is that we do not know \textit{a priori} that the image of 
$\wt{\gD}(R(f))$ under $\gF_f^{-1}$ is equal to $\gD(f)\cap \C{P}_f$, nor in fact any of the two inclusions. 
Another is that $\gD(f)\cap \C{P}_f$ and $\gD(\C{R}(f)) \cap \Dom(\gF_f^{-1})$ may be disconnected. 
Last, it is not known whether $\gD(f)$ may have a non locally connected boundary, although it is unlikely:
it is for instance conjectured that the boundary of all polynomial Siegel disks are Jordan curves.}.
\begin{lem}
The map $\gF_f^{-1}$ extends to a covering map from $\wt{\gD}(\C{R}(f))$ to $\gD(f)\sm\{0\}$.  
\end{lem}
\begin{proof}
Let us define the Fatou set of $f$ as the set of points that have a neighborhood on which all iterates of $f$ are defined 
and form a normal family.

Let $p_0$ denote the restriction of $\gF_f^{-1}$ to $\Dom \gF_f^{-1}\cap \wt{\gD}(\C{R}(f))$.  
By the definition of renormalization, if $\ex(\gz)\in \Dom \C{R}(f)$, it means that $p_0(\gz)$
has an orbit that exits the petal and comes back to it, while remaining in the finite union of $\C{P}_f$ and of the 
$\C{C}_f^{-j}$ and $(\C{C}^\sharp_f)^{-j}$, for $0\leq j\leq k_f$, using the notations of 
Section~\ref{SS:Inou-Shishikura}.
This implies that if for some $\gz\in \Dom p_0$, $\ex(\gz)$ can be iterated infinitely many times under 
$\C{R}(f)$, then $f$ can be iterated infinitely many times at $p_0(\gz)$. 

We break the rest of our argument into several steps. 
\medskip

{\em Step 1.}  The map $p_0$ extends to a map $p_2: \wt{\gD}(\C{R}(f)) \to \gD(f)\setminus \{0\}$.
\medskip

Let $\wt{\gD}' \ci \wt{\gD}(\C{R}(f))$ be the set of points $\gz_0$ such that there exists a continuous path  
$\gga: [0,1)\to \wt{\gD}(\C{R}(f))$ with $\gga(0)=\gz_0$, $\lim_{t\to 1}\im \gga(t)=+\infty$, 
$\lim_{t\to 1}\re \gga(t)$ exists, and $\re(\gga[0,1)) \ci (0,+\infty)$.  
Fix such $\gz_0$ and $\gga$ and let $\gz\in \gga[0,1)$. 
Choose a nonnegative integer $j$ as well as a positive real $\gep$ so that 
$\re (B(\gz-j, \gep))\ci (0,1/\ga-\B{k})$ and $\ex(B(\gz, \gep))\ci \gD(\C{R}(f))$. 
By the above argument,  since $\C{R}(f)$ can be iterated infinitely many times on $\ex(B(\gz, \gep))$, 
$f$ can be iterated infinitely many times on $p_0(B(\gz-j, \gep))$, with values in $\Dom f$. 
In particular, on $B(\gz, \gep)$ we can extend $p_0$ as $p_1(\cdot):=f^{j}\circ p_0(\cdot-j)$. 
It follows from the Abel functional equation that this extension of $p_0$ is independent of the choice of $j$, 
and hence, $p_1$ is uniquely determined on $B(\gz, \gep)$. 
Moreover, since $p_1$ is a non-constant open mapping, $f$ can be iterated infinitely many times on a 
neighborhood of $p_1(\gz)$,  with values in the bounded set mentioned above. 
Hence, $p_1(\gz)$ belongs to the Fatou set of $f$. 
As $\gz\in \gga[0,1)$ was arbitrary, we have a well-defined holomorphic map $p_1$ defined on a neighborhood of 
$\gga[0,1)$, with values in the Fatou set of $f$. 
Since the upper end of the curve is mapped into $\gD(f)$, and $\gD(f)$ is the connected component of the Fatou 
set on which the map is linearizable, $p_1(\gz_0)$ belongs to $\gD(f)\setminus \{0\}$. 

For each $\gz\in \wt{\gD}(\C{R}(f))$, there exists an integer $n$ such that $\gz+n\in \wt{\gD}'$.
Thus, $p_1(\gz+n)$ is defined and belongs to $\gD(f)$. 
Let $f_\gD: \gD(f)\to \gD(f)$ denote the restriction of $f$ to $\gD(f)$. 
It is a bijection. 
By the equivariance property of $p_1$, $p_2(\gz):= f_\gD\co{(-n)}(p_1(\gz+n))$ provides a well-defined extension 
of $p_1$.   

We now have an extension $p_2: \wt{\gD}(\C{R}(f)) \to \gD(f)\setminus \{0\}$ of $p_0$ that is holomorphic and 
satisfies $p_2(\gz+n)= f_\gD\co{n}(p_2(\gz))$, for all $n\in \D{Z}$ and $\gz\in \wt{\gD}(\C{R}(f))$.
\medskip

{\em Step 2.} The map $p_2: \wt{\gD}(\C{R}(f)) \to \gD(f)\setminus \{0\}$ is surjective. 
\medskip

To prove this we use a control on the post-critical set, well-known in the case of quadratic maps, and extended to the 
Inou-Shishikura class by the first author in \cite{Ch10-I}. 
Let $\pc(f)$ denote the closure of the orbit of the critical value of $f$. 
Recall that $f$ is infinitely renormalizable in the sense of Inou-Shishikura, with all consecutive renormalizations in their
class. 
It follows from the definition of renormalization that the critical point of $f$ can be iterated infinitely many times in 
$\Dom f$, and the closure of this orbit is contained in this domain (see \cite[Proposition 2.4]{Ch10-I} 
or \cite[Proposition 2.3]{Ch10-II}). 
On the other hand, it is proved in \cite[Corollary 4.9]{Ch10-I} that $\pc(f)$ forms a connected subset of 
$\Dom f$ separating zero from a neighborhood of infinity. 
This implies that the critical orbit is recurrent.
The connected component of $\D{C}\setminus \pc(f)$ containing zero is full, and must be equal to the Siegel disk 
of $f$. 
In particular, the boundary of the Siegel disk is contained in the post-critical set.
If we apply this to $\C{R}(f)$, 
we have the same results for $\pc(\C{R}(f))$, and furthermore, by the definition of renormalization, 
$\ex\circ \gF_f (\pc(f)\cap \C{P}_f)= \pc(\C{R}(f))$.    
  
Recall that the equipotentials of a Siegel disk are the images of the round circles about zero under the linearizing map.
Let $E_r$, for $r\in [0,1)$, denote the equipotentials within $\gD(f)$.
By the conjugacy property of $p_2$, the image of $p_2:\wt{\gD}(\C{R}(f)) \to \gD(f)\sm\{0\}$ is an open 
subset of $ \gD(f)$ which is connected, invariant under $f_\gD$, and contains a punctured  neighborhood of zero. 
Hence, if the map is not surjective, the image must be the region bounded by some equipotential $E_r$, $r<1$, minus 
the origin.
However, since the image has a point of $\pc(f)$ in its closure (because $\gD(\C{R}(f))$ has points of $\pc(\C{R}(f))$
in its closure), the only possibility is that it is the whole of $\gD(f)$.

\medskip


{\em Step 3.} The map $p_2: \wt{\gD}(\C{R}(f)) \to \gD(f)\setminus \{0\}$ is a covering. 
\medskip

Let $T_b$ denote the translation by $b$ defined on $\D{C}$. 
Consider the conformal isomorphism from the upper half plane $\D{H}$ to $\wt{\gD}(\C{R}(f))$ that commutes 
with $T_1$. 
Also, consider the universal covering map from $\D{H}$ to $\gD(f)\setminus \{0\}$, with deck transformation 
group generated by $T_1$. 
Then, lift $p_2$ to a map $\wt{p}_2:\D{H}\to \D{H}$. 
We claim that $p_2$ semi-conjugates $T_1$ to $f$ and 
a lift of $R(f)$ on $\wt{\gD}(\C{R}(f))$ to some iterate of $f$. 
The first claim comes from the definition and the second claim follows from the analytic continuation once we 
know it near infinity in $\wt{\gD}(\C{R}(f))$. 
By lifting, one concludes that $\wt{p}_2$ semi-conjugates (note that we do not write conjugate because we do 
not have injectivity yet) $T_1$ to a translation and $T_b$ to a translation, where $b$ is the rotation number 
of $\C{R}(f)$. 
Since $b$ is irrationnal, $\wt{p}_2$ has to be of the form $z \mapsto \theta.z+c$ for some $c\in \D{C}$ 
(the first claim implies that $f$ has a Fourier series expansion and the second claim implies that only few terms 
are non zero). 
From the surjectivity (in step $2$) we deduce that $c$ is real. 
Hence $\wt{p}_2$ is an isomorphism. 
Thus $p_2$ is a (universal) cover.

\end{proof}

The covering $\gF_f^{-1}$ provides us with a conformal isomorphism 
\[\gG_f:\C{H}(-\log r(\C{R}(f)))\to \C{H}(-\log r(f)),\]
defined as the lift $\gY_f^{-1} \circ  \ex^{-1} \circ \gF_f^{-1}  \circ \gY_{\C{R}(f)}$.
Since $\im \gG_f(t \B{i})\to +\infty$, as real $t$ tends to $+\infty$, $\gG_f$ must be an affine map of the form 
$b\mapsto Ab+B$. 
Looking at the asymptotic behavior of $\gG_f$ near infinity, we find $A=\ga$. 
Indeed, by the functional equation for $\gF_f$, $\gG_f$ maps a segment with end points $t\B{i}$ and 
$t\B{i}+2\pi$ to a curve with end points ``nearly'' $2\pi \ga$ apart, with equality as $t\to \infty$.
Adding a real constant to $\gY_f$ if necessary, and using \eqref{E:C(f)} we find 
\begin{equation}\label{E:gamma-expansion}
\gG_f(b)=\ga b-\B{i}\log \ga-\B{i}C(f).
\end{equation}

By virtue of this isomorphism, the relation between the two conformal radii can be understood from the 
asymptotic expansion of $\gG_f$ near infinity. 
To understand the dependence of $C(f)$ on $f$, we analyze how the asymptotic expansion of $\gG_f$ 
depends on $f$.

Let $f$ be a map as above, and let $\gs_f$ denote the distinguished non-zero fixed point of $f$ mentioned in 
Section~\ref{SS:Inou-Shishikura} with asymptotic expansion $\gs_f=-4\gp\ga(f)\B{i}/f''(0)+o(\ga(f))$.
We consider the covering map
\begin{equation}\label{E:covering}
\gt_f(w):=\frac{\gs_f}{1-e^{-2\gp\ga(f)\B{i}w}}:\CC\to \hat{\CC}\sm\{0, \gs_f\}.
\end{equation}
with deck transformation generated by $w\to w+1/\ga(f)$.
We have $\gt_f(w)\to 0$, as $\im w\to +\infty$, and $\gt_f(w)\to \gs_f$, as $\im w\to -\infty$.

Next, we lift $f$ under $\gt_f$ to obtain a map on an appropriate domain. 
When $f$ belongs to $\IS_{(0,\ga^*]}$, it lifts under $\gt_f$ to a map $F_f$ defined on 
$\Dom F_f:=\gt_f^{-1}(\Dom f)$. 
When $f=Q_\ga$, $\gs_\ga=(1-e^{2\pi \ga\B{i}}) 16 e^{-4\pi \ga\B{i}}/27$ is the non-zero 
fixed point of $Q_\ga$. 
We have $Q_\ga^{-1}(0)=\{0,-16 e^{-2\pi\ga \B{i}}/27\}$ and 
$Q_\ga^{-1}(\gs_\ga)=\{\gs_\ga, -16e^{-4\pi\ga\B{i}}/27\}$.
Combining with,
\[|-16 e^{-2\pi\ga \B{i}}/27-\gs_\ga|=16/27, |-16e^{-4\pi\ga\B{i}}/27-\gs_\ga|\geq16/27,\] 
we observe that the non-zero pre-image of $0$ under $Q_\ga$
and the non-$\gs_\ga$ pre-image of $\gs_{\ga}$ under $Q_\ga$ do not belong to 
$B(0,16/27)\u B(\gs_\ga, 16/27)\cup \C{P}_{Q_\ga}$. 
Hence, one can lift $Q_\ga$ on this set, to obtain a map $F_{Q_\ga}$ defined on 
$\gt_{Q_\ga}^{-1}(B(0,16/27)\u B(\gs_\ga, 16/27)\cup \C{P}_{Q_\ga})$. 
In both cases, $F_f(w)\sim w+1$ as $\im w\to \infty$, and $ \gt_f(F_f(w))=f(\gt_f(w))$, for all 
$w\in \Dom F_f$. 
We will further analyze these lifts in Subsection~\ref{SS:the-lift}.

We decompose the map $\gF_f^{-1}:\{\gx \in \CC\mid \re \gx \in (0,1/\ga(f)-\B{k})\}\to \C{P}_f$ into two maps as
$\gF_f^{-1}=\gt_f\circ L_f$, where the inverse branch of $\gt_f$ is chosen so that $\gt_f^{-1}(\C{P}_f)$ 
separates $0$ and $1/\ga(f)$.
The conformal map $L_f$ satisfies 
\begin{gather*}
L_f(\gx+1)=F_f(L_f(\gx)), \tand  \gt_f(L_f(1))=-4/27.
\end{gather*}
Define the real constant    
\[\ell_f:=\im \lim_{\im\gx\to\infty} (L_f(\gx)-\gx).\]

\begin{lem}\label{L:equi-Lipschitz}
For every $f\in \QIS$, $\lim_{\im\gx\to\infty} (L_f(\gx)-\gx)$ exists and is finite. 
Moreover, there exists a constant $K_3$ such that 
\begin{itemize}
\item for all $f,g\in \IS_{(0,\ga_*]}$ 
\begin{equation*}
|\ga(f)\ell_f-\ga(g)\ell_g|\leq K_3\big[ \Td(\pi(f), \pi(g))+d_{\log}(\ga(f),\ga(g))\big];
\end{equation*}
\item for all $\ga, \gb \in (0,\ga_*]$
\begin{equation*}
|\ga\ell_{Q_\ga}-\gb\ell_{Q_\gb}|\leq K_3d_{\log}(\ga,\gb).
\end{equation*}
\end{itemize}
\end{lem}

In a moment we will be able to prove Proposition~\ref{P:Lipschitz} using the above lemma. 
Before this can happen, we need to relate $\Td$ on $\IS_0$ to some pointwise differences.
The proof uses Koebe's distortion Theorem \cite{P} and Lehto's Majorant principle, \cite[Section~3.5]{Leh87}. 
We state the former one as it is used several times in the sequel.

\begin{thm}[Koebe]\label{T:distortion-estimates}
For every univalent map $h: B(0,1)\to \D{C}$ with $h(0)=0$ and $h'(0)=1$,  
\[\frac{1-|z|}{1+|z|^3} \leq |h'(z)|\leq \frac{1+|z|}{1-|z|^3}.\]
\end{thm}

Recall the domain $V$ defined in~\eqref{V}. 
\begin{lem}\label{L:metrics-related}
Let $V'\subset V$ be a simply connected neighborhood of zero that is compactly contained in $V$. 
There exists a constant $K_4$ such that for all $f,g \in \IS_\D{R}$, all $z\in V_g\cap V_f$ with 
$\gf^{-1}_f(z), \gf^{-1}_g(z)\in V'$ we have 
\[|f(z)-g(z)|\leq K_4 |z| \big [|\ga(f)-\ga(g)|+\Td(\pi(f), \pi(g))\big].\]   
In particular, for a possibly larger constant $K_4$ we have
\[|f''(0)-g''(0)|\leq K_4 \big [|\ga(f)-\ga(g)|+\Td(\pi(f), \pi(g))\big ].\] 
\end{lem}
\begin{proof}
First an elementary reduction as follows, 
\begin{align*} 
|f(z)-g(z)|& \leq |\gf_f^{-1}(z)-\gf_g^{-1}(z)| \cdot \sup_{w\in V'} |P'(w)|\\
&\leq |\gf_f^{-1}\circ \gf_g\circ \gf_g^{-1}(z)-\gf_f^{-1}\circ \gf_f\circ\gf_g^{-1}(z)|\cdot \sup_{w\in V'} |P'(w)|\\
&\leq\sup_{w\in V'}|\gf_f'(w)|^{-1}\cdot  |\gf_g(\gf_g^{-1}(z))-\gf_f(\gf_g^{-1}(z))| \cdot \sup_{w\in V'} |P'(w)|\\
&\leq A_1  |\gf_g(\gf_g^{-1}(z))-\gf_f(\gf_g^{-1}(z))|
\end{align*}
where, $A_1$ is a constant depending only on $V'$ (in fact, only on $\mod (V\setminus \ol{V'})$), 
by Koebe distortion Theorem.

On the other hand with $\gf:=\gf_g'(0)\gf_f/\gf_f'(0)$ and $w=\gf_g^{-1}(z)\in V'$,
\[|\gf_f(w)-\gf_g(w)| \leq |\gf_f(w)-\gf(w)|+|\gf(w)-\gf_g(w)|\]
For the first term in the right hand side, 
\begin{align*}
|\gf_f(w)-\gf(w)| &\leq |\gf_f(w)| \cdot |1-\frac{\gf_g'(0)}{\gf_f'(0)}| \\
&\leq |\gf_f\circ \gf_g^{-1}(z)| \cdot |1-\frac{\gf_g'(0)}{\gf_f'(0)}| \\
&\leq A_2 |z|\cdot |\ga(f)-\ga(g)|
\end{align*}
for some constant $A_2$, by  Koebe distortion Theorem. 

For the second term, 
\begin{align*}
|\gf(w)-\gf_g(w)|& = |\gf\circ \gf_g^{-1}(z)-z|\\ 
&\leq A_3\dil (\hat{\gf}\circ \hat{\gf}_g^{-1})\cdot |z|
\end{align*}
where $\hat{\gf}$ and $\hat{\gf}_g$ are some quasi-conformal extensions of $\gf$ and $\gf_g$, respectively. 
To obtain this inequality, we use a generalized Lehto Majorant principle on $V$ applied to the analytic 
functional $\gF(h):=(h(z)-z)/z$, which we explain below. 

Let $\gm$ be the complex dilatation of $\hat{\gf}\circ \hat{\gf}_g^{-1}$ with 
$\|\gm\|_\infty=\dil(\hat{\gf}\circ \hat{\gf}_g^{-1} ) <1$.
For complex $a$ with $|a|< 1/\|\gm\|_\infty$, let $G_a$ denote the unique normalized solution of the 
Beltrami equation $\partial_{\bar{z}}G_a= a \gm \partial _zG_a$ with $G_a(0)=0$ and $G_a'(0)=1$. 
Note that $\gm$ is identically zero on $V\ni 0$, and hence, for every fixed $a$, $G_a$ is holomorphic on $V$. 
By the holomorphic dependence of the solution of the Beltrami equation on parameter $a$, for every fixed $z\in V$,
$G_a(z)$ is also holomorphic in $a$.
Thus, $a\mapsto (G_a(z)-z)/z$ is holomorphic in $a$.
(At $z=0$, this function is defined identically zero.)
On the other hand, when $a=0$, $G_a$ is the identity mapping and hence $(G_a(z)-z)/z=0$. 
And if $z\in \gf_g(V')$, by Koebe's distortion theorem, $|(G_a(z)-z)/z|$ is uniformly bounded 
from above by a constant $A_3$, independent of $z$ and $a$. 
This is because, $G_a\circ \gf_g$ is univalent on $V$, and therefore, $G_a$ is bounded on $\gf_g(V')$, by a 
constant depending only on $V'$. 
As, $G_a(0)=0$, $G_a(z)/z$ is uniformly bounded over $\gf_g(V')$, by the maximum principle.  
In particular at $a=1$, we obtain the desired inequality from the Schwarz lemma for the variable $a$.
 

Since $\dil=\frac{\Dil-1}{\Dil +1}\leq \log \Dil$, taking infimum over all such extensions we obtain the 
desired inequality.

The second estimate in the lemma is obtained from the first one and the Cauchy integral formula. 
\footnote{Indeed, we can improve the estimates to the second order, involving $|z^2|$ and $\Td^2$ here, 
but since the coefficients of the first terms are non-zero, we do not care about the higher order terms.}
\end{proof}

\begin{proof}[Proof of Proposition~\ref{P:Lipschitz} assuming Lemma \ref{L:equi-Lipschitz}]
The plan is to find $C(f)$ in terms of $\ell_f$, by comparing the asymptotic expansion \eqref{E:gamma-expansion} 
to the one obtained from the changes of coordinates in the renormalization.

We use the following coordinates 
\[b\in \C{H}(-\log r(\C{R}(f)),\  a\in \C{H}(-\log r(f)),\ \gz\in \wt{\gD}(\C{R}(f)),\  \gh\in \wt{\gD}(f),\  z\in \gD(f)
.\]
\begin{align*}
\im a&\approx 2\gp\im \gh+\log \frac{27}{4}\\
&=-\log |z|\\
&=-[\log |\gs_f| - \log |1-e^{-2\gp\ga(f)L_f(\gz)\B{i}}|]\\
&\approx -\log |\gs_f| + 2\gp\ga(f)\im L_f(\gz)\\
&=-\log |\gs_f|+2\gp\ga(f)\im \gz+2\gp\ga(f)\im \ell_f\\
&=-\log \ga(f) + \log \frac {|f''(0)|}{4\gp} + \im b \cdot\ga(f)- \ga(f)\log \frac{27}{4}+2\gp\ga(f)\im \ell_f
\end{align*}

Now, comparing the above expression to \eqref{E:gamma-expansion}, we have 
\begin{equation}\label{E:C(f)-expanded}
C(f)=-\log|f''(0)|+\log (4\pi)-2\pi \ga \ell_f+\ga(f)\log \frac{27}{4}.
\end{equation}
Now the estimates in Lemmas~\ref{L:metrics-related} and \ref{L:equi-Lipschitz} imply Proposition~\ref{P:Lipschitz}.
Here, we use that $\log |a|-\log |b|=\log (1+(|a|-|b|)/|b|)\leq |a-b|/|b|$, and that 
$|f''(0)|, |g''(0)|\in [2,7]$.
\end{proof}

The proof of Lemma~\ref{L:equi-Lipschitz} can be divided into two parts using the triangle inequality. 
Given maps $f$ and $g$, define the new map $h(z):=f(e^{2\pi (\ga(g)-\ga(f))\B{i}}\cdot z)$, for 
$z$ in $e^{2\pi(\ga(f)-\ga(g))\B{i}}\cdot\Dom f$. 
Note that $\ga(h)=\ga(g)$ and $\pi (h)=\pi(f)$. 
By triangle inequality 
\begin{equation}\label{E:triangle}
|\ga(f)\ell_f-\ga(g)\ell_g|\leq |\ga(f)\ell_f-\ga(h)\ell_h|+|\ga(h)\ell_h-\ga(g)\ell_g|.
\end{equation}
By virtue of the above inequality, the inequalities in Lemma~\ref{L:equi-Lipschitz} follow from the next two lemmas.
The first statement in Lemma~\ref{L:equi-Lipschitz} is proved just after Sublemma~\ref{SL:log-bound}. 
\begin{lem}\label{L:vertical-direction}
There exists a constant $K_5$ such that for all $g, h\in \IS_{(0,\ga_*]}$ with $\ga(h)=\ga(g)$, we have 
\[|\ga(h)\ell_h-\ga(g)\ell_g|\leq K_5 \Td(\pi(h), \pi(g)).\]
\end{lem}
\begin{lem}\label{L:horizontal-direction}
There exists a constant $K_6$ such that 
\begin{itemize}
\item for all $f,h\in \IS_{(0,\ga_*]}$, where $f=h(e^{2\pi (\ga(f)-\ga(h))\B{i}}\cdot z)$, we have 
\[|\ga(f)\ell_f-\ga(h)\ell_h|\leq K_6d_{\log}(\ga(f),\ga(h));\]
\item for all $\ga,\gb \in (0,\ga_*]$ we have 
\[|\ga\ell_{Q_\ga}-\gb\ell_{Q_\gb}|\leq K_6d_{\log}(\ga,\gb).\]
\end{itemize}
\end{lem}
Before we embark on proving the above two lemmas, we need some basic estimates on $F_f$, as well as  
its dependence on the non-linearity of $f$ and on $\ga(f)$.

\subsection{The lift $F_f$ and its dependence on $\pi(f)$ and $\ga(f)$}\label{SS:the-lift}
Following \cite{Sh98}, the map $f\in \QIS\cup \IS_0\cup \{Q_0\}$ can be written as 
\begin{equation}\label{E:f(z)}
f(z)=z+z(z-\gs_f)u_f(z),
\end{equation} 
where $u_f$ is a holomorphic function defined on $\Dom f$, and is non-vanishing at zero and $\gs_f$.
By the pre-compactness of $\IS_0$, $u_f(0)$ is uniformly bounded from above and away from zero. 
More precisely, when $\ga=0$ one uses the bounds in Equation~\eqref{E:bound-on-f''}, and when $\ga\neq 0$
one uses the asymptotic expansion $\gs_f=-4\pi \ga \B{i}/f''(0)+o(\ga)$, where $o$ is uniform with respect to $f$.
Differentiating the above equation at zero and $\gs_f$, we obtain 
\begin{equation}\label{E:sigma-expression}
\gs_f=(1-e^{2\pi\ga(f)\B{i}})/u_f(0), f'(\gs_f)=1+\gs_f u_f(\gs_f).
\end{equation}
Recall the covering map $\gt_f(w)=\gs_f/ (1-e^{-2\pi\ga \B{i}w})$, for $f\in \QIS$.  
The map $f$ on a neighborhood of $0$ may be lifted via $\gt_f$ to a map $F_f$ defined on a domain containing 
points $w$ with large $\im w$.
The lift  $F_f$ is determined up to translation by $\D{Z}/\ga(f)$, and with any such choice, commutes with the 
translation by $1/\ga(f)$. 
One can see that $F_f$ is given by the formula
\begin{equation}\label{E:lift-formula} 
F_f(w)=w+\frac{1}{2\pi \ga(f)\B{i}}\log \big(1-\frac{\gs_f u_f(z)}{1+z u_f(z)}\big), \twith z=\gt_f(w).
\end{equation}
We shall determine a domain on which $F_f$ can be defined in Lemma~\ref{L:basic-estimates-lift}.
In particular, we pick the branch of $\log$ with $\im \log (\cdot)\ci (-\pi , +\pi)$ so that $F_f(w)\sim w+1$ as 
$\im w \to +\infty$.

\begin{lem}\label{L:estimate-on-covering}
There exists a constant $K_8$ such that for every $f \in \QIS$, we have 
\begin{gather*}
|\gt_f(w)|\leq K_8 \frac{\ga}{e^{2\pi\ga \im w}-1}, \tif \im w>0, \\
|\gt_f(w)-\gs_f| \leq K_8 \frac{\ga e^{2\pi\ga \im w}}{1-e^{2\pi \ga \im w}}, \tif \im w<0.
\end{gather*}
\end{lem}
\begin{proof}
Since $u_f(0)$ is uniformly away from zero, the inequalities  follow from the expression \eqref{E:sigma-expression} 
for $\gs_f$ and an explicit calculation on $\gt_f$. 
\end{proof}

Recall that the class $\QIS$ depends on $\ga_*$.

\begin{lem}\label{L:basic-estimates-lift}
There are constants $\ga_*$ and $K_9$ such that for all $f$ in $\QIS$, there exists a lift 
$F_f$ defined on $\D{C}\setminus B(\D{Z}/\ga,K_9)$ that satisfies the following estimates.
\begin{itemize}
\item[1)] For all $w\in \CC\setminus B(\D{Z}/\ga,K_9)$, we have 
\[|F_f(w)-(w+1)|\leq 1/4, \tand  |F_f'(w)-1|\leq 1/4.\]
\item[2)] For all $w\in \CC\setminus B(\D{Z}/\ga, K_9)$ with $\im w>0$
\[|F_f(w)-(w+1)|\leq K_9|\gt_f(w)| , \tand 
 |F_f'(w)-1|\leq K_9 |\gt_f(w)|.\]
 \item[3)] For all $w\in \CC\setminus B(\D{Z}/\ga, K_9)$ with 
 $\im w<0$
\begin{gather*}
|F_f(w)-w+\frac{1}{2\pi\ga(f) \textnormal{\B{i}}}\log f'(\gs_f) |\leq K_9 |\gt_f(w)-\gs_f|, \\
|F_f'(w)-1|\leq K_9 |\gt_f(w)-\gs_f|.
 \end{gather*}
\end{itemize}
\end{lem}

\begin{lem}\label{L:vertical-estimate-lift}
There exists a constant $K_{10}$ such that for all $f,g$ in $\IS_{(0,\ga_*]}$ with $\ga(f)=\ga(g)$, 
the lifts $F_f$ and $F_g$ satisfy the following estimates.
\begin{itemize}
\item[1)]$\forall w\in \D{C}\setminus B(\D{Z}/\ga, K_9)$, with $\im w>0$ 
\[ |F_f(w)-F_g(w)|\leq K_{10} \Td(\pi(f),\pi(g)) |\gt_f(w)|.\]
\item[2)]$\forall w\in \D{C}\setminus B(\D{Z}/\ga, K_9)$, with $\im w<0$ 
\[ \big | F_f(w)-F_g(w) + \frac{1}{2\pi\ga(f) \textnormal{\B{i}}} (\log f'(\gs_f)-\log g'(\gs_g))\big |\leq K_{10} 
\Td(\pi(f),\pi(g)) |\gt_f(w)-\gs_f|.\]
\end{itemize}
\end{lem}

\begin{lem}\label{L:horizontal-estimate-lift}
There exists a constant $K_{11}$ satisfying the following. 
Given $f$ in $\IS_0\u \{Q_0\}$ and $\ga\in (0,\ga_*]$, let $f_\ga(z):=f(e^{2\pi\ga \textnormal{\B{i}}}z)$
and $F_\ga:=F_{f_\ga}$. 
Then, 
\begin{itemize}
\item[1)]$\forall w \in \CC\setminus B(\D{Z}/\ga, K_9)$ with $\im w>0$, 
\[|\frac{d}{d\ga}F_\ga(w)|\leq K_{11} |\gt_{f_\ga}(w)|;\]
\item[2)]$\forall w \in \CC\setminus B(\D{Z}/\ga, K_9)$ with $\im w<0$,
\[|\frac{d}{d\ga}F_\ga(w)|\leq K_{11}.\]
\end{itemize}
\end{lem}

\begin{proof}[Proof of Lemmas~\ref{L:basic-estimates-lift}, \ref{L:vertical-estimate-lift}, and 
\ref{L:horizontal-estimate-lift}]\hfill

{\em \ref{L:basic-estimates-lift}-1:}
All the maps in $\IS_{\D{R}}$ are defined on a definite neighborhood of $0$.
Then, using \eqref{E:sigma-expression} and that $u_f(0)$ is uniformly away from zero, there is $\ga_*$ such that 
$\gs_f$ is contained in that neighborhood of zero, for every $f\in \QIS$.  
Now fix $\ga_*$ and choose $C$ such that $F_f$ is defined on $\D{C}\setminus B(\D{Z}/\ga,C)$, for all $f\in \QIS$. 
By making $C$ larger if necessary (independent of $\ga_*$), $z=\gt_f(w)$ can be made uniformly close to $0$ 
or $\gs_f$, for all $f\in \QIS$ and $w\in \D{C}\setminus B(\D{Z}/\ga,C)$. 
By the pre-compactness of the class $\IS$, this guarantees that $1-\frac{\gs_f u_f(z)}{1+z u_f(z)}$ is contained 
in $B(1, 0.9)\ci \D{C}$, and hence, is uniformly away from the negative real axis. 
Now, observe that
\begin{align*}
|F_f(w)-(w+1)|&=\Big |\frac{1}{2\pi \ga(f)\B{i}}\log \big(1-\frac{\gs_f u_f(z)}{1+z u_f(z)}\big)-1\Big |\\
&=\frac{1}{2\pi \ga(f)} \Big |\log \big((1-\frac{\gs_f u_f(z)}{1+z u_f(z)})-\log e^{2\pi \ga(f)\B{i}}\big) \Big| \\
&\leq \frac{C'}{2\pi \ga(f)}\Big |(1-\frac{\gs_f u_f(z)}{1+z u_f(z)})- e^{2\pi\ga(f)\B{i}} \Big|\\
& \leq C' \Big |1-\frac{u_f(z)}{(1+zu_f(z))u_f(0)}\Big|.
\end{align*}
The last expression above can be made less that $1/4$ for every $w\in \D{C}\setminus B(\D{Z}/\ga, C)$, 
by the pre-compactness of the class of maps $u_f$, provided $C$ is a large fixed constant. 
The second estimate follows from the first one using Cauchy's Integral Formula, and replacing $C$ by $C+1$. 
 
\medskip
 
{\em \ref{L:basic-estimates-lift}-2:} To prove this, we further estimate the last expression in the above estimates.
But, first note that on a fixed neighborhood of zero contained in $ \Dom f$, for all $f$, $u_f'$ is uniformly bounded above. 
Therefore, one can see that 
\begin{equation*}
 \Big |1-\frac{u_f(z)}{(1+zu_f(z))u_f(0)}\Big | \leq C' |z| =C' |\gt_f(w)|
\end{equation*} 
for $w\in \CC\setminus B(\D{Z}/\ga,C)$. 
The other estimate also follows from Cauchy Integral Formula.

\medskip

{\em \ref{L:basic-estimates-lift}-3:}
The fixed point $\gs_f$ depends continuously on $f$, as $f$ varies in $\QIS$. 
Hence, by making $\ga_*$ small enough, we may assume that $f'(\gs_f)$ is uniformly away from the negative 
real axis, say,  $|f'(\gs_f)-1|\leq 1/2$. 
This guarantees that $\log f'(\gs_f)$ is defined for our chosen branch of logarithm.  
With $z=\gt_f(w)$,  
\begin{align*}\label{E:lift-translation}
\Big | F_f(w)-w+\frac{1}{2\pi\ga(f) \B{i}}&\log f'(\gs_f)\Big | \\
&=\Big |\frac{1}{2\pi \ga(f)\B{i}}\big (\log \big (1-\frac{\gs_f u_f(z)}{1+z u_f(z)})- \log \frac{1}{f'(\gs_f)} \big) \Big |\\
& \leq \frac{C'}{2\pi \ga} \Big|(1-\frac{\gs_f u_f(z)}{1+z u_f(z)})- \frac{1}{f'(\gs_f)}\Big| \\
&\leq \frac{C'}{2\pi \ga}\sup_{t\in (0,1)}\Big|(\frac{1+(y-\gs_f)u_f(y)}{1+y u_f(y)})'\Big |_{y=t\gs_f+(1-t)z}\Big |
\cdot |z-\gs_f|\\
&\leq \frac{C' |\gs_f|}{2\pi \ga} \sup_{t\in (0,1)}\Big|(\frac{u_f(y)^2-u_f'(y)}{(1+y u_f(y))^2})\Big |_{y=t\gs_f+(1-t)z}\Big|
\cdot |z-\gs_f|\\
&\leq C'' \ga(f)  |z-\gs_f| 
\end{align*}
for some constants $C'$ and $C''$ depending only on $\QIS$. 
This implies the first inequality, and then the second one, using Cauchy Integral Formula.

\medskip

{\em \ref{L:vertical-estimate-lift}:}
For every fixed $z$ well-contained in the domains of all $f$, the analytic functionals $f\mapsto u_f(z)$, 
$f\mapsto \gs_f$, and $f\mapsto u_f(\gs_f)$ are uniformly bounded from above, by the 
pre-compactness of $\IS_0$. 
One can use the Lehto Majorant principle, or repeat the argument in the proof of Lemma~\ref{L:metrics-related}, 
to show that these functionals are Lipschitz with respect to the Teichm\"uller metric. 

For $\im w>0$, let $z=\gt_f(w)$ and $z'=\gt_g(w)$.
We have, 
\begin{align*}
|F_f(w)-F_g(w)| &= \frac{1}{2\pi \ga(f)}\Big |\log \big(1-\frac{\gs_f u_f(z)}{1+z u_f(z)}\big)-
\log \big(1-\frac{\gs_g u_g(z')}{1+z' u_g(z')}\big)\Big|\\
& \leq C' \frac{1}{2\pi \ga(f)}\Big |\frac{\gs_f u_f(z)}{1+z u_f(z)} - \frac{\gs_g u_g(z')}{1+z' u_g(z')}\Big | \\
&\leq C' \Big |\frac{u_f(z)}{(1+z u_f(z))u_f(0)}- \frac{u_g(z')}{(1+z' u_g(z'))u_g(0)}\Big|\\
&\leq C'\Big ( \Big |\frac{u_f(z)}{(1+z u_f(z))u_f(0)}- \frac{u_g(z)}{(1+z u_g(z))u_g(0)}\Big| \\
& \qq \qq \qq  + \Big |\frac{u_g(z)}{(1+z u_g(z))u_g(0)}- \frac{u_g(z')}{(1+z' u_g(z'))u_g(0)}\Big|\Big )\\
&\leq C'' |z| \Big(\Td(\pi(f),\pi(g))+\frac{1}{u_g(0)}\frac{|z-z'|}{|z|} \Big)\\
&\leq C'' |z| \Big(\Td(\pi(f),\pi(g))+\frac{1}{u_g(0)}|1-\frac{u_f(0)}{u_g(0)}| \Big)\\
&\leq C''' |\gt_f(w)|\Td(\pi(f),\pi(g)).
\end{align*}

At the lower part,  
\begin{align*}
\Big |F_f(w)-&F_g(w)+\frac{1}{2\pi\ga(f) \B{i}}\big (\log f'(\gs_f) -\log g'(\gs_g) \big)\Big |\\
&=\frac{1}{2\pi \ga(f)}\Big| \log (1-\frac{\gs_f u_f(z)}{1+z u_f(z)})+\log f'(\gs_f)  \\
&\qq \qq - \log (1-\frac{\gs_g u_g(z')}{1+z' u_g(z')})- \log g'(\gs_g) \Big |\\
&=\frac{1}{2\pi \ga(f)}\Big| \log (1-\frac{(1-e^{2\pi\ga \B{i}}) u_f(z)}{u_f(0)(1+z u_f(z))})
+\log (1+\frac{(1-e^{2\pi\ga \B{i}})u_f(\gs_f)}{u_f(0)})\\ 
&\qq \qq-\log (1-\frac{(1-e^{2\pi\ga \B{i}}) u_g(z')}{u_g(0)(1+z' u_g(z'))})-\log (1+\frac{(1-e^{2\pi\ga \B{i}})
u_g(\gs_g)}{u_g(0)})\Big|\\
&\; \;  \vdots \\
&\leq C^{(4)} |\gt_f(w)-\gs_f|\Td(\pi(f),\pi(g))
\end{align*}
with the missing steps similar to the third to sixth inequalities above, carried out near $\gs_f$ instead of $0$. 

\medskip

{\em \ref{L:horizontal-estimate-lift}:} 
Assume that $w \in\D{C}\setminus B(\D{Z}/\ga, K_9)$ with $\im w\geq 0$, 
or $w\in \partial B(\D{Z}/\ga, K_9)$. 
Let $z_\ga:=\gt_{f_\ga}(w)$ and $u_\ga:= u_{f_\ga}$. 
Substituting $\gs_f$ from \eqref{E:sigma-expression}, and by some basic algebra,
\begin{align*}
F_\ga(w)-w&=\frac{1}{2\pi\ga \B{i}} \log\big(1-\frac{\gs_\ga u_\ga(z_\ga)}{1+z_\ga u_\ga(z_\ga)}\big) \\
&=1+\frac{1}{2\pi\ga \B{i}}\log\big((1-\frac{\gs_\ga u_\ga(z_\ga)}{1+z_\ga u_\ga(z_\ga)})e^{-2\pi\ga\B{i}}\big)\\
&=1+\frac{1}{2\pi\ga \B{i}}\log\Big(1+(1-e^{-2\pi \ga\B{i}})
\big(\frac{u_\ga(z_\ga)-u_\ga(0)(1+z_\ga u_\ga(z_\ga))}{u_\ga(0)(1+z_\ga u_\ga(z_\ga))}\big)\Big).\\
\end{align*}

Let $b: \D{C}\setminus (-\infty, -1]\to \D{C}$ be a holomorphic function defined as $\log (1+z)=z b(z)$, and let
\[T(\ga, w)=\frac{1-e^{-2\pi \ga\B{i}}}{\ga}(\frac{u_\ga(z_\ga)-u_\ga(0)(1+z_\ga u_\ga(z_\ga))}{u_\ga(0) 
(1+z_\ga u_\ga(z_\ga))}).\]
Note that $|T(\ga, w)|=O(|z_\ga|)$ and $|T(\ga, w)/z_\ga-T(\gb,w)/z_\gb|= O(|\ga-\gb|)$. 
Then, for $\gb > \ga$ (without loss of generality), we have 
\begin{align*}
\lim_{\gb\to \ga} \frac{2\pi |F_\ga(w)-F_\gb(w)|}{|\ga-\gb|} &
=\lim_{\gb\to \ga} \frac{1}{|\ga-\gb|}\Big |T(\ga,w)b(\ga T(\ga,w))-T(\gb,w) b(\gb T(\gb,w))\Big|\\
& \leq|T(\ga,w)| \lim_{\gb\to \ga}\Big |\frac{b(\ga T(\ga,w))-b(\gb T(\gb,w))}{\ga-\gb}\Big| \\ 
& \qq +\lim_{\gb\to \ga} \Big |\frac{T(\ga,w)-T(\gb,w)}{\ga-\gb}\Big| |b(\ga T(\ga,w))|     \\
&\leq C'|z_\ga|. 
\end{align*}


To prove the estimate at the lower part we use the maximum principle. 
At the upper boundary $\partial B(\D{Z}/\ga, K_9)\u (\D{R}\setminus B(\D{Z}/\ga, K_9))$ we have the above inequality
where we know that $|z_\ga|$ is uniformly bounded above. 
Near the lower end, we may look at the limiting behavior of the difference as 
\begin{align*}
\lim_{\im w\to -\infty}&|F_\ga(w)- F_\gb(w)| \\
&=\Big|\frac{1}{2\pi\ga \B{i}}\log\big(1-\frac{\gs_\ga u_\ga(\gs_\ga)}{1+\gs_\ga u_\ga(\gs_\ga)}\big)-
\frac{1}{2\pi\gb \B{i}} \log\big(1-\frac{\gs_\gb u_\gb(\gs_\gb)}{1+\gs_\gb u_\gb(\gs_\gb)}\big)\Big|\\
&= \Big|\frac{1}{2\pi\ga \B{i}}\log (1+\gs_\ga u_\ga(\gs_\ga))-\frac{1}{2\pi\gb \B{i}}\log(1+\gs_\gb u_\gb(\gs_\gb))\Big|\\
&\leq C''|\ga-\gb|. 
\end{align*}  
\end{proof}

Like in \cite{Yoc95} and \cite{Sh98}, one uses the functional equation $L_f(\gz+1)=F_f\circ L_f(\gz)$ and the  estimate on $F_f$ in Lemma~\ref{L:basic-estimates-lift}-1 to extend $L_f^{-1}$ onto the set of $\gz\in \D{C}$ 
satisfying 
\[\arg (\gz-\sqrt{2}K_9)\in [\frac{-3\pi}{4}, \frac{3\pi}{4}]\mod 2\pi, 
\tand \arg (\gz-\frac{1}{\ga}+\sqrt{2}K_9)\in [\frac{\pi}{4}, \frac{7\pi}{4}]\mod 2\pi.\]   
It is a classic application of Koebe's distortion Theorem, and the inequality in the part one of   
Lemma~\ref{L:basic-estimates-lift} that $|(L_f^{-1})'|$ is uniformly bounded from above and away form zero
on the above set. 
One applies the theorem to the map $L_f$ on a disk of radius say $1+1/4$ about a point, and uses that a pair 
$\gz$ and $\gz+1$ is mapped to a pair $w$ and $F_f(w)$. 
At points where $L_f$ is not defined on such a ball, one may use the functional equation for $L_f$ to obtain bounds 
on $L_f'$ at those points. 
Now, as $\im \gz\to +\infty$ in a strictly smaller sector contained in the above set, one can apply the 
Koebe theorem on larger and larger disks, to conclude that indeed $L_f'(\gz)\to 1$.  
There are finer estimates known to hold, see \cite{Sh00,HDR08,Ch10-II}, but we do not need 
them here.
The uniform bound is frequently used in the sequel, so we present it as a lemma.

\begin{lem}\label{L:bound-on-L'}
There exists a constant $K_{12}$ such that for every $f\in \QIS$, and every 
$\gz$ with $\re(\gz)\in (1/2,1/\ga(f)-\B{k}-1/2)$ or $\re(L_f(\gz))\in (K_9,1/\ga(f)-K_9)$,
\[1/K_{12} \leq |L_f'(\gz)| \leq K_{12}.\]
Moreover, as $\im \gz\to +\infty$, $L_f'(\gz)\to 1$.
\end{lem} 

\subsection{Proof of the estimates on the asymptotic expansion}\label{SS:Proofs-estimates}
Here we shall show that Lemma~\ref{L:vertical-estimate-lift} implies Lemma~\ref{L:vertical-direction}, and 
Lemma~\ref{L:horizontal-estimate-lift} implies Lemma~\ref{L:horizontal-direction}.

\begin{proof}[Proof of Lemma \ref{L:vertical-direction}]
Within this proof all the constants $A_1, A_2, \dots$ depend only on the class $\QIS$.   

Given $f$, let $v'_f:=L_f(1)$, which is a lift of $-4/27$ under $\gt_f$ and a critical value of $F_f$.
By assuming $\ga_*\leq (2K_9+2)^{-1}$, and the pre-compactness of the class $\IS_0$, there exists a 
non-negative integer $\gi$ such that for all $f\in \QIS$ the vertical line through $v''_f:=F_f\co{\gi}(v'_f)$ is 
contained in $\CC \setminus B(\D{Z}/\ga, K_9)$. 
Again, by the pre-compactness of the class $\IS$, there is $\gep>0$ such that 
$f\co{\gi}(-4/27)\in \D{C}\setminus B(0,\gep)$, for all $f\in \QIS$. 
Then, we make $\ga_*$ small enough so that $\gs_f\in B(0,\gep)$, and 
$|(1-e^{2\pi \ga\B{i}})/(f\co{\gi}(-4/27)u_f(0))|\leq 0.9$, for all $f\in \QIS$.  
Now the Euclidean distance between any two such values can be estimated as follows. 

\begin{equation}\label{E:value-distance}
\begin{aligned}
|v''_g- v''_h|&= |\gt_g^{-1}(g\co{\gi}(-4/27))-\gt_h^{-1}(h\co{\gi}(-4/27))|\\
&\leq|\gt_g^{-1}(g\co{\gi}(-4/27))-\gt_g^{-1}(h\co{\gi}(-4/27))|\\
&\qq +|\gt_g^{-1}(h\co{\gi}(-4/27))-\gt_h^{-1}(h\co{\gi}(-4/27))|\\
&\leq \frac{1}{2\pi \ga}\Big |\log (1-\frac{1-e^{2\pi\ga\B{i}}}{g\co{\gi}(-4/27)u_g(0)})-
\log (1-\frac{1-e^{2\pi\ga\B{i}}}{h\co{\gi}(-4/27)u_g(0)})\Big|\\
&\qq+\frac{1}{2\pi \ga}\Big |\log (1-\frac{1-e^{2\pi\ga\B{i}}}{h\co{\gi}(-4/27)u_g(0)})-
\log(1-\frac{1-e^{2\pi\ga\B{i}}}{h\co{\gi}(-4/27)u_h(0)})\Big| \\
&\leq A_1 \Td(\pi(h), \pi(g)),
\end{aligned}
\end{equation}
for a constant $A_1$ independent of  $g,h$, and $\ga$.
In the above, $\log$ is the principal branch of logarithm. 

By virtue of Lemma~\ref{L:basic-estimates-lift}-1, the vertical line $v''_f+t\B{i}$, for $t\in \D{R}$, 
is mapped away from itself under $F_f$. 
Hence, using an standard procedure, we may define a homeomorphism $H_f:[0,1]\times \D{R}\to \CC$ as
\begin{align*}
H_f(s,t):=(1-s)(v''_f+t\B{i})+ sF_f(v''_f+t\B{i}).
\end{align*}
The asymptotic behavior of this map is  
\begin{equation}\label{E:H-asymptote}
H_f(s,t)\approx s+\B{i}t+v''_f, \tas t\to +\infty.
\end{equation}

With complex notation $\gz=s+\B{i}t$, the first partial derivatives of $H$ are given by
\begin{equation}\label{E:partials-of-H}
\begin{aligned}
\partial_{\gz}H_f(s,t)= \frac{1}{2}[F_f(v_f''+\B{i}t)-(v_f''+\B{i}t)+1+ s(F_f'(v_f''+\B{i}t)-1)],\\
\partial_{\bar{\gz}}H_f(s,t)=\frac{1}{2}[F_f(v_f''+\B{i}t)-(v_f''+\B{i}t)-1-s(F_f'(v_f''+\B{i}t)-1)].
\end{aligned}
\end{equation}
Using the estimates in Lemma~\ref{L:basic-estimates-lift}-1, and that $H$ is \textit{absolutely continuous on lines}, 
$H$ is indeed a quasi-conformal mapping.  
Moreover, the estimates in part $2$ and $3$ of the same lemma provide us the following inequalities
\begin{itemize}
\item if $\im H_f(\gz)>0$,  
\begin{equation}\label{E:model-estimates-up}
|\partial_{\bar{\gz}} H_f(\gz)|, |\partial_\gz H_f(\gz)-1| \leq K_9 |\gt_f(v_f''+t\B{i})|,
\end{equation}
\item if $\im H_f(\gz)<0$, 
\begin{equation}\label{E:model-estimates-down}
\begin{gathered}
|\partial_{\bar{\gz}} H_f(\gz)+\frac{1}{2}+\frac{1}{4\pi\ga\B{i}}\log f'(\gs_f)|\leq K_9 |\gt_f(v_f''+t\B{i})-\gs_f|\\
|\partial_\gz H_f(\gz)-\frac{1}{2}+\frac{1}{4\pi \ga\B{i}}\log f'(\gs_f)| \leq K_9 |\gt_f(v_f''+t\B{i})-\gs_f|.
\end{gathered}
\end{equation}
\end{itemize}
The complex dilatation of $H_f$ is given by
\[\gm_f(\gz):=\frac{\partial_{\bar{\gz}} H_f}{\partial_\gz H_f}(\gz).\]

Note that $L_f([0, 1/\ga-\B{k}]\times \D{R})$ does not necessarily contain $H_f([0,1]\times \D{R})$. 
However, by the remark before Lemma~\ref{L:bound-on-L'}, the map $L_f$ has univalent extension onto a larger 
domain such that its image contains $H_f([0,1]\times \D{R})$. 
Hence, we may consider the map  
\[G_f:=L_f^{-1} \circ H_f:[0,1]\times \D{R}\to \CC,\] 
using the uniquely extended $L_f^{-1}$. 

The differentials of $L_f$, $G_f$ and $H_f$ are contained in a compact subset of $\D{C}\setminus \{0\}$ that 
does not depend on $f$. 
Since $G_f(\gz+1)=G_f(\gz)+1$ on the boundary of the domain of $G_f$, we can extend this map to a 
quasi-conformal mapping of the complex plane using this relation. 
By the definition, $G_f$ maps $n\in \D{Z}$ to $\gi+n$, and its differential tends to the identity when 
$\im(z)$ tends to $+\infty$. 
This follows from the same statement for $H$ in \eqref{E:model-estimates-up} and for $L$ in 
Lemma~\ref{L:bound-on-L'}.
Note that the  complex dilatation of $G_f$ at $\gz$,  $\gm(G_f)(\gz)$, is equal to $\gm_f(\gz)$.
See Figure~\ref{F:comparing-Fatou-coordinates}.

\begin{figure}
\begin{center}
\begin{pspicture}(11,5)
\pspolygon[linewidth=.5pt](0,0)(11,0)(11,5)(0,5)

\psline[linewidth=.5pt]{->}(.5,3.4)(3.5,3.4)
\psline[linewidth=.5pt]{->}(1,3)(1,4.5)

\pscustom[linecolor=lightgray,linewidth=.5pt,fillstyle=solid,fillcolor=lightgray]{
    \psline[liftpen=1](1.5,3)(1.5,4.5)    
     \pscurve[liftpen=1](2,4.5)(1.95,3.8)(2.05,3.4)(2,3)}
     
\rput(1.5,4.7){\tiny $F_g$}     
 
 \psdot[dotsize=2pt](1.5,3.7)
 \rput(1.3,3.7){\tiny $v_g''$}     

\psline[linewidth=.5pt]{->}(8.5,3.4)(11,3.4)
\psline[linewidth=.5pt]{->}(9,3)(9,4.5)
\pscustom[linecolor=lightgray,linewidth=.5pt,fillstyle=solid,fillcolor=lightgray]{
    \psline[liftpen=1](9.5,3)(9.5,4.5)    
     \pscurve[liftpen=1](10,4.5)(9.95,3.8)(10.05,3.4)(10,3)}

\rput(9.5,4.7){\tiny $F_h$}     

 \psdot[dotsize=2pt](9.5,3.2)
 \rput(9.3,3.2){\tiny $v_h''$}     

\pspolygon[linecolor=lightgray,linewidth=.5pt,fillstyle=solid,fillcolor=lightgray](4,1)(4.5,1)(4.5,3.5)(4,3.5)
\psline[linewidth=.5pt]{<->}(4,1)(4,3.5)
\psline[linewidth=.5pt]{->}(3.8,2)(5.5,2)
\rput(3.9,2.15){\tiny $0$} 

\psline[linewidth=.5pt]{->}(3.7,2.4)(2.5,3)
\rput(3.1,2.9){\tiny $H_g$}          
\psline[linewidth=.5pt]{->}(3.7,1.8)(2.5,1)
\rput(3,1.7){\tiny $G_g$} 

\psline[linestyle=dashed,linewidth=.5pt]{->}(5.1,3)(6.8,3)
\rput(6,3.1){\tiny identity} 

 \psdot[dotsize=2pt](4,2)

\pspolygon[linecolor=lightgray,linewidth=.5pt,fillstyle=solid,fillcolor=lightgray](7.1,1)(7.6,1)(7.6,3.5)(7.1,3.5)
\psline[linewidth=.5pt]{<->}(7.1,1)(7.1,3.5)
\psline[linewidth=.5pt]{->}(6.7,2)(8.5,2)
\rput(7,2.15){\tiny $0$} 

\psline[linewidth=.5pt]{->}(7.7,2.4)(8.7,3)
\rput(8.1,2.9){\tiny $H_h$}          
\psline[linewidth=.5pt]{->}(7.7,1.8)(8.7,1)
\rput(8.1,1.7){\tiny $G_h$} 

 \psdot[dotsize=2pt](7.1,2)

\psline[linewidth=.5pt]{->}(.5,.4)(3.5,.4)
\psline[linewidth=.5pt]{->}(1,0)(1,1.7)

\pscustom[linecolor=lightgray,linewidth=.5pt,fillstyle=solid,fillcolor=lightgray]{
    \pscurve[liftpen=1](2,0)(2.05,.4)(1.95,.8)(2,1.5)
     \pscurve[liftpen=1](1.5,1.5)(1.45,.8)(1.55,.4)(1.5,0)}
     
\pscurve[linewidth=.5pt]{->}(2,1.7)(2,2)(2,2.8)
\rput(1.8,2.2){\tiny $L_g$}     

 \psdot[dotsize=2pt](1.55,.4)   
 
\psline[linestyle=dashed,linewidth=.5pt]{->}(4.8,.4)(7,.4)
\rput(6,.6){\tiny $\gO=G_h\circ \text{identity}\circ  G_g^{-1}$}

\psline[linewidth=.5pt]{->}(8.5,.4)(11,.4)
\psline[linewidth=.5pt]{->}(9,0)(9,1.7)
\pscustom[linecolor=lightgray,linewidth=.5pt,fillstyle=solid,fillcolor=lightgray]{
   \pscurve[liftpen=1](10,0)(10.05,.4)(9.95,.8)(10,1.5)
     \pscurve[liftpen=1](9.5,1.5)(9.45,.8)(9.55,.4)(9.5,0)}

\pscurve[linewidth=.5pt]{->}(10,1.7)(10,2)(10,2.8)
\rput(9.8,2.2){\tiny $L_h$}          

 \psdot[dotsize=2pt](9.55,.4)     
\end{pspicture}
\caption{The maps $H_g$ and $H_h$ are models for the Fatou coordinates $L_g$ and $L_h$, respectively. 
The map $L_g$ is compared to $L_h$ by studying $\gO=G_h\circ  G_g^{-1} $ using its complex dilatation. 
Besides, the complex dilatation of $\gO$ corresponds to the one for $H_h\circ H_g^{-1}$.}
\label{F:comparing-Fatou-coordinates}
\end{center}
\end{figure}
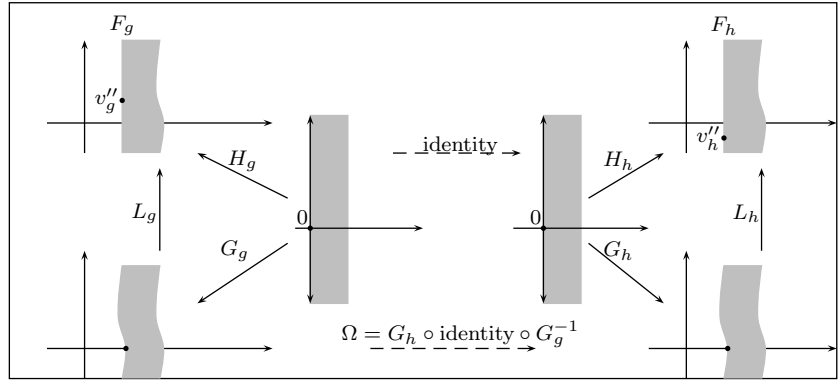
The asymptotic translation factor of every $G_f$ near $+\B{i}\infty$ is finite, with a 
bound given in the next statement.

\begin{sublem}\label{SL:log-bound}
There exists a constant $A_3$ such that for all $f\in \QIS$, and $\gz$ with $\im \gz>0$
\[| G_f(\gz)- \gz| \leq A_3(1-\log \ga(f)).\]
Moreover,  $\lim_{\im \gz\to +\infty} (G_f(\gz)- \gz)$
exists.
\end{sublem}

\begin{proof}
Let $\gP:=[0,1]\times [d,D]$, for $D> d\geq 0$, and see that
\[\oint_{\partial \gP}G_f(\gz)\,d\gz=(D-d)\B{i}-\int_{0}^1 G_f(D\B{i}+t)\,dt+\int_{0}^1 G_f(d\B{i}+t)\, dt,\]
and (recall that $L_f$ is conformal) 
\begin{align*}
|\iint_\gP \partial_{\bar{\gz}}&G_f(\gz)\, d\gz d\bar{\gz}|\\
&\leq \sup_\gP  |(L_f^{-1})' |\iint_\gP |\partial_{\bar{\gz}}H_f(\gz)|\, d\gz d\bar{\gz}
&& \\
&\leq A_4\iint_\gP|\gt_f(v_f''+t\B{i})|\,d\gz d\bar{\gz}
&& (\text{Lem.~\ref{L:bound-on-L'},\! Eq.\!\! \eqref{E:model-estimates-up}})\\
& \leq A_5 \big (1+\iint_{[0,1]\times [d+1,\infty]} |\frac{\gs_f}{1-e^{-2\pi\ga(f)\im w}}|\,dw  d\bar{w} \big )
&& (\text{Lem. \ref{L:estimate-on-covering}})\\
&\leq A_5 \big (1+\frac{2\pi \ga(f)}{u_f(0)}\int_{d+1}^{\infty} \frac{1}{e^{2\pi\ga(f) t}-1}\,dt \big )
&& \\
&\leq A_5\big (1 +\frac{1}{u_f(0)}(\log (1-e^{-2\pi \ga(f)t})\Big|_{d+1}^{\infty})  \big)
&&\\
&\leq A_6 (1-\log (1-e^{-2\pi \ga (d+1)})) &&\\
& \leq A_6 (1-\log (1-e^{-2\pi \ga})) \leq A_6 (1-\log (\ga/2)) && \\
&\leq  2 A_6 (1-\log \ga). &&
\end{align*}
By Lemma~\ref{L:bound-on-L'} and that $|\partial_\gz H(\gz)|$ is uniformly bounded with 
$\partial_\gz H(\gz)\to 1$ as $\im \gz \to +\infty$, we know that $|\partial_\gz G|$ is uniformly bounded and 
tends to one, as $\im \gz \to +\infty$. 
Also, recall that $G_f(0)=\gi$ is uniformly bounded. 
This implies that 
\[\Big |\int_{0}^1 G_f(D\B{i}+t)\,dt- G_f(D\B{i})\Big | \tand \;\Big |\int_{0}^1 G_f(d\B{i}+t)-G_f(d\B{i})\, dt\Big |\]
are uniformly bounded from above, independent of $f$. 
However, by Green's Formula, the integrals over $\partial \gP$ and $\gP$ are equal. 
With $d=0$, we obtain the desired inequality of the sub-lemma along the imaginary axis. 
Now, the bound at the other points follow from the uniform bound on $|\partial_\gz G|$. 

To prove that the limit exists we use that $\partial_\gz G_f(\gz)\to 1$, as $\im \gz\to +\infty$.  
Let $d\to +\infty$ (while $D>d$), and observe that by the estimate given above, the integral over $\Pi$ tends 
to zero, and the integral over $\partial \Pi$ tends to $D\B{i}- G_f(D\B{i})+ G_f(d\B{i})-d\B{i}$. 
The function $t \mapsto G_f(t\B{i})-t\B{i}$ satisfies Cauchy's criterion of convergence, 
when $t$ tends to infinity. 
This implies that $G_f(\gz)-\gz$ has limit along the imaginary axis. 
However, since $\partial_\gz G_f\to 1$, as $\im \gz\to +\infty$, we conclude that $G_f(\gz)-\gz$ has limit 
along the strip $[0,1]\times(0,+\infty)$.
\end{proof}
By the pre-compactness of class $\IS_0$, $\{v_f''; f\in \QIS \}$ is compactly contained in $\D{C}\setminus \{0\}$. 
Therefore, from Equation~\eqref{E:H-asymptote} and the above sub-lemma, we conclude the first part of 
lemma~\ref{L:equi-Lipschitz} and that there exists a constant $K_{13}$ such that for all $f\in \QIS$, 
\begin{equation}\label{E:log-translation-main}
|\ell_f| \leq K_{13}(1-\log \ga(f)).
\end{equation}

To compare the asymptotic translations of the maps $G_g$ and $G_h$ near infinity, we consider the map 
$\gO:=G_h\circ G_g^{-1}$.
By the composition rule, the complex dilatation of $\gO$ at  $\gx=G_g(\gz)$ is given by  
\begin{align*} 
\gm_\gO(\gx)&=\frac{\gm_h(\gz)-\gm_g(\gz)}{1-\gm_g(\gz)\ol{\gm_h(\gz)}} 
\big (\frac{\partial_\gz G_g (\gz)}{|\partial_\gz G_g (\gz)|}\big )^2.
\end{align*}
We claim that when $\im L_g(\gx)>0$
\begin{gather*} 
|\gm_\gO(\gx)|\leq A_7\Td(\pi(g),\pi(h)) |\gt_g(L_g(\gx))|.
\end{gather*}
To see this, first note that $|\gm_g(\gz)|$ and $|\gm_h(\gz)|$ are bounded above by $1/3$, using the formulas given 
in \eqref{E:partials-of-H} and the inequalities in Lemma~\ref{L:basic-estimates-lift}-1.  
For the same reason, $\partial_\gz H_g (\gz)$ and $\partial_\gz H_h (\gz)$ belong to $B(1, 1/4)$, for all 
$\gz \in [0,1]\times \D{R}$.
Also, since $\{v_f''\mid f\in \QIS\}$ is compactly contained in $\D{C}\setminus \{0\}$, 
$|\gt_h(v_h''+\B{i}t)|$ and $|\gt_g(v_g''+\B{i}t)|$ are comparable by constants independent of $g,h\in \QIS$ and 
$t\in \D{R}$.  
Then, by a simple algebraic manipulation, it only remains to show that 
\begin{gather*}
|F_h(v_h''+\B{i}t)-(v_h''+\B{i}t) - F_g(v_g''+\B{i}t)+(v_g''+\B{i}t)|  \leq \Td(\gp(g),\gp(h)) | \gt_g(v_g''+\B{i}t)|,\\
|F_h'(v_h''+\B{i}t)- F_g'(v_g''+\B{i}t)| \leq \Td(\gp(g),\gp(h)) | \gt_g(v_g''+\B{i}t)|.
\end{gather*}
To obtain the first estimate we add and subtract two terms and then use the integral remainder formula in Taylor's theorem
\begin{align*}
|F_h(v_h''&+\B{i}t)-(v_h''+\B{i}t) - F_g(v_g''+\B{i}t)+(v_g''+\B{i}t)|   &&  \\
&\leq |F_h(v_h''+\B{i}t)-(v_h''+\B{i}t) - F_g(v_h''+\B{i}t) + (v_h''+\B{i}t)  && \\ 
&\q  +F_g(v_h''+\B{i}t) - (v_h''+\B{i}t) - F_g(v_g''+\B{i}t)+(v_g''+\B{i}t)| &&\\
&\leq K_{10} \Td (\gp(g),\gp(h)) |\gt_h(v_h''+\B{i}t)|&&\!\! \text{(Lemma~\ref{L:vertical-estimate-lift})} \\
& \q +|v_h''-v_g''| \Big |1-\int_0^1 F_g'(v_g''+\B{i}t + s (v_h''+\B{i}t -v_g''-\B{i}t)) ds \Big |&&\!\! \text{(Taylor's Thm.})\\
&\leq A_7'\Td(\gp(g),\gp(h)) | \gt_g(v_g''+\B{i}t)| && \text{(Lemma~\ref{L:basic-estimates-lift})}
\end{align*}
The other inequality is obtained from Lemma~\ref{L:vertical-estimate-lift} and the Cauchy Integral Formula.  

At the lower end, $\gm_\gO$ has a constant asymptotic value $C_\gO$ that is obtained from the
equations in~\eqref{E:model-estimates-down} that satisfies 
\[|C_\gO|\leq A_7' \frac{1}{2\pi \ga}|\log h'(\gs_h)-\log g'(\gs_g)|\leq A_7\Td(\pi(g),\pi(h)). \] 
By a similar calculations, one obtains the following estimate at the lower part 
\begin{gather*}
|\gm_\gO(\gx)-C_\gO|\leq A_7\Td(\pi(g),\pi(h)) |\gt_g(L_g(\gx))-\gs_g|.
\end{gather*}

Let $\gga_1$ and $\gga_2$ be two closed horizontal curves on $\CC/\D{Z}$ at heights $-R$ and $+R$, respectively. 
Also, let $\gga_3$ be a circle of radius $\gep$ around $[0]\in \CC/\D{Z}$. 
All these curves are considered with the induced positive orientation. 
Also, let $U_{R,\gep}$ be the set on the cylinder bounded by these curves.


By Green's integral formula, 
\begin{align}\label{E:Green-on-boundary}
\int_{\gga_1\cup \gga_2\cup \gga_3} (\gO(\gx)-\gx)\frac{\tan (\pi\gx)+\B{i}}{2\tan (\pi\gx)}\, d\gx
=\iint_{U_{R,\gep}}\partial_{\bar{\gx}}\gO(\gx) 
\frac{\tan (\pi\gx)+\B{i}}{2\tan (\pi\gx)}\, d\bar{\gx}d\gx
\end{align}
As $R\to +\infty$ and $\gep\to 0$, the left hand integral tends to 
\begin{equation}\label{E:Green-on-domain}
\int_{\gga_1\cup \gga_2\cup \gga_3} (\gO(\gx)-\gx)\frac{\tan (\pi\gx)+\B{i}}{2\tan (\pi\gx)}\, d\gx
\to \lim_{\im \gx\to +\infty}(\gO(\gx)-\gx).
\end{equation}
This is because as $\gep\to 0$ and $R\to +\infty$, on $\gga_1$, we have $|\gO(\gx)-\gx|=O(|\gx|)$ 
and the ratio tends to $0$ exponentially. 
Hence, the integral tends to zero. 
On $\gga_2$, $\gO(\gx)-\gx$ tends to a constant and the ratio tends to one, providing the limit on the right hand side. 
On the curve $\gga_3$, $|\gO(\gx)-\gx|=O(\gep)$ and the ratio is $O(1/\gep)$, while the integral is on a 
circle of circumference $2\pi \gep$. 
Thus the integral on $\gga_3$ also tends to zero.   

Using the uniform bounds on $|L_f'|$ and $|\partial _\gz H_f|$ we have a uniform bound on $|\partial_\gx \gO|$.  
This implies that $|\partial _{\bar{\gx}} \gO|\leq A_8 |\gm_\gO|$. 
To bound the other integral we decompose the cylinder into two sets  
\[X^+=\{\gx\in \D{C}\mid L_g^{-1}(\gx)\in H([0,1)\times \D{R}), \im L_g^{-1}(\gx)\geq 0\}/\D{Z},\]
\[X^-=\{\gx\in \D{C}\mid L_g^{-1}(\gx)\in H([0,1)\times \D{R}), \im L_g^{-1}(\gx)< 0\}/\D{Z}.\]
Note that $(\tan (\pi\gx)+\B{i})/(2\tan (\pi\gx))$ has a singularity of order one at $\gx=0$. 
But, it is uniformly bounded at points on the cylinder above the line $\im \gx=1$. 
Then, 
\begin{equation}\label{E:Teich-bound}
\begin{aligned}
\Big |\iint_{\CC/\D{Z}}\partial_{\bar{\gx}}\gO(\gx)&
\frac{\tan (\pi\gx)+\B{i}}{2\tan (\pi\gx)}\, d\bar{\gx}d\gx \Big| \\
&\leq A_8\iint_{\CC/\D{Z}} |\gm_\gO(\gx)| \Big|\frac{\tan (\pi\gx)+\B{i}}{2\tan (\pi\gx)}\Big|\, d\bar{\gx}d\gx\\
&\leq A_8 A_7 \Td(\pi(g),\pi(h))
\iint_{X^+} |\gt_g(L_g(\gx))| \Big|\frac{\tan (\pi\gx)+\B{i}}{2\tan (\pi\gx)}\Big|\, d\bar{\gx}d\gx ,  \\ 
&\qq  +A_8 \Td(\pi(g),\pi(h)) 
 \iint_{X^-} A_9 \Big|\frac{\tan (\pi\gx)+\B{i}}{2\tan (\pi\gx)}\Big|\, d\bar{\gx}d\gx \\
& \leq A_{10} \Td(\pi(g),\pi(h)) \log \ga(g)^{-1}.
\end{aligned}
\end{equation}
To obtain the last inequality, we have estimated the definite integrals as follows. 
The integral over $X^-$ is uniformly bounded since the ratio exponentially tends to zero as $\im \gx\to -\infty$. 
Over $X^+$ we have 
\begin{align*}
\iint_{X^+} & |\gt_g(L_g(\gx))| \Big|\frac{\tan (\pi\gx)+\B{i}}{2\tan (\pi\gx)}\Big|\, d\bar{\gx}d\gx && \\
&\leq \iint_{L_g(X^+)} |L_g'(\gx)|^{-2}|\gt_g(w)| \Big|\frac{\tan (\pi\gx)+\B{i}}{2\tan (\pi\gx)}\Big|\, d\bar{w}dw&& 
(\text{with }w=L_g(\gx))\\
&\leq K_{12}^2 A_{11} \iint_{L_g(X^+) \cap (\D{R}\times [0,1])} 
\Big|\frac{\tan (\pi L_g^{-1}(w))+\B{i}}{2\tan (\pi L_g^{-1}(w))}\Big|\, d\bar{w}dw \\
&\qq +K_{12}^2 A_{11} \iint_{L_g(X^+)\cap (\D{R}\times [1,\infty])} |\gt_g(w)| \, d\bar{w}dw  &&\\
&\leq K_{12}^2 A_{11}A_{12}+ K_{12}^2 A_{11}A_{12}\log \ga(g)^{-1}. &&
\end{align*}

Finally, recall that $\ga(g)=\ga(h)$ and  
\begin{align*}
|\ga(g)\ell_g&-\ga(h)\ell_h| \\
&=\ga(g)\lim_{\im \gz\to +\infty} \big| (H_g\circ G_g^{-1}(\gx)-\gx) - (H_h\circ G_h^{-1}(\gx)-\gx)\big| \\
&=\ga(g)\lim_{\im \gz\to +\infty}\big|\im v''_g-\im v''_h-\im (G_g(\gz)-\gz)+\im (G_h(\gz)-\gz)\big|\!\!\!\!
&\text{(Eq. \ref{E:H-asymptote})} \\
&= \ga(g) |\im v''_g-\im v''_h|+\ga(g)\lim_{\im \gz\to +\infty}|\im (G_g(\gz)-G_h(\gz))| 
& \text{(Eq. \ref{E:value-distance})}  \\
&\leq (A_8+A_{10}) \Td(\pi(g),\pi(h))
&\text{(Eq. \ref{E:Teich-bound})}
\end{align*}\qedhere
\end{proof}


\begin{proof}[Proof of Lemma \ref{L:horizontal-direction}]
In what follows, all the constants $A_1$, $A_2$, $A_3$, $\dots$, depend only on the class $\QIS$.

Let $f_\ga:=f_0(e^{2\pi\ga\B{i}}z)$ be in $\QIS$ with $f_\ga'(0)=e^{2\pi\ga \B{i}}$, 
and let $F_\ga$ denote the lift of $f_\ga$ under $\gt_{f_\ga}$. 
Also, let $L_\ga:=L_{f_\ga}$ denote the linearizing map of $F_\ga$ with asymptotic translation 
$\ell_\ga:= \ell_{f_\ga}$ near infinity. 
By Inequality~\eqref{E:log-translation-main},
\begin{equation}\label{E:log-translation}
|\ell_\ga|\leq K_{13}(1-\log \ga).
\end{equation}

We continue to use the notations in the proof of Lemma~\ref{L:vertical-direction}. 
We want to show that there is a constant $A_1$ with  
\begin{equation}\label{E:translation-derivative-bound}
|\ga \frac{d}{d\ga}\ell_\ga| \leq A_1.
\end{equation}
Note that by the pre-compactness of the class of univalent maps on the disk, $f_\ga\co{\gi}(-4/27)$ is uniformly away from 
$0$. 
Therefore, we have 
\begin{align*}
|\frac{d v''_\ga}{d\ga}|
&\leq\Big|\frac{d}{d\ga}\big(\frac{1}{2\pi \ga}
\log (1-\frac{1-e^{2\pi\ga\B{i}}}{f_\ga\co{\gi}(-4/27)u_{f_\ga}(0)})\big)\Big| \leq A_2.
\end{align*}

Let $H_\ga:=H_{f_\ga}$ defined in the proof of Lemma~\ref{L:vertical-direction}, $G_\ga:=G_{f_\ga}$, and 
$\gm_\ga:= \gm_{H_\ga}=\gm_{G_\ga}$.
Fix $\ga,\gb\in (0,\ga_*]$ close together, and define $\gO:= G_\gb \circ G_\ga^{-1}$.
By a similar calculations carried out in the previous proof, using Lemma~\ref{L:horizontal-estimate-lift} here, we 
can prove that $|d\gm_\ga/d \ga|$ is either less than $A_3|\gt_\ga(H_\ga(\gz))|$ or  $A_3$, depending on whether $\im H_\ga(\gz)>0$, 
or $\im H_\ga(\gz)<0$, respectively.
For $\gb> \ga$, $|\gt_\gb(w)|$ decays faster than $|\gt_\ga(w)|$ as $\im w\to +\infty$. 
To simplify the presentation from here, let us assume that $\gb> \ga$ without loss of generality (otherwise, 
exchange $\ga$ and $\gb$ in an obvious fashion).
Hence, by the above statement,  
\begin{equation}\label{E:dilatation-of-Omega}
\begin{gathered}
|\gm_{\gO}(\gx)|\leq A_4|\gb-\ga| |\gt_\ga(L_\ga(\gx))|, \text{ when } \im L_\ga(\gx)>0,     \\
|\gm_{\gO}(\gx)|\leq A_4|\gb-\ga|, \text{ when } \im L_\ga(\gx)<0.
\end{gathered}
\end{equation}

As stated in the previous proof, $|\partial_\gx \gO|$ is uniformly bounded depending only on the class $\IS_0$. 
Hence, by similar calculations (with the same $X^+$ and $X^-$ defined in the previous proof),  
\begin{equation}
\begin{aligned}
\Big |\iint_{\CC/\D{Z}}&\partial_{\bar{\gx}}\gO(\gx)
\frac{\tan (\pi\gx)+\B{i}}{2\tan (\pi\gx)}\, d\bar{\gx}d\gx \Big| \\
&\leq A_5\iint_{\CC/\D{Z}} |\gm_\gO(\gx)| \Big|\frac{\tan (\pi\gx)+\B{i}}{2\tan (\pi\gx)}\Big|\, d\bar{\gx}d\gx\\
&\leq A_5 A_4|\gb-\ga|\Big (
\iint_{X^+} |\gt_\ga(L_\ga(\gx))| \Big|\frac{\tan (\pi\gx)+\B{i}}{2\tan (\pi\gx)}\Big|\, d\bar{\gx}d\gx  \\ 
&\qq \qq\qq\qq
+ \iint_{X^-} \Big|\frac{\tan (\pi\gx)+\B{i}}{2\tan (\pi\gx)}\Big|\, d\bar{\gx}d\gx  
\Big )\\
&\leq A_6 |\gb-\ga| \log \ga(f)^{-1}+ A_6 |\gb-\ga|
\end{aligned}
\end{equation}
Combining with equations \eqref{E:Green-on-boundary} and \eqref{E:Green-on-domain},
\begin{align*}
\big|\frac{d\ell_\ga}{d\ga}(\ga(f))\big|= \Big|\lim _{\gb\to \ga(f)} \frac{v_{\ga(f)}''-v_\gb''
+\lim_{\gx\to \infty}(\gO(\gx)-\gx)}{\ga(f)-\gb}\Big|\\
\leq A_2+ A_6+A_6\log \ga(f)^{-1},
\end{align*}
proving the desired Inequality~\eqref{E:translation-derivative-bound}. 

Combining the two inequalities \eqref{E:log-translation} and \eqref{E:translation-derivative-bound} we conclude 
that 
\[|\frac{d (\ga \ell_\ga)}{d\ga}|\leq A_7(1-\log \ga),\]
which can be integrated to produce the inequality in the lemma.
\end{proof}

Here, we prove a technical lemma that will be used in the next section. 
Recall the number of pre-images $k_f$ needed to take in the definition of the renormalization, and the constant 
$\B{k}$ defined in section~\ref{SS:Inou-Shishikura}.
\begin{lem}\label{L:turning}
There exists $K_7\in \D{Z}$ such that for all $f \in \QIS$, $k_f\leq K_7$. 
\end{lem}
\begin{proof}
From the first part of Lemma~\ref{L:equi-Lipschitz}, and the definition of $L_f$ we conclude that the image of 
every vertical line under $\gF_f^{-1}$ is a curve landing at zero with a well-defined asymptotic slope at zero. 
Fix $w\in \Csh_f$ close to zero, and note that there exists a unique inverse orbit $w$, $f^{-1}(w)$, 
$\dots,f^{-j}(w)$, of minimal cardinality that is contained in a neighborhood of zero and $f^{-j}(w)\in \C{P}_f$. 
Comparing with the rotation of angle $\ga$, $j\in \{\B{k}+1, \B{k}+2\}$. 
By the definition of $k_f$, $k_f\geq  \B{k}+2$.

The map $\gF_f\circ  f\co{k_f}  \circ \gF_f^{-1}: \gF_f(S_f)\to \D{C}$ projects via $\ex$ to the map $\C{R}''(f)$. 
A suitable restriction of this map to a smaller domain is $\C{R}(f)$. 
By Theorem~\ref{Ino-Shi2}, $\C{R} (f)$ is of the form $P\circ \psi^{-1}(e^{2\pi\ga \B{i}} \cdot)$ with 
$\gy: V\to \D{C}$ extending univalently over the larger domain $U$, which compactly contains $V$. 
Indeed, it follows from the work of Inou-Shishikura that the domain of $\C{R}''(f)$ is well contained in $\gy(U)$, 
i.e.\ the conformal modulus of $\gy(U)\setminus \Dom (\C{R}''(f))$ is uniformly away from zero. 
This is because according to \cite[Section~5.N]{IS06}, Theorem~\ref{Ino-Shi2} still holds even if one starts with the 
larger sets $\C{C}_f=\gF_f^{-1}([1/2,3/2]\times [-\gh, \gh])$ and 
$\Csh_f=\gF_f^{-1}([1/2,3/2]\times [\gh, \infty])$, for $\gh\leq 13$, in the definition of renormalization.
Say with $\gh=4$, we obtain an extension of $\C{R}''(f)$ whose image is a disk of radius $e^{2\pi \gh}$. 
The set $P^{-1}(B(0,e^{2\pi \gh})\setminus B(0,e^{4\pi}))$ provides a definite annulus on which $\gy$ has 
univalent extension.
   
By the above argument and Koebe distortion Theorem, the set of maps $\C{R}''(f)$, $f\in \QIS$, forms a compact class of maps (in the uniform convergence topology).  
Therefore, there exists a constant $b\in \D{R}$ (independent of $f$) such that every lift of $\C{R}''(f)$ under $\ex$ 
deviates from a translation by a constant, at most by $b$. 
This implies that $k_f$ is bounded by $\B{k}+3+b$. 
\end{proof}

\subsection{Univalent maps with small Schwarzian derivative}\label{SS:Schwarzian-derivatives}
In this subsection we prove Proposition~\ref{P:contracting}.
Given $f,g\in \QIS$, set $h(z):=f(e^{2\pi (\ga(g)-\ga(f))\B{i}}z)$ and note that by virtue of  Theorem~\ref{Ino-Shi2} 
we only need to show that
\[\Td(\pi\circ \C{R}(f), \pi\circ \C{R}(h))\leq K_2 |\ga(f)-\ga(h)|
,\] 
for some constant $K_2$ depending only on $\QIS$. 

Within this subsection all the constants $A_1, A_2, A_3, \dots$ depend only on $\QIS$, unless otherwise stated. 
For our convenience, we introduce the following notations.

Given $f_0\in \IS_0\u\{Q_0\}$ and $\ga\in (0,\ga_*]$, define 
\[f_\ga(z):=f_0(e^{2\pi \ga \B{i}}z).\] 
We write the renormalized map in the form
\begin{equation}\label{E:two-representation}
\begin{gathered}
\C{R}(f_\ga)(w)=P\circ \gy_\ga^{-1}(e^{2\pi\frac{1}{\ga}\B{i}}w), \text{ with } \gy_\ga:V\to \D{C}.
\end{gathered}
\end{equation}
The map $\gy_\ga:V\to V_\ga:=\gy_\ga(V)$ is univalent, $\gy_\ga(0)=0$, and $\gy_\ga'(0)=1$. 
Recall that by Theorem~\ref{Ino-Shi2} every $\gy_\ga$ has univalent extension onto the larger domain $U$.
The main step toward proving the above desired inequality is the following proposition that we shall prove here 
in several steps.

\begin{propo}\label{P:psi-derivative}
There exist a constant $A_1$ and a Jordan domain $W$, depending only on $\QIS$, with $V \Subset W \Subset U$, such that 
\[\forall z\in W,  \forall \ga\in (0,\ga_*], |\frac{\partial \gy_\ga}{\partial \ga}(z)|\leq A_1.\]
\end{propo}

Let $\gF_\ga:\C{P}_\ga\to \D{C}$ denote the normalized Fatou coordinate of $f_\ga$ that can be decomposed as in 
$\gF_\ga^{-1}= \gt_\ga\circ L_\ga$, where $\gt_\ga$ is the covering map defined in Equation~\eqref{E:covering} 
for $f_\ga$. 
The lift of $f_\ga$ under $\gt_\ga$ (defined just before Lemma~\ref{L:equi-Lipschitz}) is denoted by $F_\ga$. 
Let $\C{C}_\ga$ and $\Csh_\ga$ denote the sets defined in Equation~\eqref{E:sector-def} for the map $f_\ga$. 
Also, let $k_\ga:=k_{f_\ga}$ denote the smallest positive integer for which
$S_\ga:=\C{C}^{-k_\ga}_\ga\u(\Csh_\ga)^{-k_\ga}$ is contained in 
\[\{z\in \C{P}_\ga\mid  \re \gF_\ga(z)\in (1/2, \ga^{-1} - \B{k}-1/2)\}.\]

Consider
\[E_\ga:\gF_\ga (S_\ga)\to \gF_\ga(\C{C}_\ga\u\Csh_\ga), \;
E_\ga(w):=\gF_\ga \circ f_\ga\co{k_\ga} \circ \gF_\ga^{-1}(w).\]
By the definition of renormalization, $E_\ga$ projects under $\ex$ to the map $\C{R}''(f_\ga)$.
Recall that a suitable restriction of $\C{R}''(f_\ga)$ to a smaller domain is $\C{R}(f_\ga)$.
Also, let
\begin{equation}\label{E:widetilde(E)-difficult}
\wt{E}_\ga:L_\ga(\gF_\ga (S_\ga))   \to L_\ga(\gF_\ga(\C{C}_\ga\u\Csh_\ga)), \;
\wt{E}_\ga:= L_\ga \circ E_\ga \circ  L_\ga^{-1}. 
\end{equation}

\begin{lem}\label{L:buffer}
There exists a positive constant $A_2<1/4$ such that $\wt{E}_\ga$ has holomophic extension onto 
$X_\ga:= B(L_\ga\circ \gF_\ga (S_\ga), A_2)$, and we have
\[\wt{E}_\ga=F_\ga\co{k_\ga}-\frac{1}{\ga}. \]
Moreover, $|\partial F_\ga\co{k_\ga}/\partial \ga|$ is uniformly bounded on $X_\ga$, with a bound depending 
only on $\QIS$.    
\end{lem}
\begin{proof}
By the definition of $F_\ga$, we have the formula at points with large imaginary part. 

First we claim that there exists an $A_3>0$ such that $f_\ga\co{i}$, for $i=1, \dots, k_\ga$, is defined 
on $B(S_\ga, A_3)$. 
Obviously, these iterates are defined on a ball of uniform size $B(0, A_4)$.
Away from zero, we use a result of \cite[Section~5.N]{IS06}. 
It is proved that Theorem~\ref{Ino-Shi2} still holds even if one starts with the larger sets  
$\C{C}_{\ga, \gh}:=\gF_\ga^{-1}([1/2,3/2]\times [-\gh, \gh])$ and 
$\Csh_{\ga, \gh}:=\gF_\ga^{-1}([1/2,3/2]\times [\gh, \infty])$, 
for $\gh\leq 13$, in the definition of renormalization.
For example, using $\gh=4$, we obtain the corresponding set $\hat{S}_\ga$ on which $f_\ga\co{k_\ga}$ is defined.
Now, by the pre-compactness of $\QIS$, the continuous dependence of $\gF_\ga$ on $\ga$, and the uniform 
bound on $k_\ga$ in Lemma~\ref{L:turning}, 
$B(0, A_4)\u \hat{S}_\ga\u f_\ga(\hat{S}_\ga) \u \gF_\ga^{-1} (\gF_\ga (\hat{S}_\ga)-1)$ contains a 
neighborhood (of a uniform size) of $S_\ga$ on which the iterates $f\co{i}$, for $i=1, \dots, k_\ga$, are defined. 

By the pre-compactness of the class $\QIS$, there is $A_3'>0$ such that 
$\gt_\ga(\D{C}\setminus B(\D{Z}/\ga, A_3'))$ contains  $B(S_\ga, A_3)\setminus \{0\}$, for all $\ga$.  
Now, $|\gt'_\ga|$ is uniformly bounded on $\gt_\ga(\D{C}\setminus B(\D{Z}/\ga, A_3'))$. 
Hence, there exists $A_2>0$ such that $\gt_\ga(B(L_\ga\circ \gF_\ga(S_\ga), A_2))$ 
is contained in $B(S_\ga, A_3)$. 
This implies that there is a unique extension of $F_\ga\co{k_\ga}$ onto $X_\ga:=B(L_\ga(\gF_\ga (S_\ga)), A_2)$. 
Therefore, $\wt{E}_\ga$ has holomorphic extension onto $X_\ga$, and the above equation is meaningful and holds. 
\end{proof}

By the uniform bounds on $|L_\ga'|$ in Lemma~\ref{L:bound-on-L'}, there exists a neighborhood 
$Z_\ga:=B(\gF_\ga (S_\ga), A_4)$ such that $L_\ga(Z_\ga)\ci X_\ga$.
This implies that $E_\ga$ is defined on $Z_\ga$. 
  
Since each $\gy_\ga$ has univalent extension onto $U$, by Koebe distortion Theorem, 
there exists a constant $A_5'$ depending only on $V$ and the conformal modulus of $U\setminus V$ such that 
$\partial (\gy_\ga(V))\subseteq B(0, A_5')\setminus B(0,1/A_5')$. 
Hence,
\[\forall w\in \ex^{-1}(\partial V_\ga \cdot e^{-2\pi \frac{1}{\ga}\B{i}}),\; |\im w| \leq A_5.\]  
On the set $\{w\in \D{C}; |\im w|\leq A_5+1\}$, $|\ex'|$ is uniformly bounded from above and below. 
Therefore, the above argument implies that there exists a domain $W  \supset V$ with $\partial W\ci U\setminus V$,
such that 
\[\gy_\ga(W)\cdot e^{-2\pi \frac{1}{\ga}\B{i}}\ci  \ex(Z_\ga).\]
    
Fix $\ga\in (0,\ga_*]$.
We may redefine $k_\gb$ so that $k_\gb=k_\ga$, for $\gb\in (0,\ga_*]$ sufficiently close to $\ga$. 
This will guarantee that $\wt{E}_\gb$ depends continuously on $\gb$ near $\ga$.   
To see the claim, recall that in the definition of renormalization, $k_\ga$ is chosen so that 
$\re \gF_\ga(S_\ga)\ci (1/2, 1/\ga-k-1/2)$. 
By the continuity of $\gF_\gb$ in terms of $\gb$, for $\gb\in (0,\ga_*]$ sufficiently close to $\ga$, we have 
$|k_\ga-k_\gb|\leq 1$. 
On the other hand, one may define the renormalization using the iterate 
$f_\gb^{k_\gb+1}:\gF_\gb^{-1}(\gF_\gb(S_\gb)-1) \to \D{C}$. 
It should be clear from the definition that both $f_\gb^{k_\gb}: S_\gb\to \D{C}$ and 
$f_\gb^{k_\gb+1}:\gF_\gb^{-1}(\gF_\gb(S_\gb)-1) \to \D{C}$ result in the same map $\C{R}''(f)$. 
Hence, we may choose the local maximum of $k_\gb$ near $\ga$, and assume that the renormalization is defined 
using that integer for $\gb$ close to $\ga$.
 
Let $T_b(z):=z+b$, for $b\in \D{C}$, and recall the constant $\B{k}$ defined in Section~\ref{SS:Inou-Shishikura}. 
\begin{lem}\label{L:partials-of-L}
$\forall A_6'$, $\exists A_6$ such that at all $\gb\in (0,\ga_*]$ and 
\begin{itemize}
\item[1)] $\forall w\in B(0, A_6')$ with $\re w>1/4$, we have 
\[\Big |\frac{\partial L_\gb}{\partial\gb}(w)\Big |\leq A_6,\; \;
\Big |\frac{\partial L_\gb^{-1}}{\partial\gb}(L_\gb^{-1}(w))\Big |\leq A_6;\] 
\item[2)] $\forall w\in B(0, A_6')$ with $\re w<-\B{k}-1/4$, we have $|L_\gb(w+1/\gb)-1/\gb|\leq A_6$ and  
\[\Big|\frac{\partial (T_{-1/\gb}\circ L_\gb \circ T_{1/\gb})}{\partial\gb}(w)\Big | \leq A_6.\]
\end{itemize}
\end{lem}

\begin{proof}
{\em Part 1):}
The idea of the proof is similar to a step in the proof of Lemma~\ref{L:horizontal-direction}, which uses 
the model map for the linearizing coordinate.  
Recall the model map $H_\gb:= H_{f_\gb}$ introduced in the proof of Lemma~\ref{L:vertical-direction}, and 
the decomposition $L_\gb=H_\gb\circ G_\gb^{-1}$.
By virtue of the uniform bounds on $|\partial v''_\gb/\partial \gb|$ and $|\partial F_\gb/\partial \gb|$ 
obtained in Lemmas~\ref{L:horizontal-direction} and \ref{L:horizontal-estimate-lift}, respectively, 
$|\partial H_\gb/\partial \gb|$ and $|\partial H_\gb^{-1}/\partial\gb|$ are uniformly bounded over 
the domain and the image of $H_\gb$, respectively. 

Recall that the norm of the complex dilation of the map $H_\gb$, $|\gm_{f_\gb}|$, 
is uniformly bounded from above by a constant $<1$, independent of $\gb$. 
Thus, the norm of the complex dilation of $G_\gb$, $|\gm(G_\gb)|= |\gm_{f_\gb}|$, 
is uniformly bounded from above by a constant $<1$.   
We have the normalization $G_\gb([0])=[0]$, for all $\gb$.
The norms of the complex dilatations of $G_\gb\circ G_\ga^{-1}$ and $G_\gb^{-1}\circ G_\ga$ 
are uniformly bounded from above by a universal constant times $|\gb-\ga|$, 
as in Equation~\eqref{E:dilatation-of-Omega}. 
By the classical results on the dependence of the solution of the Beltrami equation on the Beltrami coefficient, see \cite[Section 5.1]{AhBe60}, for $w$ on a given compact part of the cylinder, 
$|G_\gb\circ G_\ga^{-1}(w)-w|$ and $|G_\gb^{-1}\circ G_\ga(w)-w|$ are bounded from above 
by a uniform constant times $|\gb-\ga|$.
(The uniform constant depends only on the compact set.) 

Combining the above bounds, we have 
\begin{align*}
|L_\ga(w)-&L_\gb(w)| \\
&\leq|H_\ga\circ G_\ga^{-1}(w)-H_\ga\circ G_\gb^{-1}(w)|+|H_\ga\circ G_\gb^{-1}(w)-H_\gb\circ G_\gb^{-1}(w)|\\
&\leq \sup_{\ga, z} |\mathrm{D} H_\ga(z)| \cdot |G_\ga^{-1}(w)-G_\gb^{-1}(w)|+ |H_\ga(z)-H_\gb(z)|\\
& \leq \sup_{\ga, z} |\mathrm{D} H_\ga(z)| \cdot \sup_{\ga,w} |\mathrm{D} G_\ga^{-1}(w)| \cdot 
|w- G_\ga\circ G_\gb^{-1}(w)|+ |H_\ga(z)-H_\gb(z)|\\
&\leq A_7 |\ga-\gb|. 
\end{align*}
Similarly, 
\begin{align*}
|L_\ga^{-1}(L_\gb(w))-L_\gb^{-1}(L_\gb(w))|&=|L_\ga^{-1}(L_\gb(w))-w|\\
&\leq \sup_{z} |\mathrm{D} L_\ga^{-1}(L_\ga(z))|\cdot |L_\gb(w)-L_\ga(w)| \\
&\leq  K_{12} \cdot A_7 |\ga-\gb|. 
\end{align*}
\medskip

{\em Part 2):}\footnote{We wish to compare $L_\gb$ to a model map near the right end of the domain,
but, the issue here is that near $1/\gb$ we do not \textit{a priori} have a well behaving normalization 
for a model map, like $H_\gb(0)=\cp_\gb$.} 
Consider the conformal map $\widehat{L}_\gb:=T_{-1/\gb}\circ L_\gb \circ T_{1/\gb}$ defined on the strip 
$\re w\in [-1/\ga, -\B{k}]$.
Recall that $\Dom F_\gb \supseteq \D{C}\setminus B(\D{Z}/\gb, K_9)$ and $F_\gb$ is univalent on this complement, 
proved in Lemma~\ref{L:basic-estimates-lift}.
Since $F_\gb \circ T_{1/\gb}=T_{1/\gb} \circ F_\gb$ and $L_\gb\circ T_1= F_\gb\circ L_\gb$, 
we have $\widehat{L}_\gb\circ T_1= F_\gb\circ \widehat{L}_\gb$. 
Moreover, $\widehat{L}_\gb(-1/\gb)= \cp_\gb-1/\gb$, where $\cp_\gb$ is the critical point of $F_\gb$ with 
$L_\gb(0)= \cp_\gb$.
First we introduce a model map $\widehat{H}_\gb: [0,1]\times \D{R} \to \Dom F_\gb$ 
for $\widehat{L}_\gb$, that is appropriately normalized at infinity.  

Recall that $G_\gb$ has uniformly bounded dilatation and is normalized at $0$. 
It follows from basic properties of quasi-conformal mappings that $|G_\gb(z)|\geq h'$ for every $z$ with 
$|z|\geq h$, where $h'$ depends only on $h$ and the dilatation of the map, and $h'\to +\infty$ as $h\to +\infty$. 
This implies the same property for the map $L_\gb$. 
Combining with the estimates in Lemma~\ref{L:basic-estimates-lift}, there is $t_0\in \D{R}$ and 
an integer $\gi$, both depending only on $\QIS$, such that for all $\gb\in (0,\ga_*)$ we have 
\[\forall x\in \{t_0\B{i}, 1+t_0\B{i}, -t_0\B{i},1-t_0\B{i}\}, \re F_\gb^{-\gi}(L_\gb(x)) \leq -K_9,\] 
and for all $w$ with $\re w\in [0,1]$ and $|\im w|\geq t_0$, $F_\gb^{-\gi}(L_\gb(w))$ is defined.

Fix $\gb\in (0,\ga_*)$. 
For $|t|\geq t_0+1$, define $\widehat{H}_\gb(s,t):= F_\gb^{-\gi}\circ L_\gb(s+t\B{i})$. 
For $|t|\leq t_0$, let 
\[\gga(t):=(\frac{t_0-t}{2t_0})F_\gb^{-\gi}(L_\gb(-t_0\B{i}))+(\frac{t+t_0}{2t_0})F_\gb^{-\gi}(L_\gb(t_0\B{i})),\]
and then define  
$\widehat{H}_\gb(s,t):=(1-s)\gga(t)+ s F_\gb (\gga(t)).$

On the two remaining squares $[0,1]\times [t_0, t_0+1]$ and $[0,1]\times [-t_0-1, -t_0]$ we quasi-conformally 
interpolate these maps. 
For the top square, consider the curves 
\begin{gather*}
\gh_1(s):=\widehat{H}_\gb(s, t_0+1), \tfor s\in [0,1]\\
\gh_2(s):=(1-s)\widehat{H}_\gb(0,t_0)+ s \widehat{H}_\gb(1, t_0+1), \tfor s\in [0,1].\\
\gh_3(s):=(1-s) \gh_2(0) + s \gh_1(0), \tfor s\in [0,1] 
\end{gather*}
By Lemmas~\ref{L:basic-estimates-lift} and \ref{L:bound-on-L'}, for $t_0$ large enough, 
depending only on $\QIS$, the curves $\gh_1$ and $\widehat{H}_\gb([0,1], t_0)$ are nearly horizontal, 
and $\gh_2$ has slope near one.  
In particular, for big enough $t_0$, $\gh_1$ and $\gh_2$ intersect only at their end points $\gh_1(1)=\gh_2(1)$, 
and $\gh_2$ and $\widehat{H}_\gb([0,1], t_0)$ intersect only at their start point $\gh_2(0)=\widehat{H}_\gb(0,t_0)$.

For $s\in [0,1]$ and $t\in (t_0+s, t_0+1)$, let 
\[\widehat{H}_\gb(s,t):= (\frac{t-t_0-s}{1-s})\gh_1(s) + \frac{t_0-t+1}{1-s}\gh_2(s).\]  
For $t\in [t_0, t_0+1]$, $s\in [t-t_0, 1]$, let
\[\widehat{H}_\gb(s,t):=(\frac{s-1}{t-t_0-1})\gh_2(t-t_0) + \frac{t-t_0-s}{t-t_0-1} F_\gb(\gh_3(t-t_0)).\] 
One can extend the map onto the lower square in a similar fashion. 
Note that the four points $F_\gb^{-\gi}(L_\gb(x))$, for $x\in\{-t_0\B{i},1-t_0\B{i},-(t_0+1)\B{i},1-(t_0+1)\B{i}\}$
tend to the vertices of a parallelogram as $t_0$ tends to $+\infty$. 

By definition, $\widehat{H}_\gb$ is a quasi-conformal homeomorphism onto its image and satisfies 
$\widehat{H}_\gb(1, t)= F_\gb\circ \widehat{H}_\gb(0,t)$, for all $t\in \D{R}$.  
Moreover, by Lemma~\ref{L:horizontal-estimate-lift} and the uniform bound in Part $1$, on a given compact part 
of the cylinder $|\partial \widehat{H}_\gb/\partial \gb|$ is uniformly bounded above, and moreover, by 
Lemma~\ref{L:basic-estimates-lift} and \ref{L:bound-on-L'}, $|\partial \widehat{H}_\gb/\partial z|$ is uniformly bounded over the whole cylinder. 
Indeed, one can obtain simple formulas for the partial derivatives of $\widehat{H}_\gb$ in terms of 
linear combination of the partial derivatives of $F_\gb$ and $L_\gb$, similar to the ones in 
Equation~\eqref{E:partials-of-H}. 

Define $\widehat{G}_\gb(s+t\B{i}):=\widehat{L}_\gb^{-1} \circ \widehat{H}_\gb(s,t)$. 
We have, $\widehat{G}_\gb(1+t\B{i})=1+\widehat{G}_\gb(t\B{i})$, that is,   
$\widehat{G}_\gb$ induces a well-defined map of the cylinder. 
The complex dilatation of $\widehat{G}_\gb$ is zero outside $[0,1]\times [-t_0-1,t_0+1]$.
Furthermore, as both $L_\gb$ and $\widehat{L}_\gb$ have the same asymptotic translation constant at infinity, we have $|\widehat{G}_\gb(z)-z|\to 0$ modulo $\D{Z}$, as $|\im z|$ tends to $+\infty$ on the cylinder. 
Hence, $\widehat{G}_\gb \circ \widehat{G}_\ga^{-1}$ is asymptotically identity near the top end of the cylinder, 
with complex dilation bounded by a uniform constant times $|\gb-\ga|$. 
This implies that on a given compact part of the cylinder, $|\widehat{G}_\gb\circ \widehat{G}_\ga^{-1}(w)-w|$ 
is bounded by a  uniform constant times $|\gb-\ga|$. 

Now the same steps as in Part $1$ prove the desired bound in the lemma at points $w$ with $\widehat{L}_\gb(w)$ 
in the image of $\widehat{H}_\gb$. 
For the points $w$ given in the lemma, one may iterate $\widehat{L}_\gb(w)$ forward or backward 
by $F_\gb$ (uniformly) bounded number of times to fall into the image of $\widehat{H}_\gb$. 
Then, the bound at $w$ follows from the uniform bounds on the partial derivatives of $F_\gb$ in 
Lemma~\ref{L:horizontal-estimate-lift}.   

The uniform bound on $|\widehat{L}_\gb(w)|$ follows from the ones for $\widehat{H}_\gb$ and 
$\widehat{G}_\gb$. 
\end{proof}

\begin{lem}\label{L:partial-of-parabolic-renormalization-lift}
$\forall A_8'$, $\exists A_8$ such that $\forall \gb\in (0,\ga_*]$ and $\forall w\in Z_\gb-1/\gb$ with 
$|\im w| \leq A_8'$
\[|\frac{\partial ({E}_\gb\circ T_{1/\gb})}{\partial \gb}(w)|\leq A_8.\] 
\end{lem}

\begin{proof}
Given $A_8'$, we may choose $A_6'$ so that for every $w\in Z_\gb$ with $|\im w| \leq A_8'$ we have 
$w\in B(1/\gb, A_6')$ and $F_\gb\co{k_\gb}(L_\gb(w))\in B(1/\gb, A_6')$. 
The latter follows from the uniform bound on $\widehat{L}_\gb(w-1/\gb)$ in Lemma~\ref{L:partials-of-L}-2, and 
Lemma~\ref{L:basic-estimates-lift}. 
Then, at $w\in Z_\gb-1/\gb$, using $F_\gb \circ T_{1/\gb} =T_{1/\gb}\circ F_\gb$, 
\begin{align*}
\frac{\partial ({E}_\gb\circ T_{1/\gb})}{\partial \gb}(w)&=
\frac{\partial (L_\gb^{-1}\circ\wt{E}_\gb \circ L_\gb \circ T_{1/\gb})}{\partial \gb}(w)\\
&=\frac{\partial (L_\gb^{-1} \circ F_\gb\co{k_\gb} \circ (T_{-1/\gb} \circ L_\gb \circ T_{1/\gb}))}{\partial \gb}(w).
\end{align*}
Given $w$ as in the lemma, we have 
\[T_{-1/\gb} \circ L_\gb \circ T_{1/\gb}(w)= \widehat{L}_\gb(w)\in B(0, A_6),  
F_\gb \co{k_\gb}(T_{-1/\gb} \circ L_\gb \circ T_{1/\gb}(w))\in B(0, A_6').\]
Now the uniform bound in the lemma follows from the bounds in Lemmas~\ref{L:horizontal-estimate-lift},
\ref{L:bound-on-L'}, \ref{L:buffer}, and \ref{L:partials-of-L}. 
\end{proof}

\begin{proof}[Proof of Proposition~\ref{P:psi-derivative}]
Projecting ${E}_\ga\circ T_{1/\ga}$ via $\ex$ onto a neighborhood of $0$, we obtain $P \circ \gy_\ga^{-1}$.    
Since $|\ex'(w)|$ is uniformly bounded from above and below on $\{w\in \D{C}; |\im w |\leq A_8\}$, 
\[\forall z\in \gy_\ga(W\setminus V)\cdot e^{-2\pi \frac{1}{\ga}\B{i}}, \; 
|\frac{\partial  ( P\circ \gy_\ga^{-1})}{\partial\ga}(z)| \leq  A_9\]
On the other hand, $|P'|$ and $|\gy_\gb'|$ are uniformly bounded from above and away from zero 
on $W\setminus V$. 
Hence, differentiating with respect to $\ga$, we obtain the estimate in the proposition at points in 
$W\setminus V$. 
Since $\partial \gy_\ga/\partial \ga$ is holomorphic in $z$, for each fixed $\ga$, by the maximum 
principle, the uniform bound holds on $W$.
\end{proof}

Fix $\ga\in (0,\ga^*]$ and for $\gb\in (0,\ga^*]$ define the univalent map 
\[\gO_\gb:=\gy_\gb\circ \gy_\ga^{-1}: V_\ga\to V_\gb.\] 
The \textit{Schwarzian derivative} of $\gO_\gb$ on $V_\ga$ is defined as
\[\SD\gO_\gb:=\Big(\frac{\gO_\gb''}{\gO_\gb'}\Big)' -\frac{1}{2}\Big(\frac{\gO_\gb''}{\gO_\gb'}\Big)^2
.\]
Since each $\gO_\gb$ has univalent extension onto $U_\ga:= \gy_\ga(U)$, the above quantity is uniformly 
bounded on $V_\ga$ independently of $\gb$ and $f_0$, by Koebe distortion Theorem. 
The value of this derivative measures the deviation of $\gO_\gb$ from the M\"obius transformations at points 
in $U_\ga$. 
Let $\gh_\ga$ denote the Poincar\'e density of $V_\ga$, that is, $\gh(z) |dz|$ is a complete metric of constant $-1$ 
curvature on $V_\ga$.
The hyperbolic sup-norm of  $\SD\gO_\gb$ on $V_\ga$ is defined, see \cite{Leh87}, as 
\[\| \SD\gO_\gb\|_{V_\ga}:=\sup_{z\in V_\ga} |\SD \gO_\gb(z)|\gh_\ga(z)^{-2}.\]

\begin{lem}\label{L:linear-schwarzian}
There exists a constant $A_9$, depending only on $A_1$ and the domains $W\supset V$,  
such that for all $\ga, \gb\in (0,\ga_*]$ we have 
\[\|\SD \gO_\gb\|_{V_\ga}\leq A_9 |\gb-\ga|.\]
\end{lem}

\begin{proof}
Choose a positive constant $r$ such that for all $z\in V$, $B(z,r)\ci W$.
Then, from the estimate in Lemma~\ref{P:psi-derivative}, and the Cauchy integral formulas for the 
derivatives, on the set $V$
\[\Big|\frac{\partial(\gy_\ga-\gy_\gb)}{\partial z}\Big|, \Big|\frac{\partial^2 (\gy_\ga-\gy_\gb)}{\partial z^2}\Big|, 
\Big|\frac{\partial^3 (\gy_\ga-\gy_\gb)}{\partial z^3}\Big|\leq A_{10}|\ga-\gb|.\]

On the other hand, from the definition of Schwarzian derivative (or see \cite{Leh87}, Section II.1.3), 
\[\|\SD \gO_\gb\|_{V_\ga}= \|\SD\gy_\ga-\SD\gy_\gb\|_V.\]
Recall that $|\gy_\ga'|$ and $|\gy_\gb'|$ are uniformly bounded away from zero, by Koebe's distortion theorem.
Thus, combining the above estimates one obtains the desired inequality.
\end{proof}

The sets $V_\ga$ are uniformly quasi-disks depending only on the modulus of $U\setminus V$. 
It is a classical result in Teichm\"uller Theory, see \cite{Ahl63} or \cite[Chaper 2, Theorem 4.1]{Leh87}, that 
there exist constants $A_{11}$ and $A_{11}'$, depending only on $V$ and $U$, such that provided $\|\SD\gO_\gb\|_{V_\ga}\leq A_{11}'$,  $\gO_\gb:V_\ga\to V_\gb$ has a quasi-conformal 
extension onto $\D{C}$ with complex dilatation less than $A_{11} \|\SD\gO_\gb\|_{V_\ga}$.
By the definition of the Teichm\"uller distance, it is easy to deduce from this the following corollary 
that finishes the proof of the proposition.
\begin{cor}
There exists a constant $A_{12}$, depending only on $\QIS$, such that for all $\ga, \gb\in (0,\ga_*]$, we have 
\[\Td(\pi\circ \C{R}(f_\ga), \pi \circ \C{R}(f_\gb))\leq A_{12} |\ga-\gb|. \]
\end{cor}
\subsection{Upper bounds on the size of Siegel disks}\label{SS:size-of-Siegels}

\begin{proof}[Proof of Proposition \ref{P:Uniformly-bounded}]
We shall show that there exists a constant $A_{13}$ such that for every $f\in \QIS$ there exists a sequence of points 
$z_n$, $n=0,1,2,\dots$ in the forward orbit of the unique critical value of $f$ satisfying the inequality 
\[\log d(0,z_n)\leq A_{13}+\sum_{i=0}^{n} \gb_{i-1}\log \ga_{i}
.\]  
Assuming this for a moment, if $z$ is an accumulation point of this sequence, 
it must lie outside of the Siegel disk of $f$ and $\log d(0,z)\leq A_{13}-B(\ga(f))$.
By $1/4$-Theorem, we obtain the desired inequality $\log r(f)+B(\ga(f))\leq A_{13}+\log 4$.  

For $n\geq 1$, let $\gz_n:= \lfloor (1/\ga_n-\B{k})/2\rfloor$, where $\lfloor \cdot \rfloor$ denotes the largest integer less 
than $\cdot$, then inductively define the sequence  
$\gz_{n-1}, \gz_{n-2}, \dots, \gz_0$ so that for $i=0,1,2\dots, n-1$, 
\[|\re (\gz_i)-(1/\ga_i-\B{k})/2|\leq 1/2, \tand \ex(\gz_{i})=\Phi_{i+1}^{-1}(\gz_{i+1})
,\]   
where, $\ex$ is defined in Equation \eqref{E:ex} and $\gF_{i}$ is the perturbed Fatou coordinate of the map 
$f_i:=\C{R}^{i}(f)$. 
Then, define $z_n:= \gF_0^{-1}(\gz_0)$. 

By an inductive argument we prove that each $\gF_i^{-1}(\gz_i)$ is in the forward orbit of the critical value of 
$f_i$, $-4/27$. 
For $i=n$, it is clear; $\gF_n(-4/27)=1$, and therefore, by 
the Abel functional equation for $\gF_n$, $\gF_n^{-1}(\gz_n)$ is in the forward orbit of $-4/27$.  
Choose the smallest positive integer $t_n$ with $f_n\co{t_n}(-4/27)=\gF_n^{-1}(\gz_n)$. 

To prove the statement for $i=n-1$, first note that by the definition of the renormalization there are
$x_i\in \gF_{n-1}(S_{f_{n-1}})$ and $ y_i\in \gF_{n-1}(\C{C}_{f_{n-1}}\u \Csh_{f_{n-1}} )$, for $ i=0,1,2,\dots,t_n-1$
such that 
\begin{gather*}
\ex(x_i)= f_n\co{i}(-4/27),\; \ex(y_i)=f_n\co{i+1}(-4/27), \\
\gF_{n-1}\circ f_{n-1}^{k_{n-1}}  \circ \gF_{n-1}^{-1}(x_i)=y_i. 
\end{gather*}
Choose integers $s_i$, so that $x_0=s_0-1$, $y_{i-1}+s_i=x_i$, for $i=0,1,2,\dots, t_n-1$, and 
$\gz_{n-1}=y_{t_n-1}+s_{t_n}$. 
This implies that iterating $-4/27$ by $f_{n-1}$, $s_0+s_1+\dots+s_{t_n}+t_nk_{n-1}$ number of times, we reach 
$\gF_{n-1}^{-1}(\gz_{n-1})$.  
One continues this process to reach level zero.

Next, we estimate the size of $z_n$. 
To do this we first show that there exists a constant $A_{14}$, depending only on 
$\QIS$, such that for all $i=n,n-1,\dots, 0$  
\begin{gather}\label{E:recursive-heights}
\im \gz_{i-1}\geq \ga_i \im \gz_i-\log \ga_i - A_{14}.
\end{gather}
Recall the maps $H_f$ and $G_f$ defined in the proof of Lemma~\ref{L:vertical-direction}, and choose an integer 
$t_i$ such that $\gz_i+t_i\in G_{f_i}([0,1]\times \D{R})$. 
We have $L_{f_i}(\gz_i)=  F_{f_i}\co{(-t_j)}\circ H_{f_i}  \circ G_{f_i}^{-1} (\gz_i+t_i)$. 
By Sublemma~\ref{SL:log-bound}, $\im H_{f_i}  \circ G_{f_i}^{-1} (\gz_i+t_i)$ is uniformly bounded 
from below by $-A_{15}(1-\log \ga_i)$. 
On the other hand, one can conclude from Lemma~\ref{L:basic-estimates-lift}-2 (by adding the errors along 
the orbit) that $|\im (F_{f_i}\co{k} (w)-w)|\leq A_{16}(1-\log \ga_i)$, for every $w$ with 
$\im w\geq -A_{15}(1-\log \ga_i)$ and every positive integer $k\leq 1/\ga_i$. 
This implies that 
\[|L_{f_i}(\gz_i)-\gz_i|\leq A_{17} (1-\log \ga_i).\]
Now, Inequality \eqref{E:recursive-heights} results from looking at the explicit formula $\ex^{-1}\circ \gt_{f_i}$. 
In this calculation that is left to the reader, one uses that the size of $\gs_{f_i}$ is comparable to $\ga_i$, 
$\im \gz_i$ is uniformly bounded from below, and the $\log \ga_i$ in the above formula is multiplied by $\ga_i$ and 
contributes a bounded value.

Recursively putting Equation~\eqref{E:recursive-heights} together,  one obtains the desired inequality on the 
size of $z_n$.
\end{proof}

\begin{proof}[Proof of Proposition~\ref{prop:samelim}]
By the definition of $C(f)$, Proposition~\ref{P:Uniformly-bounded}, Equation~\eqref{E:log-translation-main} 
and Equation~\eqref{E:C(f)-expanded}, for positive $\ga$ we have
\[\lim_{\ga\to 0^+} \gU(z\mapsto f_0(\ea z))=\lim_{\ga\to 0^+} C(z\mapsto f_0(\ea z))=-\log |f''(0)|+\log (4\pi).\] 
When $\ga$ is negative, one considers the conjugate map $s\circ (z\mapsto f_0(\ea z)) \circ s$, where $s(z)=\ol{z}$, 
and obtains a family of the form $z\mapsto g_0(e^{-2\pi \ga\B{i}}z)$. 
Moreover, the class $\IS_0$ is closed under this conjugation, $|f_0''(0)|=|g_0''(0)|$, and the conformal radius of 
the Siegel disk is preserved under this conjugation. 
Hence, the equality for the left limit follows from the above one.  
\end{proof}

%% file: section5.tex
\section{H\"older}\label{sec:arn}

\subsection{Introduction}\label{subsec:map}

We continue to use the Riemannian metric $ds=\bigl|\log|x|\bigr| |dx|$ on the interval $[-1/2,1/2]$. 
The distance with respect to this metric between two points $x,y\in[-1/2,1/2]$ will be denoted $d_{\log}(x,y)$. 
The length of $[-1/2,1/2]$ for this metric is finite (it is equal to $1+\log 2$). 
This distance is H\"older continuous with respect to the Euclidean distance, for any exponent in $(0,1)$ 
(see Section~\ref{sss:bads}).

Recall our notations: $\IS_0$ is (a big enough subset of) the Inou-Shishikura class. 
Maps $f\in\IS_0$ have rotation number $0$ at $0$: $f(0)=0$ and $f'(0)=1$. 
Given a holomorphic map $f$ whose domain contains $0$, fixes $0$, and with $f'(0)\neq 0$, we denote
\[f = f_0 \circ R_{\alpha(f)}\]
with $f_0'(0)=1$ and
\[R_{\theta}(z)=e^{2\pi i \theta} z.\]
With the notation of Sections~\ref{sec:prelim} and~\ref{S:lipschitz}, $f_0 = \pi(f)$.
For $A\subset \R$, $\IS_{A}$ is the set of maps $f$ with $f_0\in\IS_0$ and $\alpha(f)\in A$.
In Section~\ref{sec:prelim}, we also introduced the family of quadratic polynomials $Q_\alpha$ that is closely 
related to $P_\alpha$:
\[P_\alpha = A_\alpha \circ Q_\alpha \circ A_\alpha^{-1},\quad A_\alpha(z) = (e^{2\pi i\alpha})^2\frac{27}{16}z.\]
In particular, with $r(f)$ denoting the conformal radius of the Siegel disk of $f$:
\[r(P_\alpha) = \frac{27}{16} r(Q_\alpha)\]

The renormalization $\cal R(f)$ was defined in Section~\ref{sec:prelim} for $\alpha(f)$ small and positive, 
more precisely, for all $f\in\IS_\alpha \cup \{Q_\alpha\}$ with $\alpha\in(0,\ga^*]$. 
We prefer here to work with $P_\alpha$ instead of $Q_\alpha$, so we define
\[\cal R(P_\alpha) := \cal R(Q_\alpha).\]
To extend the renormalization to maps with $\alpha\in[-\ga^*,0)$ we proceed as follows: let $s(z)=\ov{z}$; 
let us shorten $s\circ f\circ s^{-1}$ as $sfs^{-1}$ then $\alpha(sfs^{-1}) = -\alpha(f)\in (0,\ga^*]$; 
for $f\in\IS_\R$ with $\alpha(f)\in (-1/2,0)\cup(0,1/2)$, let 
\bEA && \cal S(f) := f\text{ if }\alpha(f)\in(0,1/2)\ \text{ and }
\\ && \cal S(f)  :=  sf s^{-1}\text{ if }\alpha(f)\in (-1/2,0).\eEA
The symbol $\cal S$ stands for \emph{saw}, because it acts on the rotation number as the saw map mentioned 
in Section~\ref{subsec:def}.
For $\alpha\in(-1/2,1/2)$, let also 
\[\cal S(Q_{\alpha}) := Q_{|\alpha|},\quad \cal S(P_{\alpha}) := P_{|\alpha|}.\]
Then the composition of $\cal RS$ is an extension of $\cal R$ to maps with $\alpha(f)\in [-\ga^*,0)\cup(0,\ga^*]$.
Let us denote $\cal{RS}_0(f) = (\cal{RS}(f))_0$, so that whenever $\alpha(f)\in [-\ga^*,0)\cup(0,\ga^*]$:
\[\cal{RS}(f) = \cal{RS}_0(f) \circ R_{1/|\alpha(f)|}.\]
Note also that
\[B(-\alpha) = B(\alpha),\quad r(sfs^{-1})=r(f),\quad \Upsilon(sfs^{-1})=\Upsilon(f)\]
thus
\[r(\cal S f)=r(f),\quad \Upsilon(\cal Sf)=\Upsilon(f),\]
and similarly
\[r(P_{-\alpha})=r(P_\alpha), \quad \Upsilon(P_{-\alpha}) = \Upsilon(P_\alpha), \quad 
r(\cal S P_\alpha) = r(P_\alpha), \quad \Upsilon(\cal S P_\alpha) =\Upsilon(P_\alpha).\]

Recall that the set $\HT_N$ is the set of irrational numbers whose modified continued fraction entries are 
all $\geq N$, and that $\IS_N := \IS_{\HT_N}$.
Note that for $N\geq 1/\ga^*$, $\HT_N\subset A$ with $A=[-\ga^*,0)\cup(0,\ga^*]$, and that maps $f$ in 
$\IS_N$ or in $\setof{P_\alpha}{\alpha\in\HT_N}$ are infinitely renormalizable for $\cal{RS}$.

Assume a distance function $d_T(f_0,g_0)$ has been defined\footnote{The subscript $T$ stands for Teich\"uller, 
but our proof is valid for any metric with the required properties stated here.} over $\IS_0$, has been extended 
to $\IS_\R$ as follows:
\[d_{\log}(f,g)=d_{\log}(\alpha(f),\alpha(g))+d_T(f_0,g_0)\]
and that the following holds:
\[f\in\IS_N \mapsto C(f) = \Upsilon(f)-|\alpha(f)|\Upsilon(\cal{RS}(f))\]
is bounded and is Lipschitz continuous over $\IS_N \cap [-1/2,1/2]$ with respect to the metric $d_{\log}$,
\[\alpha\in\HT_N \mapsto C(P_\alpha) = \Upsilon(P_\alpha)-|\alpha|\Upsilon(\cal{RS}(P_\alpha))\]
too. 
Note that 
\[C(\cal S f)=C(f)\text{ and }C(\cal S P_{\alpha}) = C(P_\alpha).\]
Assume that the transversal part of renormalization is Lipschitz continuous over $\IS_N$ with respect to the metric 
$d_{\log}$, and transversally contracting (see below). Assume also that $\Upsilon$ is bounded over $\IS_N$ 
(in \cite{BC04} it has been proved that $\Upsilon$ is bounded over $\Pol_\R$).
This sums up as follows: there exists constants $K_1$, $K_2$, $K_3$ depending only on $\QIS$ such that
\begin{enumerate}\setlength{\itemsep}{.5ex}
\item[.] $\forall f,g\in\cal IS_N$:
\item\label{item:it1} $\Upsilon(f)\leq K_0$
\item\label{item:it2} $|C(f)-C(g)| \leq K_1 d_{\log}(f,g)$
\item\label{item:it3} $|C(f)|\leq K_3$
\item\label{item:it4} If $\alpha(f)$ and $\alpha(g)$ have the same sign, then\\
    $d_T(\cal{RS}_0(f),\cal{RS}_0(g)) \leq  \lambda d_T(f_0,g_0) + K_2 |\alpha(f)-\alpha(g)|$
\item\label{item:it5} $\lambda<1$
\item[.] $\forall \alpha,\beta \in (-1/N,0)\cup(0,1/N)$:
\item\label{item:it6} $|C(P_\alpha)-C(P_\beta)| \leq K_1 d_{\log}(\alpha,\beta)$
\item\label{item:it7} $|C(P_\alpha)|\leq K_3$
\item\label{item:it8} If $\alpha$ and $\beta$ have the same sign, then\\ 
    $d_T(\cal{RS}_0(P_\alpha), \cal{RS}_0(P_\beta)) \leq K_2 |\alpha-\beta|$
\end{enumerate}

Under the above hypotheses, we will prove:

\begin{thm}\label{thm:map}
 $\exists A,B>0$ such that $\forall f,g \in \IS_N$,
\[|\Upsilon(f)-\Upsilon(g)| \leq A\,  | \alpha(f)-\alpha(g)|^{1/2} + B\, d_T(f_0,g_0)\]
and $\forall \alpha,\beta\in \HT_N$,
\[|\Upsilon(P_\alpha)-\Upsilon(P_\beta)| \leq A\, |\alpha(f)-\alpha(g)|^{1/2}.\]
\end{thm}

\begin{rem}
As previously noted, the (uniform) modulus of continuity of
$\alpha\mapsto\Upsilon(\alpha):=\Upsilon(P_\alpha)$ with respect to the rotation number will not be better than 
$1/2$-H\"older. However, the pointwise modulus of continuity can be better at some values. Note that since 
H\"older continuous maps are uniformly continuous, $\Upsilon$ extends to the closure of $\HT_N$, which 
consists of $\HT_N$ plus some rationnals. Let $x$ be one of them (for instance, $x=0$). 
Then with a finite number of renormalizations, it follows from the continuity of $\Upsilon$, and the above 
hypotheses that $\Upsilon$ has a modulus of continuity at $x\in\Q$ of the form: 
\[|\Upsilon(x+\epsilon)-\Upsilon(x)| \leq c_x \big|\epsilon\log|\epsilon|\big|
\]
(in a neighborhood of $x$). This is a lot better than $|\epsilon|^{1/2}$, almost as good as $|\epsilon|$. 
However, $c_x$ depends on $x$.
It cannot be better than $\big|\epsilon\log|\epsilon|\big|$, though: from~\cite{HDR08}, we know that, 
for quadratic polynomials, at rationals $x$ and for some well chosen subsequences $\epsilon_n\tend 0$, 
the function $\Upsilon$ has expansion of the form 
$\Upsilon(x+\epsilon_n) = a_x + b_x\epsilon_n \log |\epsilon_n| + c_x \epsilon_n + o(\epsilon_n)$.
\end{rem}

\begin{rem} Theorem~\ref{thm:map} could be proved under weaker hypotheses. 
Indeed, the term $d_{\log}(\alpha(f),\alpha(g))$ in Hyp.~\ref{item:it2}, $d_{\log}(\alpha,\beta)$ in 
Hyp.~\ref{item:it6}, $|\alpha(f)-\alpha(g)|$ in Hyp.~\ref{item:it4} and $|\alpha-\beta|$ in Hyp.~\ref{item:it8} 
could respectively be replaced by $|\alpha(f)-\alpha(g)|^\delta$, $|\alpha-\beta|^\delta$, $|\alpha(f)-\alpha(g)|^\delta$ 
and $|\alpha-\beta|^\delta$ with $\delta \in (1/2,1)$: the proofs would still work.
\end{rem}

\medskip

Since $\Upsilon(P_\alpha) = \Upsilon(Q_\alpha)+ \log\frac{27}{16}$ and $\cal R(P_\alpha):=\cal R(Q_\alpha)$, 
thus $C(P_\alpha) = C(Q_\alpha) + \log\frac{27}{16}$, Hypotheses~\ref{item:it6}, \ref{item:it7} 
and~\ref{item:it8} are equivalent if $P_\alpha$ is replaced by $Q_\alpha$. 
Hypothesis~\ref{item:it1} has been proved in Proposition~\ref{P:Uniformly-bounded}.
Hypotheses~\ref{item:it3} and ~\ref{item:it7} follow from Proposition~\ref{P:Uniformly-bounded} and the 
definition of $C(f)$.
Hypotheses~\ref{item:it4},~\ref{item:it5}, and ~\ref{item:it8} have been proved in Proposition~\ref{P:contracting}.
Hypotheses~\ref{item:it2} and ~\ref{item:it6} have been proved in Proposition~\ref{P:Lipschitz} in the particular 
case when $\alpha(f)$ and $\alpha(g)$ are both positive.
\begin{proof}[Proof of Hypotheses~\ref{item:it2} and ~\ref{item:it6}]
The case when $\alpha(f)$ and $\alpha(g)$ are both negative follows at once from the case when they are both 
positive since $C(f)=C(\cal S f)$.
For the case when $\alpha(f)$ and $\alpha(g)$ have opposite sign, let us assume $\alpha(f)<0<\alpha(g)$ 
(the other case follows by permuting $f$ and $g$).
For Hypothesis~\ref{item:it6}, it is also simple because $\cal S P_\alpha = P_{-\alpha}$ hence 
\[|C(P_\alpha)-C(P_\beta)| = |C(P_{-\alpha})-C(P_\beta)| \leq K_1 d_{\log}(-\alpha,\beta) < K_1 d_{\log}(\alpha,\beta).\footnote{For maps in $\IS_{N}$, we cannot use this trick because there is also the term $d_T(f_0,g_0)$ to take into account, and $d_T((s f s^{-1})_0,g_0) = d_T(sf_0 s^{-1},g_0)$ could be a lot bigger than $d_T(f_0,g_0)$, even when $\alpha(f)$ and $\alpha(g)$ are small.}\]
For Hypothesis~\ref{item:it2}, recall that we proved in Proposition~\ref{prop:samelim} that $\alpha\mapsto\Upsilon(f_0\circ R_{\alpha})$ has left and right limits at $0$ that are equal. Since $\Upsilon$ is bounded (Hypothesis~\ref{item:it1}) and $C(f):=\Upsilon(f)-|\alpha(f)|\Upsilon(\cal{RS}(f))$, it follows that $\alpha\mapsto C(f_0\circ R_{\alpha})$ also has left and right limits at $0$ that are equal.
Then:
\bEA C(f)-C(g) & = & \phantom{+}\,C(f)-\lim_{\alpha \to 0^-} C(f_0 \circ R_\alpha)  
\\ & & + \lim_{\alpha \to 0^-} C(f_0 \circ R_\alpha) - \lim_{\alpha \to 0^+} C(f_0 \circ R_\alpha)
\\ & & + \lim_{\alpha \to 0^+} C(f_0 \circ R_\alpha) - C(f_0\circ R_{\alpha(g)}) 
\\ & & + \, C(f_0\circ R_{\alpha(g)})-C(g).\eEA
Recall that $d_{\log}(f,g) := d_{\log}(\alpha(f),\alpha(g)) + d_T(f_0,g_0)$. 
By the case when both rotation numbers are negative, we know that for $\alpha\in (\alpha(f),0)$,
\[|C(f)-C(f_0\circ R_\alpha)| \leq K_1 d_{\log}(f, f_0\circ R_\alpha)=0+ K_1 d_{\log}(\alpha(f),\alpha).\] 
Thus by passing to the limit $\alpha \tend 0^-$, we get that the first term in the sum above is $\leq K_1 d_{\log}(\alpha(f),0)$. Similarly, the third is $\leq K_1 d_{\log}(0,\alpha(g))$. 
We have proved that the second term is null. 
For the fourth term, by the case when both rotation numbers are positive, we get 
\[|C(f_0\circ R_{\alpha(g)})-C(g)| \leq K_1 d_{\log}(f_0\circ R_{\alpha(g)},g) = K_1 d_T(f_0,g_0)+0.\] 
All in all:
\[|C(f)-C(g)| \leq K_1\big( d_{\log}(\alpha(f),0)+ d_{\log}(0,\alpha(g)) + d_T(f_0,g_0)\big).\] 
Since $\alpha(f)$ and $\alpha(g)$ have opposite signs, 
\[d_{\log}(\alpha(f),0)+ d_{\log}(0,\alpha(g)) = d_{\log}(\alpha(f),\alpha(g)).\] 
Whence $|C(f)-C(g)| \leq K_1 d_{\log}(f,g)$.
\end{proof}

\subsection{Discussion about rotation number one half}

The Marmi Moussa Yoccoz conjecture concerns a function $\Upsilon$ that is defined on the 1-torus $\R/\Z$ (or it can be viewed as a $1$-periodic function on $\R$).
Recall the number $N\geq 2$ involved in the definition of high type numbers (using modified continued fractions).
\footnote{\label{fn:is}According to discussion with them, the number $N$ provided by Inou and Shishikura is likely to be no less than $20$.
}
Note that for $N=2$, the high type numbers are all the irrationals! In $\R/\Z$, the set of high type numbers is bounded away from $1/2$ as soon as $N\neq 2$.
Extra caution should be taken near rotation number $1/2$. Indeed, it is known that as $\alpha$ varies from $0$ to $1/2$, the fundamental cylinders used in parabolic renormalization tend to different domains as when $\alpha$ varies from $0$ to $-1/2$.
But in $\R/\Z$, $-1/2=1/2$. \emph{This means that there is a discontinuity of cylinder renormalization as one crosses $1/2$.}
To solve this, we believe one may use an extension of one of the two branches of the renormalization operator beyond $1/2$. However, in the present article, we will not worry about that and we will assume $N\geq 3$. We will work in $\R$ instead of $\R/\Z$ and make statements about $\HT_N\cap [-1/2,1/2]$.

\begin{rem}
Crossing $0$ also creates problems and a discontinuity of the renormalization operator. But sizes get multiplied by a factor tending to $0$. This absorbs the discontinuity. We used this in the proof of hypothesis~\ref{item:it2} in Section~\ref{subsec:map}.
\end{rem}

\subsection{Proof, first steps}

The proof of Theorem~\ref{thm:map} is split across the remaining subsections of Section~\ref{sec:arn}.

\bigskip

For $f\in\cal IS_\R$ recall that $\cal S(f) = f$ if $\alpha(f) \in(0,1/2) \bmod \Z$ and $\cal S(f) = s\circ f\circ s$  if $\alpha(f) \in(-1/2,0) \bmod \Z$, where $s(z)=\ov z$.
To shorten the notations in subsequent computations, we will denote
\[f_n = \cal S (\cal{RS})^n(f)\text{ and }g_n = \cal S (\cal{RS})^n(g).\]
They are defined so that the sequence of rotation numbers of $f_n$ coincides with the sequence $\alpha_n(f)$ associated to the modified continued fraction of $\alpha(f)$. Note that $\cal S^2=\cal S$ and thus $\cal S (\cal{RS})^n = (\cal{SRS})^n$. However, it is important to note that $\cal S$ is not continuous with respect to $f$ as $\alpha(f)$ crosses $0$. Hence an inequality like Hypothesis~\ref{item:it4} cannot hold with $\cal{SRS}_0$ replacing $\cal {RS}_0$, even if $\alpha(f)$ and $\alpha(g)$ have the same sign, because the representatives in $[-1/2,1/2)$ of $\alpha(\cal{RS}(f))$ and $\alpha(\cal{RS}(g))$ could still have opposite signs.

Let $f\in\IS_N\cup \Pol_N$. Then for all $n>0$, $f_n\in\IS_N$.
By induction, for all $n\in\N$,
\[\Upsilon(f) = \sum_{k=0}^{n} \beta_{k-1}(f) C(f_k) + \beta_n(f)\Upsilon(f_{n+1}).\]
Let $n\tend+\infty$. Since $\beta_n(f)\tend 0$, and $\Upsilon$ is bounded over $\IS_N$, we get
\[\Upsilon(f) = \sum_{k=0}^{+\infty} \beta_{k-1}(f) C(f_k),\]
which is by itself an interesting formula.

Now assume either $f,g \in \IS_N$ or $f,g\in \Pol_N$: if $\alpha(f)=\alpha(g)$ then denoting $\beta_k=\beta_k(f)=\beta_k(g)$,
\[\Upsilon(f)-\Upsilon(g)  =  \sum_{k=0}^{+\infty} \beta_{k-1} \big(C(f_k) - C(g_k)\big),\]
whence
\begin{equation}\label{eq:upup1} 
\big|\Upsilon(f)-\Upsilon(g)\big|  \leq  K_1\sum_{k=0}^{+\infty}\beta_{k-1}\,d_{T}(f_k,g_k).
\end{equation}
Otherwise, 
let us split the summand the usual way:
\begin{multline*}
\beta_{k-1}(f) C(f_k) - \beta_{k-1}(g) C(g_k)  = 
\beta_{k-1}(f) C(f_k) - \beta_{k-1}(f) C(g_k) \\
+  \beta_{k-1}(f) C(g_k) - \beta_{k-1}(g) C(g_k) .
\end{multline*}
Hence
\begin{multline*}
\big|\beta_{k-1}(f) C(f_k) - \beta_{k-1}(g) C(g_k)\big|  \leq
\beta_{k-1}(f) \times \big| C(f_k) - C(g_k)\big|
\\ + \big|\beta_{k-1}(f) - \beta_{k-1}(g)\big| \times |C(g_k)|.
\end{multline*}
Whence
\begin{equation}\label{eq:upup}
\big| \Upsilon(f)-\Upsilon(g) \big|  \leq \sum_{k=0}^{+\infty}\beta_{k-1}(f)\,|C(f_k)-C(g_k)| + K_3\sum_{k=0}^{+\infty}|\beta_{k-1}(f)-\beta_{k-1}(g)|.
\end{equation}

\subsection{First term}\label{subsec:firstterm}

Here we deal with the following term in Equation~\eqref{eq:upup}
\[\sum_{k=0}^{+\infty}\beta_{k-1}(f)\,|C(f_k)-C(g_k)|.\]
To shorten the notations we let $\alpha = \alpha(f)$, $\alpha' = \alpha(g)$, $\alpha_n$ and $\beta_n$ associated to $\alpha$, as well as $\alpha'_n$ and $\beta'_n$ associated to $\alpha'$ and
\[ d_n=|\alpha_n-\alpha'_n|,\quad d'_n = d_{\log}(\alpha_n,\alpha'_n) 
.\]

Recall the definition given in Section~\ref{subsec:def} of $n$-th generation fundamental intervals and of the map $H_n$ sending bijectively each of these intervals to $(0,1/2)$. For convenience, we copy here a few statements made in that Section.
First, Lemma~\ref{lem:bd} states that there exists $C>0$ such that for all $n\geq 0$, for all $n$-th generation fundamental interval $I$, 
\begin{equation}\label{eq:lembd}
\forall x,y\in I, \quad e^{-C}\leq \left| \frac{H_n'(x)}{H_n'(y)}\right | \leq e^C
.\end{equation}
The union of the fundamental intervals of a given generation is the complement of a countable closed set.
The map $H_n$ is differentiable on each fundamental interval, and
\[H'_n(\alpha) = \pm \frac{1}{\beta_{n-1}^2}
.\]
Last:
\[e^{-C} \leq \frac{\big|I_n(\alpha)\big|}{\beta_{n-1}^2/2}\leq e^{C}
.\]

Let $n_0\geq 0$ be \emph{the first integer so that $\alpha_{n_0}$ and $\alpha'_{n_0}$ belong to different fundamental intervals}.
For some values of $k$ we will use $|C(f_k)-C(g_k)|\leq K_1 d_{\log}(f_k,g_k)$.
Recall that either $f,g\in\IS_N$ or $f,g\in\Pol_N$. In the second case we set $d_T((f)_0,(g)_0)=0$. Then:
\begin{itemize}
\item for $k=0$, $d_{\log}(f_k, g_k) = d_{\log}(f,g) = d_T((f)_0,(g)_0) + d'_0$,
\item for $1\leq k\leq n_0$, $d_T(( f_k)_0,(g_k)_0) \leq  \lambda d_T((f_{k-1})_0,(g_{k-1})_0) + K_2 d_{k-1}$
\item thus by induction: for $0\leq k\leq n_0$,
\[ d_T(( f_k)_0,( g_k)_0) \leq \lambda^k d_T((f)_0,(g)_0) + K_2 \sum_{m=0}^{k-1} \lambda^{k-1-m} d_m
,\]
whence
\[ d_{\log}( f_k, g_k) \leq \lambda^k d_T((f)_0,(g)_0) + K_2 \sum_{m=0}^{k-1} \lambda^{k-1-m} d_m + d'_k 
.\]
\item $d_T((\cal{RS}f_{n_0})_0,(\cal{RS}g_{n_0})_0) \leq  \lambda d_T((f_{n_0})_0,(g_{n_0})_0) + K_2 d_{n_0}$
\end{itemize}
Then
\begin{multline*}
\sum_{k=0}^{n_0}\beta_{k-1}\,d_{\log}(f_k,g_k) \leq  \big(\sum_{k=0}^{n_0} \beta_{k-1}\lambda^k\big) d_T(f_0,g_0)  \\
+ K_2\sum_{k=0}^{n_0} \beta_{k-1} \sum_{m=0}^{k-1} \lambda^{k-1-m}d_m + \sum_{k=0}^{n_0} \beta_{k-1} d'_k.
\end{multline*}
Let us regroup the terms according to $d'_0$, $d'_1$, $d'_2$, \ldots and $d_0$, $d_1$, $d_2$, \ldots:
\begin{multline*}
\sum_{k=0}^{n_0}\beta_{k-1}\,d_{\log}(f_k,g_k) \leq  \Big(\sum_{k=0}^{n_0} \beta_{k-1}\lambda^k\Big)  d_T(f_0,g_0) \\
+ \sum_{j=0}^{n_0}   d'_j \beta_{j-1}
+K_2 \sum_{j=0}^{n_0}  d_j \Big(\sum_{k=j+1}^{n_0} \beta_{k-1} \lambda^{k-1-j}\Big).
\end{multline*}
Note that since $\lambda<1$ and $\beta_{k-1}\leq 1/2^k$, we get $\sum_{k=0}^{n_0} \beta_{k-1}\lambda^k \leq 2$.
Let us bound the following factor in the preceding estimate:
\begin{align*}
 \sum_{k=j+1}^{n_0} \beta_{k-1} \lambda^{k-1-j}
& \leq \sum_{k=j+1}^{+\infty} \beta_{k-1} \lambda^{k-1-j} & \\
& =  \beta_{j-1} \sum_{k=j+1}^{+\infty} \alpha_{j}\cdots \alpha_{k-1} \lambda^{k-1-j} & \\
& \leq\beta_{j-1}\sum_{j=1}^{+\infty} 2^{-j} &\text{(}\lambda\leq 1 \text{ and }a_j\leq 1/2\text{)} \\
& =  \beta_{j-1} &
\end{align*}
Finally we get 
\[ \sum_{k=0}^{n_0}\beta_{k-1}\,|C(f_k)-C(g_k)| \leq 2 K_1\,d_T(f_0,g_0)
+K_1 K_2 \sum_{j=0}^{n_0}\beta_{j-1}d_j
+K_1 \sum_{j=0}^{n_0}\beta_{j-1}d'_j
.\]
For the remaining terms, we will make two cases. Recall that we denote $I_{n_0}(\alpha)$ and $I_{n_0}(\alpha)$ the $n_0$-th generation fundamental intervals $\alpha_{n_0}$ and $\alpha'_{n_0}$ belong to.
\begin{enumerate}
\item[(A)] The fundamental intervals $I_{n_0}(\alpha)$ and $I_{n_0}(\alpha)$ are adjacent and have symbols $(a,+)$, $(a,-)$ as the  last entry of their symbolic representation (see Section~\ref{subsec:def}), with the same value of $a$. 
\item[(B)] the other case.
\end{enumerate}
In case (A) we let $n_1:=n_0+2$, in case (B) $n_1:=n_0+1$.
If necessary, let us permute\footnote{This permutation is done only at this point of the article, so it is safe.} $f$ and $g$, so that $\alpha_{n_0}>\alpha'_{n_0}$.
\begin{lem}
$|\alpha-\alpha'| \geq e^{-2C} \beta_{n_1-1}^2 /12$.
\end{lem}
\begin{proof}
If the two intervals are not adjacent, then we are in case (B), so $n_1=n_0+1$.  Let $I$ be the first generation interval just left to that containing $\alpha_{n_0}$. Then it separates $\alpha_{n_0}$ from $\alpha'_{n_0}$ thus $\alpha$ and $\alpha'$ are separated by the fundamental interval of generation $n_1$ adjacent to that of $\alpha$ and corresponding under $H_n$ to $I$. By the estimates following Equation~\eqref{eq:lembd},
\[ |\alpha-\alpha'|\geq e^{-2C}\beta_{n_1-1}^2/2.\]

If the intervals are adjacent but with symbols $(a,+)$, $(a+1,-)$ we are aslo in case (B). Recall that we assumed $N\geq 3$. There is thus a pair of generation $n_0+2$ intervals $I$ and $I'$ with symbol ending with $(2,+)$ inside $I_{n_0+1}(\alpha)$ and $I_{n_0+1}(\alpha')$ separating $\alpha$ from $\alpha'$. The quotients $|I|/|I_{n_0+1}(\alpha)|$ and $|I'|/|I_{n_0+1}(\alpha')|$ are both bounded from below, using bounded distorsion (Equation~\eqref{eq:lembd}), by $e^{-C}$ times the value of this quotient for first generation intervals, which is equal to $1/6$. We have $|\alpha-\alpha'| \geq |I|+|I'|$ but for convenience, let us forget $|I'|$:
\[\alpha-\alpha'\geq e^{-2C}\beta_{n_1-1}^2/12\]

Otherwise, we are in case (A): $n_1=n_0+2$. There is a fundamental interval of generation $n_0+2$ inside $I_{n_0+1}(\alpha)$ that separates $\alpha$ from $\alpha'$. 
Hence\footnote{This estimate is sufficient but is far from optimal. A bound of order $\beta_{n_1-2}^2(\alpha_{n_1-1}+\alpha_{n_1-1})$ can be obtained.}
\begin{equation*}
|\alpha-\alpha'| \geq e^{-2C}\beta_{n_1-1}^2/2.\qedhere
\end{equation*}
\end{proof}
In both cases (A) and (B), to bound the part of the sum with $k\geq n_1$, let us use $\beta_{k-1}\,|C(f_k)-C(g_k)| \leq 2K_3\beta_{k-1}$ and
$\sum_{k=n_1}^{+\infty} \beta_{k-1} \leq 2 \beta_{n_1-1}$, whence
\[ \sum_{k=n_1}^{+\infty} \beta_{k-1}\,|C(f_k)-C(g_k)| \leq 4 K_3 \beta_{n_1-1} \leq 4 e^C K_3 \sqrt{12}  |\alpha-\alpha'|^{1/2}.
\]

The only term of the sum that we have not bounded is the one with $k=n_0+1$ in case (A). 
Let $\wt f_{n_0+1} = \cal{RS} f_{n_0}$ (compare with $f_{n_0+1} = \cal{SRS}f_{n_0}$) and $\wt g_{n_0+1} = \cal{RS} g_{n_0}$.
Since $C(f)=C(\cal S f)$ for all $f$, this implies that 
\[ |C(f_{n_0+1})-C(g_{n_0+1})| = |C(\wt f_{n_0+1})-C(\wt g_{n_0+1})|
.\]
Similarly, let $\wt \alpha_{n_0+1} = 1/\alpha_{n_0} - a_{n_0+1} \in [-1/2,1/2]$, and $\wt \alpha'_{n_0+1} = 1/\alpha'_{n_0} - a'_{n_0+1} \in [-1/2,1/2]$.
This has been chosen so that: $\wt f_{n_0+1}'(0)=e^{2\pi i \wt \alpha_{n_0+1}}$ and $\wt g_{n_0+1}'(0)=e^{2\pi i \wt \alpha'_{n_0+1}}$.
Let us now introduce the intermediary function $h=e^{2\pi i \wt \alpha'_{n_0+1}} (\wt f_{n_0+1})_0$. Then
\[|C(f_{n_0+1})-C(g_{n_0+1})| \leq |C(\wt f_{n_0+1})-C(h)|+|C(h)-C(\wt g_{n_0+1})|
.\]
Then we will use the Lipschitz property of $C$ with respect to $d_{\log}$.
For the term $|C(h)-C(\wt g_{n_0+1})|$, we obtain the upper bound $K_1\big(d_T((h)_0,(\wt g_{n_0+1})_0)+0\big)$, which equals 
$K_1\times0+K_1d_T\big((\wt f_{n_0+1})_0,(\wt g_{n_0+1})_0\big)$. 
For the term $|C(\wt f_{n_0+1})-C(h)|$ we obtain the upper bound $K_1 d_{\log}(\wt \alpha_{n_0+1},\wt\alpha'_{n_0+1})$, which equals $K_1d_{\log}(-\alpha_{n_0+1},\alpha'_{n_0+1})$. 
Hence
\[|C(f_{n_0+1})-C(g_{n_0+1})| \leq K_1 \Big( d_T\big((\wt f_{n_0+1})_0, (\wt g_{n_0+1})_0\big)+ d_{\log}(-\alpha_{n_0+1},\alpha'_{n_0+1})\Big)
.\]
Now we can use
\[d_T\big((\wt f_{n_0+1})_0,(\wt g_{n_0+1})_0\big) \leq \lambda d_T\big((f_{n_0})_0,(g_{n_0})_0\big) + K_2 |\alpha_{n_0}-\alpha'_{n_0}|,\]
thus
\[d_T\big((\wt f_{n_0+1})_0,(\wt g_{n_0+1})_0\big) \leq \lambda^{n_0+1} d_T((f)_0,(g)_0) + K_2 \sum_{m=0}^{n_0} \lambda^{n_0-m} d_m \]
which can be incorporated to the computation made earlier of the sum from $k=0$ to $n_0$. 
So, in case (A):
\begin{multline*}
 \sum_{k=0}^{n_0+1}\beta_{k-1}\,|C(f_k)-C(g_k)|  \leq 2 K_1\,d_T(f_0,g_0) +K_1K_2\sum_{j=0}^{n_0+1}\beta_{j-1}d_j \\ 
 + K_1\sum_{j=0}^{n_0}\beta_{j-1}d'_j   + K_1\beta_{n_0} d''_{n_0+1}
\end{multline*}
with $d''_{n_0+1} = d_{\log}(-\alpha_{n_0+1},\alpha'_{n_0+1})$. Let $x$ be the touching point between $I_{n_0+1}(\alpha)$ and $I_{n_0+1}(\alpha')$: $|\alpha-\alpha'|=|\alpha-x|+|\alpha'-x|$.
By bounded distorsion (Equation~\eqref{eq:lembd}), $|\alpha-x| \geq e^{-C} \beta_{n_0}^2 \alpha_{n_0+1}$.
Similarly (see the discussion following Equation~\eqref{eq:lembd}) $|\alpha'-x| \geq e^{-2C} \beta_{n_0}^2 \alpha'_{n_0+1}$.
So $|\alpha-\alpha'| \geq \beta_{n_0}^2e^{-2C}(\alpha_{n_0+1} +\alpha'_{n_0+1})$. 
Now, using $d_{\log}(-\alpha_{n_0+1},\alpha'_{n_0+1}) \leq M_0(\alpha_{n_0+1} +\alpha'_{n_0+1})^{1/2}$ 
we get
\[\beta_{n_0} d''_{n_0+1} \leq M_0 e^C |\alpha-\alpha'|^{1/2}.\]

Taking both cases (A) and (B) into account, we get:
\begin{multline*}
\sum_{k=0}^{+\infty}\beta_{k-1}\,|C(f_k)-C(g_k)| \leq M_1 |\alpha-\alpha'|^{1/2} + 2 K_1\,d_T(f_0,g_0) \\ 
+ K_1 K_2 \sum_{j=0}^{+\infty}\beta_{j-1}d_j + K_1 \sum_{j=0}^{+\infty}\beta_{j-1}d'_j
\end{multline*}
where $M_1 = 4e^C K_3\sqrt{12}+K_1 M_0 e^C$.

\subsection{Second term}\label{subsec:secondterm}

Let us now deal with the following term in Equation~\eqref{eq:upup}
\[ K_3 \sum_{k=0}^{+\infty}|\beta_{k-1}(f)-\beta_{k-1}(g)|.
\]

\begin{rem} For the modified continued fraction algorithm, each $\beta_k$ is a continuous function of $\alpha$, whereas it is not the case for the classical continued fractions. 
Note also that each $\beta_k$ is in fact $1/2$-H\"older continuous (with respect to $\alpha$) with constant $B_k$, and that the sequence $B_k$ is bounded but unfortunaltely does not tend to $0$, thus $\sum B_k$ is not convergent.
\end{rem}

From
{
\renewcommand{\a}{\alpha}
\begin{multline*}
\beta_k(f)-\beta_k(g) = \a_0\cdots\a_{k} -\a'_0\cdots\a'_{k} 
\\ = \a_0 \cdots a_{k-2}\a_{k-1} (\a_k -  \a'_k) + \a_0 \cdots \a_{k-2} (\a_{k-1} -  \a'_{k-1}) \a'_k 
\\+ \a_0 \cdots (\a_{k-2} -  \a'_{k-2})  \a'_{k-1} \a'_k
 + \ldots + (\a_0-\a'_0) \a'_1 \cdots \a'_k
\end{multline*}
}
we get:
\begin{align*}
\sum_{k=0}^{+\infty}|\beta_{k-1}(f)-\beta_{k-1}(g)| &= \sum_{k=0}^{+\infty}|\beta_{k}(f)-\beta_{k}(g)| \\
& \leq  \sum_{k=0}^{+\infty} \sum_{j=0}^{k} \alpha_0 \cdots \alpha_{j-1} \Big|\alpha_j-\alpha'_j\Big|\alpha'_{j+1} \cdots \alpha'_{k} \\
& = \sum_{j=0}^{+\infty} \alpha_0\cdots\alpha_{j-1}|\alpha_j-\alpha'_j|\sum_{k=j}^{+\infty} \alpha'_{j+1}\cdots\alpha'_{k}
\end{align*}
With the convention that for $k=j$, $\alpha'_{j+1}\cdots\alpha'_{k} = 1$. Since each $\alpha'_i$ is $\leq 1/2$, we get
\[\sum_{k=j}^{+\infty} \alpha'_{j+1}\cdots\alpha'_{k} \leq 2
.\]
Hence, using the following notation:
\[ d_n = |\alpha_n-\alpha_n'|,
\]
we obtain the following bound on the second term:
\begin{equation*} 
K_3\sum_{k=0}^{+\infty}|\beta_{k-1}(f)-\beta_{k-1}(g)| \leq 2K_3 \sum_{j=0}^{+\infty} \beta_{j-1} d_j.
\end{equation*}

\subsection{Arithmetics}\label{subsec:arith}

The work presented below is not original; it is basically a simplified version of part of~\cite{MMY97}, reformulated in different notations.

\subsubsection{Estimates on the size of fundamental intervals}\label{sss:esfi}

As a trick to simplify the presentation of the next section, we introduce \emph{extended fundamental intervals} of generation $n$. These are the sets on which $H_{n+1}$ is continuous. They can also be characterized as the union of two consecutive fundamental intervals of generation $n$ and their common boundary point, but only for those pairs whose symbols end with respectively $(a_n,+)$ and $(a_n+1,-)$ (the rest of their symbol must be identical). Like the fundamental intervals, their collection also form a nested sequence of partitions of the irrationals. Like above, for $\alpha$ irrational, we will denote
\[\wt I_n(\alpha)\]
the $n$-th generation extended fundamental interval containing $\alpha$.
Another corollary of Equation~\eqref{eq:lembd} is that for two consecutive fundamental intervals $I,I'$ at a given level, we have:
\[ e^{-2C} \leq \frac{|I'|}{|I|} \leq e^{2C}
.\]
In fact we will use the following slightly better\footnote{The proof could be as well carried out without this improvement.} estimate, using the fact that the left derivative and the right derivative of $H_n$ at the point where $I$ and $I'$ meet, have equal absolute values:
\[\forall \alpha\in I,\ e^{-2C} \leq \frac{\big|I'\big|}{\beta_{n-1}^2/2}\leq e^{2C}\]
whence
\begin{equation}\label{eq:iiprime}
\forall \alpha\in I,\ e^{-C}+e^{-2C} \leq \frac{|\wt I_n(\alpha)|}{\beta_{n-1}^2/2}\leq e^{C}+e^{2C}.
\end{equation}

More generally, if one has $k$ consecutive adjacent fundamental intervals of generation $n$, and $\alpha$ belongs to one of them, then
the sum $\Sigma$ of their lengths satisfies
\[ e^{-C}+e^{-2C}+\cdots+e^{-kC}\leq \frac{ \Sigma}{\beta_{n-1}^2/2}\leq e^C+e^{2C}+\cdots+e^{kC}
.\]

Let us now try and give various estimates of $|\alpha-\alpha'|$ in terms of their modified continued fraction expansion.
Assume that $\alpha,\alpha'\in[-1/2,1/2]$, that they are both irrationals and that $\alpha\neq \alpha'$. Let $n\geq 0$ be the first integer such that $\alpha$ and $\alpha'$ do not belong to the same generation $n$ extended fundamental interval.
If $n>0$, from $\wt I_{n_2-1}(\alpha)= \wt I_{n-1}(\alpha')$ we get the upper bound $|\alpha-\alpha'| \leq |\wt I_{n-1}(\alpha)| $ thus
\[|\alpha-\alpha'| \leq \frac{e^{C}+e^{2C}}{2} \beta_{n-2}^2.\]
Depending on the situations, this upper bound can be far from sharp. 
For a lower bound, let us consider the two cases: 
\begin{itemize}
\item If $\wt I_n(\alpha)$ and $\wt I_n(\alpha')$ are not adjacent, then
\[ \frac{e^{-2C}+e^{-3C}}{2} \beta_{n-1}^2 \leq |\alpha-\alpha'| 
.\]
\item If they are adjacent, then
\[  e^{-4C} (\alpha_n+\alpha'_n)\beta_{n-1}^2 \leq |\alpha-\alpha'| 
.\]
\end{itemize}
Those estimates can also be far from optimal.
\begin{proof}
In the first case, $\alpha$ and $\alpha'$ are separated by at least one extended interval adjacent to the one containing $\alpha$.

In the second case, $[\alpha,\alpha']$ intersects between $2$ and $4$ fundamental intervals. On each, $H_n$ is a bijection. Let $u$ be the common point in the boundaries of $\wt I_n(\alpha)$ and $\wt I_n(\alpha')$.
Then $[\alpha,\alpha']\setminus \{u\}$ splits into two components: $[\alpha,u)\subset \wt I_n(\alpha)$, and $(u,\alpha']\subset \wt I_n(\alpha')$. 
Consider $[\alpha,u)$. Then it is either contained in only one fundamental interval, in which case it is equal to the connected component $J$ containing $\alpha$ of $H_n^{-1}((0,\alpha_{n}))$, and $|J| \geq e^{-C}\beta_{n-1}^{-2} \alpha_n$.
Or it is not and then it contains the component $J$ of $H_n^{-1}((0,1/2))$ that is contained in $\wt I_n(\alpha)$, and $|J|\geq e^{-2C} \beta_{n-1}^{-2} /2$. In both cases, $|J| \geq e^{-2C}\beta_{n-1}^{-2}\alpha_n$. Similarly, $(u,\alpha']$ contains an interval $J'$ whose length is $\geq e^{-4C}\beta_{n-1}^{-2}\alpha'_n$. 
Now $|\alpha-\alpha'|\geq |J|+|J'| \geq (e^{-2C}\alpha_n+e^{-4C}\alpha'_n)\beta_{n-1}^2 \geq e^{-4C}(\alpha_n+\alpha'_n)\beta_{n-1}^2$, thus we get the lower bound 
$e^{-4C}(\alpha'_n+\alpha_{n})\beta_{n-1}^2
 \leq |\alpha-\alpha'|$.
The upper bound follows from the remark after Equation~\eqref{eq:iiprime}.
\end{proof}
Note that in the two cases, we get the following weaker lower bound:
\begin{equation}\label{eq:wlb} 2e^{-4C}\beta_{n}^2 \leq |\alpha-\alpha'|
.\end{equation}

\subsubsection{Bounding an arithmetically defined sum}\label{sss:bads}

By the work of Sections~\ref{subsec:firstterm} and~\ref{subsec:secondterm}, to prove the main theorem, we are reduced to bound the quantities $\sum_{j=0}^{+\infty} \beta_{j-1} d'_j$ and $\sum_{j=0}^{+\infty} \beta_{j-1} d_j$ by some uniform constants times $|\alpha-\alpha'|^{1/2}$. Since $d'_j>d_j$, it is enough to bound the first sum.

Note that $\forall x,y\in[-1/2,1/2]$, $\forall a\in (0,1)$,
\[d_{\log}(x,y)\leq M_a |x-y|^a.\]
This follows for instance from H\"older's inequality, or more simply from the following computation:
if $x,y$ have the same sign then $\big|\int_x^y dt/|t|\big| \leq \int_0^{|x-y|} dt/t$, if $x,y$ have opposite sign then $\big|\int_x^y dt/|t|\big| \leq 2\int_0^{|x-y|/2} dt/t$. It implies in particular, for all $a\in(0,1)$, 
\[d'_j \leq M_a |\alpha_j-\alpha'_j|^a\]
where $M_a$ depends only on $a$.

The first term of the sum to be bounded is $\beta_{-1}d'_0 = 1\times d_{\log}(\alpha,\alpha')$. Applying the above with $a=1/2$, we see that it is enough to bound the sum over $j\geq1$. Choose your preferred $a\in(1/2,1)$, for instance $a=3/4$. By the discussion above, the following lemma is sufficient to conclude the proof of the main theorem.

\begin{lem} Let $a\in(1/2,1)$. There exists $K>0$, which depends on $a$, such that $\forall \alpha,\alpha' \in [-1/2,1/2]$ that are irrational,
\[\sum_{j=1}^{+\infty} \beta_{j-1} |\alpha_j-\alpha'_j|^a \leq K |\alpha-\alpha'|^{1/2}
.\]
\end{lem}
\begin{rem} Marmi, Moussa and Yoccoz realized that we cannot get a better exponent on the right hand side of the above inequality.
They also noted that taking $a\in (0,1/2)$ gives a function which is $a$-H\"older (Theorem~4.2 in~\cite{MMY97} with $\nu=1$ and $\eta=a$).
\end{rem}
\begin{proof}

Assume $\alpha\neq\alpha'$ for otherwise the  sum is $0$. As in Section~\ref{sss:esfi}, let $n\geq 0$ be the first integer for which $\alpha$ and $\alpha'$ belong to different extended fundamental intervals of generation $n$.
Let $j\geq 1$. To bound $\beta_{j-1} |\alpha_j-\alpha'_j|^a$ we will consider several cases.

\smallskip\noindent\textbf{Case 1, $j<n$:} If $j\leq n-2$, then $\alpha$ and $\alpha'$ lie in the same $j$-th generation fundamental interval.
By bounded distorsion of $H_j:\alpha\mapsto \alpha_j$ on fundamental intervals, $|\alpha_j-\alpha'_j| \leq e^{C} H_j'(\alpha) |\alpha-\alpha'| = e^{C} \beta_{j-1}^{-2} |\alpha-\alpha'|$.
For $j=n-1$, $\alpha$ and $\alpha'$ lie in the same $j$-th generation extended fundamental interval, whence $|\alpha_j-\alpha'_j| \leq e^{2C} \beta_{j-1}^{-2} |\alpha-\alpha'|$. In both cases, i.e.\ for all $j\leq n-1$, we have
\[|\alpha_j-\alpha'_j| \leq e^{2C} \beta_{j-1}^{-2} |\alpha-\alpha'|.\]
Recall that $|\alpha-\alpha'| \leq (e^C+e^{2C})\beta_{n-2}^2/2$, thus
\bEA
& \ds \frac{\beta_{j-1}|\alpha_j-\alpha'_j|^a}{|\alpha-\alpha'|^{1/2}}
\leq e^{2Ca}\left(\frac{|\alpha-\alpha'|}{\beta_{j-1}^{2}}\right)^{a-1/2}
\leq  K_4\left(\frac{\beta_{n-2}}{\beta_{j-1}}\right)^{a-1/2} \ =
& \\  & \ds =\ K_4 (\alpha_{j} \cdots \alpha_{n-2})^{a-1/2}\ \leq\ K_4\left(\frac{1}{2^{(n-1)-j}}\right)^{a-1/2}
. & \eEA
With $K_4 = e^{2Ca} \left(\frac{e^C+e^{2C}}{2}\right)^{a-1/2}$ and the convention that $\alpha_{j} \cdots \alpha_{n-2}=1$ if $j=n-1$.
 Whence
\[ \beta_{j-1}|\alpha_j-\alpha'_j|^a \leq \frac{K_4}{u^{n-1-j}} |\alpha-\alpha'|^{1/2}
\]
with $u=2^{a-1/2}>1$.

\smallskip\noindent\textbf{Case 2, $j > n$:} Then $|\alpha_j-\alpha_j'| \leq 1/2 < 1$ and 
\[\beta_{j-1} = \alpha_0 \cdots \alpha_{j-1} = \beta_{n} \alpha_{n+1} \cdots \alpha_{j-1} \leq \beta_{n} /2^{j-n-1} \leq \frac{e^{2C}}{\sqrt{2}} |\alpha-\alpha'|^{1/2} /2^{j-n-1}\] 
by Equation~\eqref{eq:wlb}. Whence
\[\beta_{j-1}|\alpha_j-\alpha'_j|^a \leq \frac{K_5}{2^{j-(n+1)}} |\alpha-\alpha'|^{1/2}\]
with $K_5 = e^{2C}\sqrt{2}$.

\smallskip\noindent\textbf{Case 3, $j=n$:} We consider two sub-cases:
\begin{itemize}
\item $\wt I_n(\alpha)$ and $\wt I_n(\alpha')$ are not adjacent. Then starting from 
$|\alpha_n-\alpha'_n| \leq 1/2<1$ and $|\alpha-\alpha'|\geq (e^{-2C}+e^{-3C})\beta_{n-1}^2/2$, 
we get 
\begin{gather*}
\beta_{n-1}|\alpha_n-\alpha'_n|^a / |\alpha-\alpha' |^{1/2} \leq  \sqrt{2/(e^{-2C}+e^{-3C})},\\
\beta_{n-1}|\alpha_n-\alpha'_n| \leq K_6 |\alpha-\alpha'|^{1/2}
\end{gather*}
with $K_6 = \sqrt{2/(e^{-2C}+e^{-3C})}$.
\item $\wt I_n(\alpha)$ and $\wt I_n(\alpha')$ are adjacent. 
Then starting from 
\[|\alpha_n-\alpha'_n| \leq 1/2 <1 \tand e^{-4C}(\alpha'_n+\alpha_{n})\beta_{n-1}^2 \leq |\alpha-\alpha'|\]
we get
\[
\ds \beta_{n-1}|\alpha_n-\alpha'_n|^a / |\alpha-\alpha'|^{1/2} \leq e^{2C}\frac{|\alpha_n-\alpha'_n|^a}
{\sqrt{\alpha_n+\alpha'_n}}.\] 
Now $|\alpha_n-\alpha'_n|\leq \max(\alpha_n,\alpha'_n)$ and $\alpha_n+\alpha'_n \geq \max(\alpha_n,\alpha'_n)$ 
thus
\[\ds \beta_{n-1}|\alpha_n-\alpha'_n|^a / |\alpha-\alpha'|^{1/2} \leq e^{2C} \frac{\max(\alpha_n,\alpha'_n)^a}{\sqrt{\max(\alpha_n,\alpha'_n)}} = e^{2C} \left(\max(\alpha_n,\alpha'_n)\right)^{a-1/2} \leq e^{2C}.
\]
Hence, 
\[ \beta_{n-1}|\alpha_n-\alpha'_n|^a \leq K'_6 |\alpha-\alpha'|^{1/2}
\]
with $K'_6 = e^{2C}$.
\end{itemize}

\medskip

\noindent By the first case:
\[\sum_{j=1}^{n-1}\beta_{j-1}|\alpha_j-\alpha'_j|^a \leq K_4\, |\alpha-\alpha'|^{1/2} \sum_{s=0}^{+\infty} \frac{1}{u^s}\]
with $u=2^{a-1/2}>1$.
By the second case:
\[\sum_{j=n+1}^{+\infty}\beta_{j-1}|\alpha_j-\alpha'_j|^a \leq K_5\, |\alpha-\alpha'|^{1/2} \sum_{s=0}^{+\infty} \frac{1}{2^s}.\]
Using the three cases altogether we thus get:
\[ \sum_{j=1}^{+\infty}\beta_{j-1}|\alpha_j-\alpha'_j|^a \leq K_7\,|\alpha-\alpha'|^{1/2}
\]
with $K_7 = \frac{K_4}{1-u}+2K_5+\on{max}(K_6,K'_6)$.
\end{proof}

This proves the lemma, which was the last step to get the theorem.